\documentclass{amsart}

\usepackage{mathrsfs}
\usepackage{mathrsfs}
\usepackage{amsfonts}
\usepackage{amsfonts}
\usepackage{mathrsfs}
\usepackage{amsfonts}
\usepackage{amsfonts}
\usepackage{amsfonts}
\usepackage{mathrsfs}
\usepackage{amsfonts}
\usepackage{amssymb,amsmath}
\usepackage{amsthm}
\usepackage{amsfonts}
\usepackage{dsfont}

\newtheorem{theorem}{Theorem}[section]
\newtheorem{lemma}[theorem]{Lemma}

\theoremstyle{definition}
\newtheorem{definition}[theorem]{Definition}

\newtheorem{proposition}[theorem]{Proposition}

\theoremstyle{remark}
\newtheorem{remark}[theorem]{Remark}

\numberwithin{equation}{section}

\newcommand{\abs}[1]{\lvert#1\rvert}

\newcommand{\norm}[1]{\lVert#1\rVert}

\begin{document}

\title[Compressible MHD system]
{Well-posedness for compressible MHD system with highly oscillating initial data}
\author[J. Jia]{Junxiong Jia}
\address{Department of Mathematics,
Xi'an Jiaotong University,
 Xi'an
710049, China;Beijing Center for Mathematics and Information Interdisciplinary Sciences (BCMIIS);}
\email{jjx323@mail.xjtu.edu.cn}

\author[J. Peng]{Jigen Peng}
\address{Department of Mathematics,
Xi'an Jiaotong University,
 Xi'an
710049, China;Beijing Center for Mathematics and Information Interdisciplinary Sciences (BCMIIS);}
\email{jgpengg@mail.xjtu.edu.cn}

\author[J. Gao]{Jinghuai Gao}
\address{School of Electronic and Information Engineering,
Xi'an Jiaotong University,
 Xi'an
710049, China;}
\email{jhgao@mail.xjtu.edu.cn}







\begin{abstract}
In this paper, we transform compressible MHD system written in Euler coordinate to Lagrange coordinate in critical Besov space.
Then we construct unique local solutions for compressible MHD system.
Our results improve the range of Lebesgue exponent in Besov space from $[2, N)$ to $[2, 2N)$ with $N$ stands for dimension.
In addition, we give a lower bound for the maximal existence time which is important for our construction of global solutions.
Based on the local solution, we obtain a unique global solution with high oscillating initial velocity and density by
using effective viscous flux and Hoff's energy methods to explore the structure of compressible MHD system.
\end{abstract}

\maketitle

\section{Introduction}

Magnetohydrodynamics (MHD) is concerned with the interaction between
fluid flow and magnetic field. The governing equations of isentropic compressible MHD system has the form
\begin{align}\label{origianl MHD}
\begin{split}
\begin{cases}
\partial_{t}\rho + \mathrm{div} (\rho u) = 0,     \\
\partial_{t}(\rho u) + \mathrm{div} (\rho u \otimes u) + \nabla P(\rho)     \\
\quad\quad\quad\quad\quad\quad\quad\quad\quad
 = B\cdot\nabla B - \frac{1}{2} \nabla (|B|^{2}) + \mu \Delta u + (\lambda + \mu) \nabla\mathrm{div} u,  \\
\partial_{t}B + (\mathrm{div} u)B + u\cdot\nabla B - B\cdot\nabla u = \nu \Delta B,   \\
\mathrm{div} B = 0,   \\
(\rho, u, B)|_{t=0} = (\rho_{0}, u_{0}, B_{0}), \quad \mathrm{div} B_{0}=0,
\end{cases}
\end{split}
\end{align}
where $\rho = \rho(t,x)$ denotes the density, $x \in \mathbb{R}^{N}$, $t > 0$, $u \in  \mathbb{R}^{N}$ is the velocity of the flow,
and $B \in \mathbb{R}^{N}$ stands for the magnetic field. The constants $\mu$ and $\lambda$ are the shear and bulk viscosity coefficients of the flow,
respectively, satisfying $\mu > 0$ and $\lambda + 2\mu > 0$, the constant $\nu > 0$ is the magnetic diffusivity acting as a magnetic diffusion coefficient
of the magnetic field. $P(\rho)$ is the scalar pressure function satisfying $P'(\bar{\rho}) > 0$ where $\bar{\rho}$ is the equilibrium
state of the density.

MHD system has been studied by many authors for its physical importance and mathematical challenges.
X. Hu and D. Wang constructed the global weak solutions with large initial data in a series of papers \cite{huxian1,huxian2,huxian3}.
In addition, they also studied the important low march number limit problem in their papers.
In 2012, A. Suen and D. Hoff \cite{mhd2012hoff} extends small energy weak solution theory about Navier-Stokes system
to MHD case. Then global weak solutions with initial data contain vacuum also be constructed by many authors in \cite{vacuum1,vacuum2}.

Besides the studies about weak solutions mentioned above, there are also a lot of investigations about classical and strong solutions.
In 1984, S. Kawashima studied a very general symmetric hyperbolic-parabolic system
which incorporate MHD system as a special one in his PhD thesis\cite{jia1}.
He constructed the global solution for MHD system when the initial data is small in $H^{4}(\mathbb{R}^{3})$.
After S. Kawashima's investigation, recently, many authors \cite{mhd1,mhd2,mhd3,mhd4} studied the global well-posedness and optimal time decay rate
for the MHD system with the magnetic equilibrium state to be zero.
Their studies get the optimal time decay rate for the derivatives of the solution which generalize the optimal
time decay rate results for Navier-Stokes equations to the more complex MHD equations.

R. Danchin did a seminal work in \cite{danchin NS} in 2000. By scaling analysis R. Danchin gave the definitions about critical space for compressible
Navier-Stokes equations, then under this new framework he constructed a global solution when the initial data to be small in the critical Besov space.
After this important work, there are a lot of further studies \cite{NavierStokesLp,NavierStokesLpDanchin,danchin full ns}.
Especially, C. Wang, W. Wang and Z. Zhang \cite{NS2013Zhifei} construct a global solution for isentropic Navier-Stokes system
which allow highly oscillating initial density and velocity in 2014.
In \cite{localBesov1,localBesov2}, D. Bian, B. Yuan and B. Guo, firstly, constructed the local solution for compressible MHD equations in critical Besov space framework.
In \cite{mhd2011Chengchun}, C. Hao analyzed the linearized hyperbolic-parabolic system related to MHD equations carefully,
then he constructed the global solution in critical Besov space framework.
Recently, there are some important works for incompressible MHD equations \cite{inmhd1,inmhd2,inmhd3}.
In these excellent papers, F. Lin, L. Xu and P. Zhang constructed global solutions for incompressible MHD equations without magnetic diffusion by
cutely using the dissipative structure of the system and the anisotropic Besov space technique.
Then, X. Hu in paper \cite{huxian4} constructed the global solutions for compressible MHD equations without magnetic diffusion.
Due to the structure is very different for incompressible and compressible MHD equations, X. Hu's work is very different
to F. Lin, L. Xu and P. Zhang's work and X. Hu used lots of structure information about the compressible MHD equations.

In this manuscript, we plan to construct global solution for MHD system with highly oscillating initial velocity and density.
For isentropic Navier-Stokes system, global solutions with highly oscillating initial velocity can be obtained
by a careful analysis about Green's matrix for the linear system \cite{NavierStokesLp,NavierStokesLpDanchin}.
In paper \cite{jiamei}, the authors studied the global well-posedness for compressible viscoelastic fluids in $L^{p}$
based critical Besov space by decomposing the whole system into three small systems with each one similar to Navier-Stokes case.
However, as mentioned by A. Suen and D. Hoff in their paper \cite{mhd2012hoff},
the presence of magnetic field effects results in a different and more intricate coupling of rates of partial
regularization in the initial layer, so we need more careful analysis.
we firstly transform compressible MHD system written in Euler coordinate into Lagrange coordinate in critical Besov space
which is not appeared before. Then based on Lagrange coordinate form, we construct the local solution.
Comparing to the known local results \cite{localBesov1,localBesov2}, we improve the uniqueness range for $p$ from
$[2, N)$ to $[2, 2N)$ where $N$ stands for dimension and we also give a lower bound for the maximal existence time
which is important for our global result.
Based on the local well-posedness, we obtain a global solution by using effective viscous flux and Hoff's energy estimates
to explore the structure of the system.
For more explicit statements, please see Theorem \ref{local well-posedness} and Theorem \ref{mainexistence} in next section.

The other parts of this paper are organized as follows.
In section 2, we give a brief introduction about Besov space and the two main results of this paper.
In Section 3, we prove the local well-posedness in
critical Besov space framework by using Lagrange coordinate methods and prove that the local solution can
propagate the smoothness of the initial data. In Section 4, we use Hoff's energy methods to get a uniform estimates
about the solution. In Section 5, we prove a blow up criterion and then combine the results from Section 3 to complete
the proof of Theorem \ref{mainexistence}.
At last, for the reader's convenience, we collect a lot of useful Lemmas in the Appendix.


\section{Main Results and Some Preliminaries}

In this section, we will introduce some notations and give the three main results of this paper.
Firstly, let me give some basic knowledge about Besov space, which can be found in \cite{FourierPDE}.
The homogeneous Littlewood-Paley decomposition relies upon a dyadic partition of unity. We can use for instance any $\phi \in C^{\infty}(\mathbb{R}^{N})$,
supported in $\mathcal{C} := \{\xi \in \mathbb{R}^{N}, 3/4 \leq \abs{\xi} \leq 8/3 \}$ such that
\begin{align*}
\sum_{q \in \mathbb{Z}} \phi(2^{-q} \xi) = 1 \quad \mathrm{if} \quad \xi \neq 0.
\end{align*}
Denote $h = \mathcal{F}^{-1} \phi$, we then define the dyadic blocks as follows
\begin{align*}
\Delta_{q} u := \phi(2^{-q}D)u = 2^{qN} \int_{\mathbb{R}^{N}} h(2^{q}y)u(x-y) \, dy, \quad \mathrm{and} \quad S_{q}u = \sum_{k\leq q-1} \Delta_{k}u.
\end{align*}
The formal decomposition
\begin{align*}
u = \sum_{q\in \mathbb{Z}} \Delta_{q} u
\end{align*}
is called homogeneous Littlewood-Paley decomposition.
The above dyadic decomposition has nice properties of quasi-orthogonality: with our choice of $\phi$, we have
\begin{align*}
& \Delta_{k}\Delta_{q}u = 0 \quad \mathrm{if} \quad \abs{k-q} \geq 2, \\
& \Delta_{k}(S_{q-1}\Delta_{q}u) = 0 \quad \mathrm{if} \quad \abs{k-q} \geq 5.
\end{align*}

Let us now introduce the homogeneous Besov space.
\begin{definition}
We denote by $\mathcal{S}_{h}'$ the space of tempered distributions $u$ such that
\begin{align*}
\lim_{j \rightarrow -\infty} S_{j}u = 0 \quad \mathrm{in} \quad \mathcal{S}'.
\end{align*}
\end{definition}

\begin{definition}
Let $s$ be a real number and $(p,r)$ be in $[1, \infty]^{2}$. The homogeneous Besov space $\dot{B}_{p,r}^{s}$ consists of distributions
$u$ in $\mathcal{S}_{h}'$ such that
\begin{align*}
\|u\|_{\dot{B}_{p,r}^{s}} := \left( \sum_{j \in \mathbb{Z}} 2^{rjs} \|\Delta_{j}u\|_{L^{p}}^{r} \right)^{1/r} < +\infty.
\end{align*}
\end{definition}
From now on, the notation $\dot{B}_{p}^{s}$ will stand for $\dot{B}_{p,1}^{s}$ and the notation $\dot{B}^{s}$ will stand for $\dot{B}_{2,1}^{s}$.

The study of non stationary PDE's usually requires spaces of type $L_{T}^{r}(X) := L^{r}(0,T; X)$ for appropriate Banach spaces $X$.
In our case, we expect $X$ to be a Besov spaces, so that it is natural to localize the equations through Littlewood-Paley decomposition.
We then get estimates for each dyadic block and perform integration in time. However, in doing so, we obtain bounds in spaces which are not
of type $L^{r}(0,T; \dot{B}_{p}^{s})$. This approach was initiated in \cite{chemin} and naturally leads to the following definitions:
\begin{definition}
Let $(r,p) \in [1, +\infty]^{2}$, $T \in (0, +\infty]$ and $s\in \mathbb{R}$. We set
\begin{align*}
\|u\|_{\tilde{L}_{T}^{r}(\dot{B}_{p}^{s})} := \sum_{q \in \mathbb{Z}} 2^{qs} \left( \int_{0}^{T} \|\Delta_{q}u(t)\|_{L^{p}}^{r} \, dt \right)^{1/r}
\end{align*}
and
\begin{align*}
\tilde{L}_{T}^{r}(\dot{B}_{p}^{s}) := \left\{ u \in L_{T}^{r}(\dot{B}_{p}^{s}), \|u\|_{\tilde{L}_{T}^{r}(\dot{B}_{p}^{s})}< +\infty \right\}.
\end{align*}
\end{definition}
Owing to Minkowski inequality, we have $\tilde{L}_{T}^{r}(\dot{B}_{p}^{s}) \hookrightarrow L_{T}^{r}(\dot{B}_{p}^{s})$. That embedding is strict in general
if $r>1$. We will denote by $\tilde{C}_{T}(\dot{B}_{p}^{s})$ the set of function $u$ belonging to
$\tilde{L}_{T}^{\infty}(\dot{B}_{p}^{s})\cap C([0,T]; \dot{B}_{p}^{s})$.

We will often use the following interpolation inequality:
\begin{align*}
\|u\|_{\tilde{L}_{T}^{r}(\dot{B}_{p}^{s})} \leq \|u\|_{\tilde{L}_{T}^{r_{1}}(\dot{B}_{p}^{s_{1}})}^{\theta} \|u\|_{\tilde{L}_{T}^{r_{2}}(\dot{B}_{p}^{s_{2}})}^{1-\theta},
\end{align*}
with
\begin{align*}
\frac{1}{r} = \frac{\theta}{r_{1}} + \frac{1-\theta}{r_{2}} \quad \mathrm{and} \quad s = \theta s_{1} + (1-\theta) s_{2},
\end{align*}
and the following embeddings
\begin{align*}
\tilde{L}_{T}^{r}(\dot{B}_{p}^{n/p}) \hookrightarrow L_{T}^{r}(\mathcal{C}_{0}) \quad \mathrm{and} \quad \tilde{C}_{T}(\dot{B}_{p}^{n/p})
\hookrightarrow C([0,T]\times \mathbb{R}^{N}).
\end{align*}

Another important space is the bybrid Besov space, we give the definitions and collect some properties.
\begin{definition}
let $s,t \in \mathbb{R}$. We set
\begin{align*}
\|u\|_{B_{q,p}^{s,t}} := \sum_{k \leq R_{0}} 2^{ks} \|\Delta_{k}u\|_{L^{q}} + \sum_{k > R_{0}}2^{kt} \|\Delta_{k}u\|_{L^{p}}.
\end{align*}
and
\begin{align*}
B_{q,p}^{s,t}(\mathbb{R}^{N}) := \left\{ u \in \mathcal{S}_{h}'(\mathbb{R}^{N}), \|u\|_{B_{q,p}^{s,t}} < +\infty \right\},
\end{align*}
where $R_{0}$ is a fixed constant.
\end{definition}

\begin{lemma}

1) We have $B_{2, 2}^{s,s} = \dot{B}^{s}$.

2) If $s\leq t$ then $B_{p,p}^{s,t} = B_{p,p}^{s} \cap B_{p,p}^{t}$. Otherwise, $B_{p,p}^{s,t} = B_{p,p}^{s} + B_{p,p}^{t}$.

3) The space $B_{p,p}^{0,s}$ coincide with the usual inhomogeneous Besov space.

4) If $s_{1} \leq s_{2}$ and $t_{1} \geq t_{2}$ then $B_{p,p}^{s_{1},t_{1}} \hookrightarrow B_{p,p}^{s_{2},t_{2}}$.

5) Interpolation: For $s_{1}, s_{2}, \sigma_{1}, \sigma_{2} \in \mathbb{R}$ and $\theta \in [0,1]$, we have
\begin{align*}
\|f\|_{B_{2,p}^{\theta s_{1} + (1-\theta)s_{2}, \theta \sigma_{1} + (1-\theta)\sigma_{2}}} \leq \|f\|_{B_{2,p}^{s_{1},\sigma_{1}}}^{\theta}
\|f\|_{B_{2,p}^{s_{2},\sigma_{2}}}^{1-\theta}.
\end{align*}
\end{lemma}
From now on, the notation $B_{p}^{s,t}$ will stand for $B_{p,p}^{s,t}$ and the notation $B^{s,t}$ will stand for $B_{2,2}^{s,t}$.

Throughout the paper, we shall use some paradifferential calculus. It is a nice way to define a generalized product between
distributions which is continuous in functional spaces where the usual product does not make sense. The paraproduct between $u$ and $v$
is defined by
\begin{align*}
T_{u}v := \sum_{q \in \mathbb{Z}}S_{q-1}u \Delta_{q}v.
\end{align*}
We thus have the following formal decomposition:
\begin{align*}
uv = T_{u}v + T_{v}u + R(u,v),
\end{align*}
with
\begin{align*}
R(u,v) := \sum_{q\in \mathbb{Z}} \Delta_{q}u \tilde{\Delta}_{q}v, \quad
\tilde{\Delta}_{q} := \Delta_{q-1} + \Delta_{q} + \Delta_{q+1}.
\end{align*}
We will sometimes use the notation $T_{u}'v := T_{u}v + R(u,v)$.

For more information about Besov space and hybrid Besov space, we give reference \cite{FourierPDE,Danchin2005,NavierStokesLp,NavierStokesLpDanchin}.
Throughout this paper, $C$ stands for a ``harmless'' constant, and we sometimes use the notation $A \lesssim B$ as an
equivalent of $A \leq C B$. The notation $A \thickapprox B$ means that $A \lesssim B$ and $B \lesssim A$.

With these preparations, we can state our two main results.
\begin{theorem}\label{local well-posedness}
Let $\bar{\rho} > 0$ and $c_{0} > 0$. Assume that the initial data $(\rho_{0}, u_{0}, B_{0})$ satisfies
\begin{align*}
& \rho_{0} - \bar{\rho} \in \dot{B}_{p,1}^{N/p}, \quad c_{0} \leq \rho_{0} \leq c_{0}^{-1},     \\
& u_{0},B_{0} \in \dot{B}_{p,1}^{N/p-1}.
\end{align*}
Then there exists a positive time $T > 0$ such that if $p \in [2,2N)$, the system has a unique solution $(\rho - \bar{\rho}, u, B)$ satisfies
\begin{align*}
& \rho - \bar{\rho} \in \tilde{C}([0,T];\dot{B}_{p,1}^{N/p}), \quad \frac{1}{2}c_{0} \leq \rho \leq 2 c_{0}^{-1},  \\
& u,B \in \tilde{C}([0,T];\dot{B}_{p,1}^{N/p-1}) \cap L^{1}(0,T;\dot{B}_{p,1}^{N/p+1}).
\end{align*}
In addition, we can take the maximal existence time $T$ as follows
\begin{align*}
T = \frac{\bar{c}}{(1+\|a_{0}\|_{\dot{B}_{p,1}^{N/p}})^4},
\end{align*}
where $\bar{c}$ is a small enough positive constant.
\end{theorem}

\begin{remark}\label{local_remark}
The above results are necessary for our following global well-posedness results, however,
Theorem \ref{local well-posedness} has its own interest.
As C. Noboru and R. Danchin \cite{LagrangianCompressible} improved the uniqueness condition for Navier-Stokes system,
we improve the uniqueness condition from $p \in [2, N)$ \cite{localBesov1,localBesov2} to $p \in [2, 2N)$ by using Lagrange coordinate technique.
In our case, we need to get the magnetic field equation in Lagrange coordinate and estimate more complex coupling terms.
In addition, we get a lower bound for the maximal existence time which is necessary for our global existence theorem.
\end{remark}

\begin{theorem}\label{mainexistence}
Suppose the equilibrium state $\bar{\rho} > 0$, $\bar{B} = 0$ and dimension $N = 3$.
Let $c_{0}$ to be a small constant, $$P_{+} = \sup_{c_{0}/4 \leq \rho \leq 4 c_{0}^{-1}} |P^{(k)}(\rho)| \quad \text{and}
\quad P_{-} = \inf_{c_{0}/4 \leq \rho \leq 4 c_{0}^{-1}}|P'(\rho)|.$$
Assume that the initial data $(\rho_{0}, u_{0})$ satisfies
\begin{align*}
& \rho_{0} - \bar{\rho} \in H^{s} \cap \dot{B}_{p,1}^{3/p}, \quad c_{0} \leq \rho_{0} \leq c_{0}^{-1},  \\
& u_{0}, B_{0} \in H^{s-1} \cap \dot{B}_{p,1}^{3/p-1},
\end{align*}
for some $p \in (3,6)$ and $s \geq 3$. There exist five constants $c_{1}$, $c_{2}$, $c_{3}$ such that
\begin{align*}
& \|\rho_{0} - \bar{\rho}\|_{L^{2}} \leq c_{1}, \quad \|u_{0}\|_{\dot{B}_{p,1}^{3/p-1}} \leq \frac{c_{2}}{(1+\|a_{0}\|_{\dot{B}_{p,1}^{3/p}})^{5}}, \quad \|B_{0}\|_{\dot{B}_{p,1}^{3/p-1}} \leq \frac{c_{3}}{(1+\|a_{0}\|_{\dot{B}_{p,1}^{3/p}})^{2}}, \\
& \|u_{0}\|_{\dot{H}^{-\delta}} \leq \frac{c_{2}}{(1+\|a_{0}\|_{\dot{B}_{p,1}^{3/p}})^{5}(1+\|\rho_{0}-\bar{\rho}\|_{H^{2}}^{8}+\|u_{0}\|_{L^{2}}^{4/3}
+ \|B_{0}\|_{L^{2}}^{4/3})},    \\
& \|B_{0}\|_{\dot{H}^{-\delta}} \leq \frac{c_{3}}{(1+\|a_{0}\|_{\dot{B}_{p,1}^{3/p}})^{3}(1+\|\rho_{0}-\bar{\rho}\|_{H^{2}}^{8}+\|u_{0}\|_{L^{2}}^{4/3}
+ \|B_{0}\|_{L^{2}}^{4/3})},
\end{align*}
for some $\delta \in \left(1-\frac{3}{p}, \frac{3}{p} \right)$, then there exist a unique global solution $(\rho, u, B)$ satisfying
\begin{align*}
& \rho \geq \frac{c_{0}}{4}, \quad \rho - \bar{\rho} \in C([0,\infty);H^{s}), \\
& u,B \in C([0,\infty);H^{s-1}) \cap L^{2}(0,T;H^{s}),
\end{align*}
for any $T > 0$.
\end{theorem}

The proof of the above theorem used the advantage of
harmonic analysis and the structure variable called effective viscous flux.
The magnetic field and velocity field coupled with each other in the compressible MHD system
so we need to estimate many coupled terms.
Based on the Theorem \ref{local well-posedness}, we used energy methods in \cite{mhd2012hoff}
which can incorporate the structure conditions of our system.
However, our regularity conditions are better than the conditions in \cite{mhd2012hoff}, so
we need not add additional conditions for viscosity coefficient to extend the local solution to
a global one.

\begin{remark}
Isentropic Navier-Stokes system is well-posedness when the initial velocity field have large oscillation which
is proved in \cite{NavierStokesLp,NavierStokesLpDanchin}.
Until now, there are no such results for compressible MHD systems for the coupling between
velocity and magnetic field make MHD system are more complex than Navier-Stokes equations.
The above theorem allows the initial velocity field
and magnetic field have large oscillation and allows the initial density to have large oscillation on the set of small measure.
Comparing to previous results for MHD system, the theorem above highly reduce the requirements about the small condition on the
initial data.
\end{remark}


\section{Local Existence and Propagation of Regularity}

In this section, we will prove Theorem \ref{local well-posedness} and prove that the obtained solution can propagate the smoothness
of the initial data. Without loss of generality, we can assume $\bar{\rho} = 1$ from this section to the last section.

Let $X_{u}$ be the flow associated to the vector field $u$, that is the solution to
\begin{align}
X_{u}(t,y) = y + \int_{0}^{t} u(\tau, X_{u}(\tau,y)) \, d\tau.
\end{align}
Denoting
\begin{align*}
\tilde{\rho}(t,y) = \rho(t,X_{u}(t,y)), \quad
\tilde{u}(t,y) = u(t,X_{u}(t,y)), \quad
\tilde{B}(t,y) = B(t,X_{u}(t,y)).
\end{align*}
Then we reformulate the MHD system in the Lagrangian coordinate to obtain
\begin{align}\label{MHD in Largangian}
\begin{split}
\begin{cases}
\partial_{t}(J_{u}\tilde{\rho}) = 0,    \\
J_{u}\rho_{0}\partial_{t}\tilde{u} - \mathrm{div}\left( \mathrm{adj}(DX_{u})\left( 2\mu D_{A_{u}}\tilde{u} + \lambda \mathrm{div}_{A_{u}}\tilde{u}\,\mathrm{Id}
 + P(\tilde{\rho})\mathrm{Id} \right) \right) \\
\quad\quad\quad\quad\quad\quad\quad
= \mathrm{div}\left( \mathrm{adj}(DX_{u}) \tilde{B} \tilde{B}^{T} \right) - \frac{1}{2} \mathrm{div}\left( \mathrm{adj}(DX_{u}) |\tilde{B}|^{2} \right),  \\
J_{u} \partial_{t}\tilde{B} - \nu \mathrm{div}\left( \mathrm{adj}(DX_{u}) A_{u}^{T} \nabla\tilde{B} \right) \\
\quad\quad\quad\quad\quad\quad\quad
= \mathrm{div}\left( \mathrm{adj}(DX_{u})\tilde{B} \tilde{u}^{T} \right) - \mathrm{div}\left( \mathrm{adj}(DX_{u})\tilde{u}\tilde{B} \right),     \\
\mathrm{div}(A_{u}\tilde{B}) = 0,
\end{cases}
\end{split}
\end{align}
where $A_{u} = (D_{y}X_{u})^{-1}$.
During the transformation between Euler coordinates and Lagrangian coordinate, we used (\ref{transform}).

Next, we introduce $u_{L}$ the ``free solution'' to
\begin{align}\label{free velocity}
\begin{split}
\begin{cases}
\partial_{t}u_{L} - \mu \Delta u_{L} - (\lambda + \mu) \nabla\mathrm{div} u_{L} = 0,   \\
u_{L}|_{t=0} = u_{0},
\end{cases}
\end{split}
\end{align}
and $B_{L}$ the ``free solution'' to
\begin{align}\label{free megnetic}
\begin{split}
\begin{cases}
\partial_{t}B_{L} - \nu \Delta B_{L} = 0,   \\
\mathrm{div}B_{L} = 0,  \\
B_{L}|_{t=0} = B_{0}.
\end{cases}
\end{split}
\end{align}
Denoting
\begin{align}
L_{\rho_{0}}(\tilde{u}) =
\partial_{t}\tilde{u} - \frac{1}{\rho_{0}} \mathrm{div} \left( 2\mu \mathcal{D}(\tilde{u}) + \lambda \mathrm{div}\tilde{u}\, \mathrm{Id} \right),
\end{align}
and
\begin{align}
L_{\nu}(\tilde{B}) = \partial_{t}\tilde{B} - \nu \Delta\tilde{B}.
\end{align}
We can rewrite the second and third equations in system (\ref{MHD in Largangian}) as follows
\begin{align}\label{mhdeq2}
\begin{split}
\begin{cases}
L_{\rho_{0}}(\tilde{u}) = I_{1}(\tilde{u},\tilde{u}) + \rho_{0}^{-1}\mathrm{div}(I_{2}(\tilde{u},\tilde{u})
+ I_{3}(\tilde{u},\tilde{u}) + I_{4}(\tilde{u}) )   \\
\quad\quad\quad\quad\quad\quad\quad\quad\quad\quad\quad\quad\quad\quad
+ \rho_{0}^{-1}\mathrm{div}(I_{5}(\tilde{u},\tilde{B}) + I_{6}(\tilde{u},\tilde{B})),   \\
L_{\nu}(\tilde{B}) = I_{7}(\tilde{u},\tilde{B}) + \mathrm{div}(I_{8}(\tilde{u},\tilde{B}) + I_{9}(\tilde{u},\tilde{B})  \\
\quad\quad\quad\quad\quad\quad\quad\quad\quad\quad\quad\quad\quad\quad
+ I_{10}(\tilde{u},\tilde{B},\tilde{u}) + I_{11}(\tilde{u},\tilde{B},\tilde{u}) ),
\end{cases}
\end{split}
\end{align}
where
\begin{align*}
& I_{1}(v,w) = (1-J_{v})\partial_{t}w,    \\
& I_{2}(v,w) = (\mathrm{adj}(DX_{v})-Id)(\mu Dw\cdot A_{v} + \mu A_{v}^{T}\nabla w + \lambda A_{v}^{T} : \nabla w), \\
& I_{3}(v,w) = \mu Dw \cdot (A_{v}-Id) + \mu (A_{v}^{T}-Id) \nabla w + \lambda (A_{v}^{T}-Id) : \nabla w,   \\
& I_{4}(v) = \mathrm{adj}(DX_{v})P(J_{v}^{-1}\rho_{0})\, Id, \quad I_{5}(v,w) = \mathrm{adj}(DX_{v})ww^{T},  \\
& I_{6}(v,w) = \frac{1}{2} \mathrm{adj}(DX_{v})|w|^{2}, \quad I_{7}(v,w) = (1-J_{v})\partial_{t}w,  \\
& I_{8}(v,w) = \nu (\mathrm{adj}(DX_{v})-Id)A_{v}^{T} \nabla w,  \quad I_{9}(v,w) = \nu (A_{v}^{T}-Id)\nabla w, \\
& I_{10}(v,w,k) = \mathrm{adj}(\nabla X_{v})wk^{T}, \quad I_{11}(v,w,k) = \mathrm{adj}(DX_{v})wk.
\end{align*}

In order to prove the local well-posedness, we define a map $\mathbf{\Phi} : (v,b) \mapsto (\tilde{u},\tilde{B})$ where
$(\tilde{u}, \tilde{B})$ satisfies (\ref{mhdeq2}) with the $\tilde{u}$, $\tilde{B}$ on the right hand side changed to $v$, $b$ separately.
Denoting $\hat{u} = \tilde{u}-u_{L}$ and $\hat{B} = \tilde{B}-B_{L}$, then we know that $\hat{u}$ and $\hat{B}$ satisfies
\begin{align}\label{mhdeq3}
\begin{split}
\begin{cases}
L_{\rho_{0}}(\hat{u}) = (L_{1} - L_{\rho_{0}})u_{L} + I_{1}(v,v)    \\
\quad\quad\quad\quad\quad
+ \rho_{0}^{-1}\mathrm{div}(I_{2}(v,v) + I_{3}(v,v) + I_{4}(v) + I_{5}(v,b) + I_{6}(v,b)),  \\
L_{\nu}(\hat{B}) = I_{7}(v,b) + \mathrm{div}(I_{8}(v,b) + I_{9}(v,b) + I_{10}(v,b,v) + I_{11}(v,b,v)),  \\
\mathrm{div}(\hat{B}) = \mathrm{div}((Id - A_{v})b),        \\
(\hat{u}, \hat{B})|_{t=0} = (0,0).
\end{cases}
\end{split}
\end{align}
In order to give a clear statement, we need to define
\begin{align*}
& E_{p}(T) = \left\{ v \in \tilde{L}([0,T];\dot{B}_{p,1}^{N/p-1}): \partial_{t}v, \nabla^{2}v \in \tilde{L}^{1}(0,T;\dot{B}_{p,1}^{N/p-1}) \right\},  \\
& \bar{B}_{E_{p}(T)}(u_{L},B_{L};R) = \left\{ v,b: \|v-u_{L}\|_{E_{p}(T)} + \|b-B_{L}\|_{E_{p}(T)} \leq R \right\}.
\end{align*}
Suppose that $T,R \leq 1$, and $T,R$ small enough such that
\begin{align}\label{smallcondition}
\int_{0}^{T} \|\nabla v\|_{\dot{B}_{p,1}^{N/p}}\, dt \leq c < 1.
\end{align}

$\mathbf{Step \,\, 1:Stability\,\,\, of\,\,\, \bar{B}_{E_{p}(T)}(u_{L},B_{L};R)\,\,\, for\,\,\, small\,\,\,
enough\,\,\, R\,\,\, and\,\,\, T .}$
Using Lemma \ref{variable heat eq}, we could get
\begin{align*}
\|\hat{u}\|_{E_{p}(T)} \leq & C e^{C_{\rho_{0},m}T} \Big( \|I_{1}(v,v)\|_{L_{T}^{1}(\dot{B}_{p,1}^{N/p-1})}
+ \|(L_{1}-L_{\rho_{0}})u_{L}\|_{L_{T}^{1}(\dot{B}_{p,1}^{N/p-1})} \Big)   \\
& + (1+\|a_{0}\|_{\dot{B}_{p,1}^{N/p}})\Big( \|I_{2}(v,v)\|_{L_{T}^{1}(\dot{B}_{p,1}^{N/p})} + \|I_{3}(v,v)\|_{L_{T}^{1}(\dot{B}_{p,1}^{N/p})}    \\
& + \|I_{4}(v)\|_{L_{T}^{1}(\dot{B}_{p,1}^{N/p})} + \|I_{5}(v,b)\|_{L_{T}^{1}(\dot{B}_{p,1}^{N/p})} + \|I_{6}(v,b)\|_{L_{T}^{1}(\dot{B}_{p,1}^{N/p})} \Big).
\end{align*}
Using Lemma \ref{heat eq}, we could get
\begin{align*}
\|\hat{B}\|_{E_{p}(T)} \leq & C \Big( \|I_{7}(v,b)\|_{L_{T}^{1}(\dot{B}_{p,1}^{N/p-1})} + \|I_{8}(v,b)\|_{L_{T}^{1}(\dot{B}_{p,1}^{N/p})}
+ \|I_{9}(v,b)\|_{L_{T}^{1}(\dot{B}_{p,1}^{N/p})} \\
& + \|I_{10}(v,b,v)\|_{L_{T}^{1}(\dot{B}_{p,1}^{N/p})} + \|I_{11}(v,b,v)\|_{L_{T}^{1}(\dot{B}_{p,1}^{N/p})} \Big).
\end{align*}
There are many terms can be estimated similar to Navier-Stokes system as follows
\begin{align*}
& \|I_{1}(v,w)\|_{L_{T}^{1}(\dot{B}_{p,1}^{N/p-1})} \leq C \|v\|_{L_{T}^{1}(\dot{B}_{p,1}^{N/p+1})}\|\partial_{t}w\|_{L_{T}^{1}(\dot{B}_{p,1}^{N/p-1})},  \\
& \|(L_{1}-L_{\rho_{0}})u_{L}\|_{L_{T}^{1}(\dot{B}_{p,1}^{N/p-1})}
\leq C \|a_{0}\|_{\dot{B}_{p,1}^{N/p}}(1+\|a_{0}\|_{\dot{B}_{p,1}^{N/p}})\|u_{L}\|_{L_{T}^{1}(\dot{B}_{p,1}^{N/p+1})},   \\
& \|I_{i}(v,w)\|_{L_{T}^{1}(\dot{B}_{p,1}^{N/p})} \leq C (1+\|a_{0}\|_{\dot{B}_{p,1}^{N/p}})\|v\|_{L_{T}^{1}(\dot{B}_{p,1}^{N/p})}\|w\|_{L_{T}^{1}(\dot{B}_{p,1}^{N/p})},
\quad (i=2,3),  \\
& \|I_{4}(v)\|_{L_{T}^{1}(\dot{B}_{p,1}^{N/p})} \leq C T (1+\|a_{0}\|_{\dot{B}_{p,1}^{N/p}}) (1+\|v\|_{L_{T}^{1}(\dot{B}_{p,1}^{N/p})}).
\end{align*}
Here, we only give the estimates for the terms different to Navier-Stokes systems.
\begin{align*}
\|I_{5}(v,w)\|_{L_{T}^{1}(\dot{B}_{p,1}^{N/p})} & \leq C \left(\|\mathrm{adj}(DX_{v})-Id\|_{\tilde{L}_{T}^{\infty}(\dot{B}_{p,1}^{N/p})} + 1\right)\|w\|^{2}_{\tilde{L}_{T}^{2}(\dot{B}_{p,1}^{N/p})}    \\
& \leq C (1+ \|v\|_{\tilde{L}_{T}^{1}(\dot{B}_{p,1}^{N/p+1})}) \|w\|_{\tilde{L}_{T}^{2}(\dot{B}_{p,1}^{N/p})}^{2},
\end{align*}
\begin{align*}
\|I_{6}(v,w)\|_{L_{T}^{1}(\dot{B}_{p,1}^{N/p})} & \leq C (1+ \|v\|_{\tilde{L}_{T}^{1}(\dot{B}_{p,1}^{N/p+1})}) \|w\|_{\tilde{L}_{T}^{2}(\dot{B}_{p,1}^{N/p})}^{2},
\end{align*}
\begin{align*}
\|I_{7}(v,w)\|_{L_{T}^{1}(\dot{B}_{p,1}^{N/p-1})} & \leq C \|1-J_{v}\|_{\tilde{L}_{T}^{\infty}(\dot{B}_{p,1}^{N/p})}\|\partial_{t}w\|_{\tilde{L}_{T}^{1}(\dot{B}_{p,1}^{N/p-1})}   \\
& \leq C \|v\|_{\tilde{L}_{T}^{1}(\dot{B}_{p,1}^{N/p+1})} \|\partial_{t}w\|_{\tilde{L}_{T}^{1}(\dot{B}_{p,1}^{N/p-1})},
\end{align*}
\begin{align*}
\|I_{8}(v,w)\|_{L_{T}^{1}(\dot{B}_{p,1}^{N/p})} & \leq C \|\mathrm{adj}(DX_{v}) - Id\|_{\tilde{L}_{T}^{\infty}(\dot{B}_{p,1}^{N/p})}
\|A_{v}^{T} - Id\|_{\tilde{L}_{T}^{\infty}(\dot{B}_{p,1}^{N/p})} \|w\|_{\tilde{L}_{T}^{1}(\dot{B}_{p,1}^{N/p+1})}   \\
& + C \|\mathrm{adj}(DX_{v})-Id\|_{\tilde{L}_{T}^{\infty}(\dot{B}_{p,1}^{N/p})}\|w\|_{\tilde{L}_{T}^{1}(\dot{B}_{p,1}^{N/p})}   \\
& \leq C \|v\|_{\tilde{L}_{T}^{1}(\dot{B}_{p,1}^{N/p+1})} \|w\|_{\tilde{L}_{T}^{1}(\dot{B}_{p,1}^{N/p+1})},
\end{align*}
\begin{align*}
\|I_{9}(v,w)\|_{L_{T}^{1}(\dot{B}_{p,1}^{N/p})} & \leq C \|A_{v}^{T}-Id\|_{\tilde{L}_{T}^{\infty}(\dot{B}_{p,1}^{N/p})}
\|w\|_{\tilde{L}_{T}^{1}(\dot{B}_{p,1}^{N/p+1})}  \\
& \leq C \|v\|_{\tilde{L}_{T}^{1}(\dot{B}_{p,1}^{N/p+1})} \|w\|_{\tilde{L}_{T}^{1}(\dot{B}_{p,1}^{N/p+1})},
\end{align*}
\begin{align*}
\|I_{10}(v,w,k)\|_{L_{T}^{1}(\dot{B}_{p,1}^{N/p})} & \leq C (1+ \|\mathrm{adj}(\nabla X_{v})-Id\|_{\tilde{L}_{T}^{\infty}(\dot{B}_{p,1}^{N/p})})
\|w\|_{\tilde{L}_{T}^{2}(\dot{B}_{p,1}^{N/p})} \|k\|_{\tilde{L}_{T}^{2}(\dot{B}_{p,1}^{N/p})}   \\
& \leq C (1+ \|v\|_{\tilde{L}_{T}^{1}(\dot{B}_{p,1}^{N/p})}) \|k\|_{\tilde{L}_{T}^{2}(\dot{B}_{p,1}^{N/p})} \|w\|_{\tilde{L}_{T}^{2}(\dot{B}_{p,1}^{N/p})},
\end{align*}
\begin{align*}
\|I_{11}(v,w,k)\|_{L_{T}^{1}(\dot{B}_{p,1}^{N/p})} & \leq C (1+ \|v\|_{\tilde{L}_{T}^{1}(\dot{B}_{p,1}^{N/p})}) \|k\|_{\tilde{L}_{T}^{2}(\dot{B}_{p,1}^{N/p})} \|w\|_{\tilde{L}_{T}^{2}(\dot{B}_{p,1}^{N/p})}.
\end{align*}
Combining the estimates about $I_{1}$ to $I_{11}$, we could finally get
\begin{align*}
\|\hat{u}\|_{E_{p}(T)} \leq & C e^{C_{\rho_{0},m}T} (1+\|a_{0}\|_{\dot{B}_{p,1}^{N/p}})^{2} \Big\{ T + \|a_{0}\|_{\dot{B}_{p,1}^{N/p}}\|u_{L}\|_{\tilde{L}_{T}^{1}(\dot{B}_{p,1}^{N/p+1})}   \\
+ & (\|\hat{v}\|_{\tilde{L}_{T}^{1}(\dot{B}_{p,1}^{N/p+1})}+\|u_{L}\|_{\tilde{L}_{T}^{1}(\dot{B}_{p,1}^{N/p+1})})
(\|\partial_{t}\hat{v}\|_{\tilde{L}_{T}^{1}(\dot{B}_{p,1}^{N/p-1})} + \|\partial_{t}u_{L}\|_{\tilde{L}_{T}^{1}(\dot{B}_{p,1}^{N/p-1})}) \\
& + (\|\hat{v}\|_{\tilde{L}_{T}^{1}(\dot{B}_{p,1}^{N/p+1})} + \|u_{L}\|_{\tilde{L}_{T}^{1}(\dot{B}_{p,1}^{N/p+1})})^{2} \\
& + (\|\hat{b}\|_{\tilde{L}_{T}^{2}(\dot{B}_{p,1}^{N/p})} + \|b_{L}\|_{\tilde{L}_{T}^{2}(\dot{B}_{p,1}^{N/p})})^{2}
\Big\},
\end{align*}
and
\begin{align*}
\|\hat{B}\|_{E_{p}(T)} \leq & C \Big( (\|u_{L}\|_{\tilde{L}_{T}^{1}(\dot{B}_{p,1}^{N/p+1})}
+ \|\hat{v}\|_{\tilde{L}_{T}^{1}(\dot{B}_{p,1}^{N/p+1})}) (\|\partial_{t}B_{L}\|_{\tilde{L}_{T}^{1}(\dot{B}_{p,1}^{N/p-1})} \\
& + \|\partial_{t}\hat{b}\|_{\tilde{L}_{T}^{1}(\dot{B}_{p,1}^{N/p-1})}
+ \|\hat{b}\|_{\tilde{L}_{T}^{1}(\dot{B}_{p,1}^{N/p+1})} + \|B_{L}\|_{\tilde{L}_{T}^{1}(\dot{B}_{p,1}^{N/p+1})})    \\
& + (\|\hat{v}\|_{\tilde{L}_{T}^{2}(\dot{B}_{p,1}^{N/p})}+\|v_{L}\|_{\tilde{L}_{T}^{2}(\dot{B}_{p,1}^{N/p})})
(\|\hat{b}\|_{\tilde{L}_{T}^{2}(\dot{B}_{p,1}^{N/p})} + \|B_{L}\|_{\tilde{L}_{T}^{2}(\dot{B}_{p,1}^{N/p})}) \Big),
\end{align*}
where $\hat{v} = v-u_{L}$, $\hat{b} = b-B_{L}$.

We could reduce the above two inequalities further to get
\begin{align}\label{local u}
\begin{split}
\|\hat{u}\|_{E_{p}(T)} \leq & C e^{C_{\rho_{0},m}T} (1+\|a_{0}\|_{\dot{B}_{p,1}^{N/p}})^{2} \Big\{ T
+ \|a_{0}\|_{\dot{B}_{p,1}^{N/p}}\|u_{L}\|_{\tilde{L}_{T}^{1}(\dot{B}_{p,1}^{N/p+1})}   \\
& + (R + \|u_{L}\|_{\tilde{L}_{T}^{1}(\dot{B}_{p,1}^{N/p})})(R + \|\partial_{t}u_{L}\|_{\tilde{L}_{T}^{1}(\dot{B}_{p,1}^{N/p-1})}   \\
& + \|u_{L}\|_{\tilde{L}_{T}^{1}(\dot{B}_{p,1}^{N/p+1})}) + (R + \|B_{L}\|_{\tilde{L}_{T}^{2}(\dot{B}_{p,1}^{N/p})})^{2}
\Big\},
\end{split}
\end{align}
and
\begin{align}\label{local B}
\begin{split}
\|\hat{B}\|_{E_{p}(T)} \leq & C (R + \|u_{L}\|_{\tilde{L}_{T}^{2}(\dot{B}_{p,1}^{N/p})} + \|u_{L}\|_{\tilde{L}_{T}^{1}(\dot{B}_{p,1}^{N/p})})
(R + \|\partial_{t}B_{L}\|_{\tilde{L}_{T}^{1}(\dot{B}_{p,1}^{N/p-1})} \\
& + \|B_{L}\|_{\tilde{L}_{T}^{1}(\dot{B}_{p,1}^{N/p+1})} +
\|B_{L}\|_{\tilde{L}_{T}^{2}(\dot{B}_{p,1}^{N/p})}).
\end{split}
\end{align}
Here, if we assume
\begin{align}\label{timebound1}
\begin{split}
& C_{\rho_{0},m}T \leq \log 2, \quad T \leq R^{2},  \\
& \|a_{0}\|_{\dot{B}_{p,1}^{N/p}}\|u_{L}\|_{\tilde{L}_{T}^{1}(\dot{B}_{p,1}^{N/p})} \leq R^{2}, \\
& \|\partial_{t}u_{L}\|_{\tilde{L}_{T}^{1}(\dot{B}_{p,1}^{N/p-1})} + \|u_{L}\|_{\tilde{L}_{T}^{1}(\dot{B}_{p,1}^{N/p+1})}
+ \|u_{L}\|_{\tilde{L}_{T}^{2}(\dot{B}_{p,1}^{N/p})} \leq R,  \\
& \|\partial_{t}B_{L}\|_{\tilde{L}_{T}^{1}(\dot{B}_{p,1}^{N/p-1})} + \|B_{L}\|_{\tilde{L}_{T}^{1}(\dot{B}_{p,1}^{N/p+1})}
+ \|B_{L}\|_{\tilde{L}_{T}^{2}(\dot{B}_{p,1}^{N/p})} \leq R,  \\
& (1+\|a_{0}\|_{\dot{B}_{p,1}^{N/p}})^{2}R \leq \eta < \frac{1}{20 C},
\end{split}
\end{align}
then we know that
\begin{align*}
\|\hat{u}\|_{E_{p}(T)} \leq R, \quad \|\hat{B}\|_{E_{p}(T)} \leq R.
\end{align*}
Hence, $\mathbf{\Phi}$ is a self map on the ball $\bar{B}_{E_{p}(T)}(u_{L},B_{L};R)$.

$\mathbf{Step \,\, 2:Contraction\,\,\, estimates.}$
Taking $(v^{1},b^{1}) \in \bar{B}_{E_{p}(T)}(u_{L},B_{L};R)$ and $(v^{2},b^{2}) \in \bar{B}_{E_{p}(T)}(u_{L},B_{L};R)$, we define
$(\tilde{u}^{1},\tilde{B}^{1}) = \mathbf{\Phi}(v^{1},b^{1})$ and $(\tilde{u}^{2},\tilde{B}^{2}) = \mathbf{\Phi}(v^{1},b^{1})$.
Denote $\delta\tilde{u} = \tilde{u}^{2} - \tilde{u}^{1}$, $\delta\tilde{B} = \tilde{B}^{2} - \tilde{B}^{1}$,
$\delta v = v^{2} - v^{1}$ and $\delta b = b^{2} - b^{1}$.

Through simple calculations, we have
\begin{align*}
L_{\rho_{0}}(\delta \tilde{u}) = & I_{1}(v^{1},\delta v) + (J_{v^{1}} - J_{v^{2}})\partial_{t}v^{2} \\
& + \rho_{0}^{-1}\Big\{ (I_{2}(v^{2},v^{2})-I_{2}(v^{1},v^{1})) + (I_{3}(v^{2},v^{2})-I_{3}(v^{1},v^{1}))   \\
& + (I_{4}(v^{2})-I_{4}(v^{1})) + (I_{5}(v^{2},b^{2})-I_{5}(v^{1},b^{1}))   \\
& + (I_{6}(v^{2},b^{2})-I_{6}(v^{1},b^{1}))
\Big\},
\end{align*}
and
\begin{align*}
L_{\nu}(\delta \tilde{B}) = & (I_{7}(v^{2},b^{2})-I_{7}(v^{1},b^{1})) + \mathrm{div}\Big\{ (I_{8}(v^{2},b^{2})-I_{8}(v^{1},b^{1}))  \\
& + (I_{9}(v^{2},b^{2})-I_{9}(v^{1},b^{1})) + (I_{10}(v^{2},b^{2},v^{2})-I_{10}(v^{1},b^{1},v^{1}))      \\
& + (I_{11}(v^{2},b^{2},v^{2})-I_{11}(v^{1},b^{1},v^{1}))
\Big\}.
\end{align*}
Under the condition $C_{\rho_{0},m}T \leq \log 2$, using Lemma \ref{variable heat eq} and Lemma \ref{heat eq}, we get
\begin{align*}
\|\delta \tilde{u}\|_{E_{p}(T)} \leq & C (1+\|a_{0}\|_{\dot{B}_{p,1}^{N/p}})\Big( \|I_{1}(v^{1},\delta v)\|_{\tilde{L}_{T}^{1}(\dot{B}_{p,1}^{N/p-1})}  \\
& + \|(J_{v^{1}}-J_{v^{2}})\partial_{t}v^{2}\|_{\tilde{L}_{T}^{1}(\dot{B}_{p,1}^{N/p-1})}
+ \|I_{2}(v^{2},v^{2})-I_{2}(v^{1},v^{1})\|_{\tilde{L}_{T}^{1}(\dot{B}_{p,1}^{N/p})}      \\
& + \|I_{3}(v^{2},v^{2})-I_{3}(v^{1},v^{1})\|_{\tilde{L}_{T}^{1}(\dot{B}_{p,1}^{N/p})}
+ \|I_{4}(v^{2})-I_{4}(v^{1})\|_{\tilde{L}_{T}^{1}(\dot{B}_{p,1}^{N/p})}      \\
& + \|I_{5}(v^{2},b^{2})-I_{5}(v^{1},b^{1})\|_{\tilde{L}_{T}^{1}(\dot{B}_{p,1}^{N/p})}    \\
& + \|I_{6}(v^{2},b^{2})-I_{6}(v^{1},b^{1})\|_{\tilde{L}_{T}^{1}(\dot{B}_{p,1}^{N/p})}
\Big),
\end{align*}
and
\begin{align*}
\|\delta\tilde{B}\|_{E_{p}(T)} \leq & C \Big( \|I_{7}(v^{2},b^{2})-I_{7}(v^{1},b^{1})\|_{\tilde{L}_{T}^{1}(\dot{B}_{p,1}^{N/p-1})}    \\
& + \|I_{8}(v^{2},b^{2})-I_{8}(v^{1},b^{1})\|_{\tilde{L}_{T}^{1}(\dot{B}_{p,1}^{N/p})}    \\
& + \|I_{9}(v^{2},b^{2})-I_{9}(v^{1},b^{1})\|_{\tilde{L}_{T}^{1}(\dot{B}_{p,1}^{N/p})}    \\
& + \|I_{10}(v^{2},b^{2},v^{2})-I_{10}(v^{1},b^{1},v^{1})\|_{\tilde{L}_{T}^{1}(\dot{B}_{p,1}^{N/p})}  \\
& + \|I_{11}(v^{2},b^{2},v^{2})-I_{11}(v^{1},b^{1},v^{1})\|_{\tilde{L}_{T}^{1}(\dot{B}_{p,1}^{N/p})}
\Big).
\end{align*}
There are many terms appeared in the Navier-Stokes system, we only list the estimates
\begin{align*}
\|I_{1}(v^{1},\delta v)\|_{\tilde{L}_{T}^{1}(\dot{B}_{p,1}^{N/p-1})} \leq & C \|v^{1}\|_{\tilde{L}_{T}^{1}(\dot{B}_{p,1}^{N/p+1})}
\|\partial_{t}\delta v\|_{\tilde{L}_{T}^{1}(\dot{B}_{p,1}^{N/p-1})},
\end{align*}
\begin{align*}
\|(J_{v^{1}}-J_{v^{2}})\partial_{t}v^{2}\|_{\tilde{L}_{T}^{1}(\dot{B}_{p,1}^{N/p-1})} \leq & C \|\delta v\|_{\tilde{L}_{T}^{1}(\dot{B}_{p,1}^{N/p+1})}
\|\partial_{t}v^{2}\|_{\tilde{L}_{T}^{1}(\dot{B}_{p,1}^{N/p-1})},
\end{align*}
\begin{align*}
\|I_{2}(v^{2},v^{2})-I_{2}(v^{1},v^{1})\|_{\tilde{L}_{T}^{1}(\dot{B}_{p,1}^{N/p})} \leq & C \|(v^{1},v^{2})\|_{\tilde{L}_{T}^{1}(\dot{B}_{p,1}^{N/p+1})}
\|\delta v\|_{\tilde{L}_{T}^{1}(\dot{B}_{p,1}^{N/p+1})},
\end{align*}
\begin{align*}
\|I_{3}(v^{2},v^{2})-I_{3}(v^{1},v^{1})\|_{\tilde{L}_{T}^{1}(\dot{B}_{p,1}^{N/p})} \leq & C \|(v^{1},v^{2})\|_{\tilde{L}_{T}^{1}(\dot{B}_{p,1}^{N/p+1})}
\|\delta v\|_{\tilde{L}_{T}^{1}(\dot{B}_{p,1}^{N/p+1})},
\end{align*}
\begin{align*}
\|I_{4}(v^{2}) - I_{4}(v^{1})\|_{\tilde{L}_{T}^{1}(\dot{B}_{p,1}^{N/p})} \leq & C (1+ \|a_{0}\|_{\dot{B}_{p,1}^{N/p}}) T
\|\delta v\|_{\tilde{L}_{T}^{1}(\dot{B}_{p,1}^{N/p+1})}.
\end{align*}
Now, we analyze the different terms carefully. Since
\begin{align*}
I_{5}(v^{2},b^{2}) - I_{5}(v^{1},b^{1}) = & (\mathrm{adj}(DX_{v^{2}})-\mathrm{adj}(DX_{v^{1}}))b^{2}(b^{2})^{T}   \\
& + \mathrm{adj}(DX_{v^{1}})(b^{2}-b^{1})(b^{2})^{T}    \\
& + \mathrm{adj}(DX_{v^{1}})b^{1}(b^{2}-b^{1})^{T},
\end{align*}
we have
\begin{align*}
\|I_{5}(v^{2},b^{2})-I_{5}(v^{1},b^{1})\|_{\tilde{L}_{T}^{1}(\dot{B}_{p,1}^{N/p})} \leq &
C \|\delta v\|_{\tilde{L}_{T}^{1}(\dot{B}_{p,1}^{N/p+1})} \|b^{2}\|_{\tilde{L}_{T}^{2}(\dot{B}_{p,1}^{N/p})}^{2}    \\
+ C (1 + & \|v^{1}\|_{\tilde{L}_{T}^{1}(\dot{B}_{p,1}^{N/p+1})})\|\delta b\|_{\tilde{L}_{T}^{2}(\dot{B}_{p,1}^{N/p})}
\|(b^{1},b^{2})\|_{\tilde{L}_{T}^{2}(\dot{B}_{p,1}^{N/p})}.
\end{align*}
$I_{6}$ can be estimated same as $I_{5}$. Since
\begin{align*}
I_{7}(v^{2},b^{2}) - I_{7}(v^{1},b^{1}) = (J_{v^{1}}-J_{v^{2}})\partial_{t}b^{2}
+ (1-J_{v^{1}})(\partial_{t}b^{2}-\partial_{t}b^{1}),
\end{align*}
we know that
\begin{align*}
\|I_{7}(v^{2},b^{2})-I_{7}(v^{1},b^{1})\|_{\tilde{L}_{T}^{1}(\dot{B}_{p,1}^{N/p-1})} \leq &
C \|J_{v^{2}}-J_{v^{1}}\|_{\tilde{L}_{T}^{\infty}(\dot{B}_{p,1}^{N/p})}\|\partial_{t}b^{2}\|_{\tilde{L}_{T}^{1}(\dot{B}_{p,1}^{N/p-1})} \\
& + C \|J_{v^{1}}-1\|_{\tilde{L}_{T}^{\infty}(\dot{B}_{p,1}^{N/p})} \|\partial_{t}\delta b\|_{\tilde{L}_{T}^{1}(\dot{B}_{p,1}^{N/p-1})} \\
\leq & C \|\partial_{t}b^{2}\|_{\tilde{L}_{T}^{1}(\dot{B}_{p,1}^{N/p-1})}\|\delta v\|_{\tilde{L}_{T}^{1}(\dot{B}_{p,1}^{N/p+1})}    \\
& + C \|v^{1}\|_{\tilde{L}_{T}^{1}(\dot{B}_{p,1}^{N/p})} \|\partial_{t}\delta b\|_{\tilde{L}_{T}^{1}(\dot{B}_{p,1}^{N/p-1})}.
\end{align*}
After some calculations, we have
\begin{align*}
I_{8}(v^{2},b^{2}) - I_{8}(v^{1},b^{1}) = & \nu (\mathrm{adj}(DX_{v^{2}})-\mathrm{adj}(DX_{v^{1}}))A_{v^{2}}^{T}\nabla b^{2}  \\
& + \nu (\mathrm{adj}(DX_{v^{1}})-Id)(A_{v^{2}}^{T}-A_{v^{1}}^{T})\nabla b^{2}  \\
& + \nu (\mathrm{adj}(DX_{v^{1}})-Id)A_{v^{1}}^{T}\nabla (b^{2}-b^{1}),
\end{align*}
so we can get
\begin{align*}
& \|I_{8}(v^{2},b^{2}) - I_{8}(v^{1},b^{1})\|_{\tilde{L}_{T}^{1}(\dot{B}_{p,1}^{N/p})} \\
\leq & C \|\delta v\|_{\tilde{L}_{T}^{1}(\dot{B}_{p,1}^{N/p+1})}
(1+\|v^{2}\|_{\tilde{L}_{T}^{1}(\dot{B}_{p,1}^{N/p+1})})\|b^{2}\|_{\tilde{L}_{T}^{1}(\dot{B}_{p,1}^{N/p+1})}    \\
& + C \|v^{1}\|_{\tilde{L}_{T}^{1}(\dot{B}_{p,1}^{N/p+1})}\|\delta v\|_{\tilde{L}_{T}^{1}(\dot{B}_{p,1}^{N/p+1})}\|b^{2}\|_{\tilde{L}_{T}^{1}(\dot{B}_{p,1}^{N/p+1})} \\
& + C \|v^{1}\|_{\tilde{L}_{T}^{1}(\dot{B}_{p,1}^{N/p+1})} (1 + \|v^{1}\|_{\tilde{L}_{T}^{1}(\dot{B}_{p,1}^{N/p+1})}) \|\delta b\|_{\tilde{L}_{T}^{1}(\dot{B}_{p,1}^{N/p+1})}.
\end{align*}
Due to
\begin{align*}
I_{9}(v^{2},b^{2})-I_{9}(v^{1},b^{1}) = \nu (A_{v^{2}}^{T}-A_{v^{1}}^{T})\nabla b^{2} + \nu (A_{v^{1}}^{T}-Id)\nabla \delta b,
\end{align*}
we have
\begin{align*}
\|I_{9}(v^{2},b^{2})-I_{9}(v^{1},b^{1})\|_{\tilde{L}_{T}^{1}(\dot{B}_{p,1}^{N/p})} \leq &
C \|\delta v\|_{\tilde{L}_{T}^{1}(\dot{B}_{p,1}^{N/p+1})} \|b^{2}\|_{\tilde{L}_{T}^{1}(\dot{B}_{p,1}^{N/p})}    \\
& + C \|v^{1}\|_{\tilde{L}_{T}^{1}(\dot{B}_{p,1}^{N/p+1})} \|\delta b\|_{\tilde{L}_{T}^{1}(\dot{B}_{p,1}^{N/p+1})}.
\end{align*}
Since
\begin{align*}
I_{10}(v^{2},b^{2},v^{2}) - I_{10}(v^{1},b^{1},v^{1})
= & (\mathrm{adj}(\nabla X_{v^{2}})-\mathrm{adj}(\nabla X_{v^{1}}))b^{2}(v^{2})^{T} \\
& + \mathrm{adj}(\nabla X_{v^{1}})(b^{2}-b^{1})(v^{2})^{T}      \\
& + \mathrm{adj}(\nabla X_{v^{1}})b^{1}(v^{2}-v^{1})^{T},
\end{align*}
we easily have
\begin{align*}
& \|I_{10}(v^{2},b^{2},v^{2}) - I_{10}(v^{1},b^{1},v^{1})\|_{\tilde{L}_{T}^{1}(\dot{B}_{p,1}^{N/p})}    \\
\leq & C
\|\delta v\|_{\tilde{L}_{T}^{1}(\dot{B}_{p,1}^{N/p+1})} \|b^{2}\|_{\tilde{L}_{T}^{2}(\dot{B}_{p,1}^{N/p})} \|v^{2}\|_{\tilde{L}_{T}^{2}(\dot{B}_{p,1}^{N/p})} \\
& + C (1+\|v^{1}\|_{\tilde{L}_{T}^{1}(\dot{B}_{p,1}^{N/p+1})})\|\delta b\|_{\tilde{L}_{T}^{2}(\dot{B}_{p,1}^{N/p})}\|v^{2}\|_{\tilde{L}_{T}^{2}(\dot{B}_{p,1}^{N/p})} \\
& + C (1+\|v^{1}\|_{\tilde{L}_{T}^{1}(\dot{B}_{p,1}^{N/p+1})})\|b^{1}\|_{\tilde{L}_{T}^{2}(\dot{B}_{p,1}^{N/p})}\|\delta v\|_{\tilde{L}_{T}^{2}(\dot{B}_{p,1}^{N/p})}.
\end{align*}
The term $\|I_{11}(v^{2},b^{2},v^{2}) - I_{11}(v^{1},b^{1},v^{1})\|_{\tilde{L}_{T}^{1}(\dot{B}_{p,1}^{N/p})}$ can be
estimated similar to the above terms,
so we not write the details.
Combining all the above estimates, we finally arrive at
\begin{align*}
\|\delta \tilde{u}\|_{E_{p}(T)} \leq & C (1+\|a_{0}\|_{\dot{B}_{p,1}^{N/p}})^{2}\Big\{ (T + \|(v^{1},v^{2})\|_{\tilde{L}_{T}^{1}(\dot{B}_{p,1}^{N/p+1})}    \\
& + \|\partial_{t}v^{2}\|_{\tilde{L}_{T}^{1}(\dot{B}_{p,1}^{N/p})} + \|b^{2}\|_{\tilde{L}_{T}^{2}(\dot{B}_{p,1}^{N/p})}^{2} )
\|\delta v\|_{\tilde{L}_{T}^{1}(\dot{B}_{p,1}^{N/p})}     \\
& + \|v^{1}\|_{\tilde{L}_{T}^{1}(\dot{B}_{p,1}^{N/p+1})}\|\partial_{t}\delta v\|_{\tilde{L}_{T}^{1}(\dot{B}_{p,1}^{N/p-1})}
+ \|(b^{1},b^{2})\|_{\tilde{L}_{T}^{2}(\dot{B}_{p,1}^{N/p})}\|\delta b\|_{\tilde{L}_{T}^{2}(\dot{B}_{p,1}^{N/p})}
\Big\},
\end{align*}
and
\begin{align*}
\|\delta \tilde{B}\|_{E_{p}(T)} \leq & C \Big\{ (\|\partial_{t}b^{2}\|_{\tilde{L}_{T}^{1}(\dot{B}_{p,1}^{N/p-1})} + \|b^{2}\|_{\tilde{L}_{T}^{1}(\dot{B}_{p,1}^{N/p+1})} + \|v^{1}\|_{\tilde{L}_{T}^{1}(\dot{B}_{p,1}^{N/p+1})}\|b^{2}\|_{\tilde{L}_{T}^{1}(\dot{B}_{p,1}^{N/p+1})}   \\
& + \|v^{2}\|_{\tilde{L}_{T}^{2}(\dot{B}_{p,1}^{N/p})}\|b^{2}\|_{\tilde{L}_{T}^{2}(\dot{B}_{p,1}^{N/p})} ) \|\delta v\|_{\tilde{L}_{T}^{1}(\dot{B}_{p,1}^{N/p+1})}    \\
& + \|b^{1}\|_{\tilde{L}_{T}^{2}(\dot{B}_{p,1}^{N/p})}\|\delta v\|_{\tilde{L}_{T}^{2}(\dot{B}_{p,1}^{N/p})}
+ \|v^{1}\|_{\tilde{L}_{T}^{1}(\dot{B}_{p,1}^{N/p+1})}\|\partial_{t}\delta b\|_{\tilde{L}_{T}^{1}(\dot{B}_{p,1}^{N/p-1})}  \\
& + \|v^{1}\|_{\tilde{L}_{T}^{1}(\dot{B}_{p,1}^{N/p+1})} \|\delta b\|_{\tilde{L}_{T}^{1}(\dot{B}_{p,1}^{N/p+1})}
+ \|v^{2}\|_{\tilde{L}_{T}^{2}(\dot{B}_{p,1}^{N/p})}\|\delta b\|_{\tilde{L}_{T}^{2}(\dot{B}_{p,1}^{N/p})}
\Big\}.
\end{align*}
Through some simple calculations, using conditions (\ref{timebound1}) with may be lager constant $C$, we have
\begin{align*}
\|\delta \tilde{u}\|_{E_{p}(T)} \leq C (1+\|a_{0}\|_{\dot{B}_{p,1}^{N/p}})^{2}R(\|\delta v\|_{E_{p}(T)}+\|\delta b\|_{E_{p}(T)}),
\end{align*}
and
\begin{align*}
\|\delta\tilde{B}\|_{E_{p}(T)} \leq C R (\|\delta v\|_{E_{p}(T)} + \|\delta b\|_{E_{p}(T)}).
\end{align*}
Hence, we get
\begin{align*}
\|(\delta\tilde{u},\delta\tilde{B})\|_{E_{p}(T)} \leq C (1+\|a_{0}\|_{\dot{B}_{p,1}^{N/p}})^{2}R\|(\delta v, \delta b)\|_{E_{p}(T)}.
\end{align*}
Due to conditions (\ref{timebound1}), we finally get
\begin{align}\label{lipschitz half}
\|(\delta\tilde{u},\delta\tilde{B})\|_{E_{p}(T)} \leq \frac{1}{2} \|(\delta v,\delta b)\|_{E_{p}(T)}.
\end{align}
Now, by contraction mapping theorem, we know that $\mathbf{\Phi}$ admits a unique fixed point in $\bar{B}_{E_{p}(T)}(u_{L},B_{L};R)$.

$\mathbf{Step \,\, 3:Regularity\,\,\, of\,\,\, the\,\,\, density.}$
Denoting $a = \rho - 1$, and we already know that $\rho = J_{u}^{-1}\rho_{0}$.
So we have
\begin{align*}
a = (J_{u}^{-1}-1)a_{0} + a_{0} + (J_{u}^{-1}-1).
\end{align*}
Due to $u \in L^{1}(0,T;\dot{B}_{p,1}^{N/p+1})$, we know that $J_{u}^{-1}(t)-1 \in C([0,T];\dot{B}_{p,1}^{N/p})$.
So we know that $a$ belongs to $C([0,T];\dot{B}_{p,1}^{N/p})$.
If
\begin{align}\label{condition for R}
R \leq \frac{c_{0}}{4(1+\|a_{0}\|_{\dot{B}_{p,1}^{N/p}})},
\end{align}
we have
\begin{align*}
\rho(t) & \geq (1+a_{0}) - (1+\|a_{0}\|_{L^{\infty}})(\|u_{L}\|_{\tilde{L}_{T}^{1}(\dot{B}_{p,1}^{N/p})}+R)   \\
& \geq (1+a_{0}) - 2(1+\|a_{0}\|_{\dot{B}_{p,1}^{N/p}})R  \\
& > 0,
\end{align*}
where $c_{0}$ defined as in Theorem \ref{mainexistence} and $t \in [0,T]$.

$\mathbf{Step \,\, 4:Uniqueness\,\,\, and\,\,\, continuity\,\,\, of\,\,\, the\,\,\, flow\,\,\, map.}$
Consider two couples of initial data $(\rho_{0}^{1},u_{0}^{1},B_{0}^{1})$ and $(\rho_{0}^{2},u_{0}^{2},B_{0}^{2})$.
Define $\delta \rho = \rho^{2} - \rho^{1}$, $\delta u = u^{2} - u^{1}$ and $\delta B = B^{2} - B^{1}$,
we can easily find the equations for $\delta u$ and $\delta B$. Using similar methods as Navier-Stokes system and estimate
the terms containing $B$ as in Step 3, we can finish this part. For there are no new gradients, we omit the details.

At this stage, by same reasons as in \cite{LagrangianCompressible}, we can get the local well-posedness results for our MHD system.

$\mathbf{Step \,\, 5:Lower\,\,\, bound\,\,\, of\,\,\, the\,\,\, existence\,\,\, time.}$
Taking $m$ to be a large constant fixed as before and $C$ to be a large enough constant, let
\begin{align}
\begin{split}
& \eta \leq \frac{1}{C}, \quad C_{\rho_{0},m}T \leq \log 2, \quad (1+\|a_{0}\|_{\dot{B}_{p,1}^{N/p}})^{2} R \leq \eta, \quad T \leq R^{2},    \\
& \|a_{0}\|_{\dot{B}_{p,1}^{N/p}}\|u_{0}\|_{\dot{B}_{p,1}^{N/p-1}} \leq R^{2}, \quad \|u_{0}\|_{\dot{B}_{p,1}^{N/p-1}} \leq R,
\quad \|B_{0}\|_{\dot{B}_{p,1}^{N/p-1}}\leq R,
\end{split}
\end{align}
then the conditions (\ref{timebound1}) and (\ref{condition for R}) are satisfied.
From the above conditions, we know that
\begin{align}\label{lower bound of T 1}
T \leq \frac{1}{C^{2}(1+\|a_{0}\|_{\dot{B}_{p,1}^{N/p}})^{4}}, \quad T \leq \frac{\log 2}{C_{\rho_{0},m}}.
\end{align}
Next, we calculate $C_{\rho_{0},m}$ carefully as follows
\begin{align}\label{explicit form of C rho m}
\begin{split}
C_{\rho_{0},m} & \leq C \int_{0}^{t} \left\|S_{m}\nabla\left( \frac{1}{\rho_{0}} \right)\right\|_{\dot{B}_{p,1}^{N/p}}^{2} \, d\tau   \\
& \leq C \int_{0}^{t}\|S_{m}\nabla\rho_{0}\|_{\dot{B}_{p,1}^{N/p}}^{2}\,d\tau     \\
& \leq C T 2^{2m} \|a_{0}\|_{\dot{B}_{p,1}^{N/p}}^{2}
\end{split}
\end{align}
Combining (\ref{lower bound of T 1}) and (\ref{explicit form of C rho m}), we know that we can take $T$ as
\begin{align*}
T = \frac{\bar{c}}{(1+\|a_{0}\|_{\dot{B}_{p,1}^{N/p}})^{4}},
\end{align*}
where $\bar{c}$ is a small enough positive constant.
At this point, we finished our proof of Theorem \ref{local well-posedness}.

Next, let us go to the second part of this section to
prove the solution in Theorem \ref{local well-posedness} can propagate the smoothness of the initial data.

Before our proof, we need to introduce some notations which can also be found in \cite{NS2013Zhifei,localMiao2010}.
Define a weight function $\{ \omega_{k}(t) \}_{k \in \mathds{Z}}$ as follows
\begin{align*}
\omega_{k}(t) = \sum_{\ell \geq k} 2^{k - \ell} (1 - e^{-c 2^{2j}t})^{1/2}, \quad k \in \mathds{Z},
\end{align*}
where $c$ is a positive constant.
We easily know that for any $k \in \mathds{Z}$,
\begin{align}
\begin{split}
& \omega_{k}(t) \leq 2, \quad \omega_{k}(t) \sim \omega_{k'}(t) \quad \text{if} \,\, k \sim k', \\
& \omega_{k}(t) \leq 2^{k-k'}\omega_{k'}(t) \quad \text{if} \,\, k \geq k', \quad \omega_{k}(t) \leq 3 \omega_{k'}(t) \quad \text{if}\,\, k \leq k'.
\end{split}
\end{align}
Now, we can introduce the following weighted Besov space.
\begin{definition}
Let $s \in \mathds{R}$, $1\leq p,r\leq +\infty$, $0 < T < +\infty$. The weighted Besov space $\dot{B}_{p,r}^{s}(\omega)$ is defined by
\begin{align*}
\dot{B}_{p,r}^{s}(\omega) = \left\{ f \in \mathcal{S}_{h}' \, : \, \|f\|_{\dot{B}_{p,r}^{s}(\omega)} < +\infty \right\},
\end{align*}
where
\begin{align*}
\|f\|_{\dot{B}_{p,r}^{s}(\omega)} := \left\| \left( 2^{ks} \omega_{k}(T) \|\Delta_{k}f\|_{L^{p}} \right)_{k} \right\|.
\end{align*}
\end{definition}
Obviously, $\dot{B}_{p,r}^{s} \subset \dot{B}_{p,r}^{s}(\omega)$ and
\begin{align*}
\|f\|_{\dot{B}_{p,r}^{s}(\omega)} \leq 2 \|f\|_{\dot{B}_{p,r}^{s}}.
\end{align*}

We also need to define the Chemin-Lerner type weighted Besov spaces as follows.
\begin{definition}
Let $s \in \mathds{R}$, $1\leq p, q, r \leq +\infty$, $0 < T \leq +\infty$. The weighted functional space $\tilde{L}_{T}^{q}(\dot{B}_{p,r}^{s}(\omega))$
is defined as the set of all the distributions $f$ satisfying
\begin{align*}
\|f\|_{\tilde{L}_{T}^{q}(\dot{B}_{p,r}^{s}(\omega))} := \left\| \left( 2^{ks}\omega_{k}(T)\|\Delta_{k}f(t)\|_{L^{q}(0,T;L^{p})} \right)_{k} \right\|_{\ell^{r}}.
\end{align*}
\end{definition}

After this short introduction about weighted Besov space, we deduce some properties for the solution we obtained in the first part of this section.
Obviously, we have
\begin{align}\label{rho1}
\frac{c_{0}}{2} \leq \rho(t,x) \leq 2 c_{0}^{-1}.
\end{align}
Since
\begin{align*}
\|\tilde{u}\|_{\tilde{L}_{T}^{\infty}(\dot{B}_{p}^{N/p-1})} + \|\tilde{u}\|_{\tilde{L}_{T}^{1}(\dot{B}_{p}^{N/p+1})} \leq R \leq \frac{c_{3}}{(1+\|a_{0}\|_{\dot{B}_{p}^{N/p}})^{2}},
\end{align*}
and
\begin{align*}
\|u_{L}\|_{\tilde{L}_{T}^{\infty}(\dot{B}_{p}^{N/p-1})} + \|u_{L}\|_{\tilde{L}_{T}^{1}(\dot{B}_{p}^{N/p+1})} \leq \|u_{0}\|_{\dot{B}_{p}^{N/p-1}}
\leq \frac{c_{2}}{(1+\|a_{0}\|_{\dot{B}_{p}^{N/p}})^{5}},
\end{align*}
we know that
\begin{align}\label{u1}
\|u\|_{\tilde{L}_{T}^{\infty}(\dot{B}_{p}^{N/p-1})} + \|u\|_{\tilde{L}_{T}^{1}(\dot{B}_{p}^{N/p+1})} \leq \frac{c_{2}+c_{3}}{(1+\|a_{0}\|_{\dot{B}_{p}^{N/p}})^{5}}.
\end{align}
For the magnetic field, we have
\begin{align}\label{B1}
\|B\|_{\tilde{L}_{T}^{\infty}(\dot{B}_{p}^{N/p-1})} + \|B\|_{\tilde{L}_{T}^{1}(\dot{B}_{p}^{N/p+1})} \leq R + \|B_{0}\|_{\dot{B}_{p}^{N/p}}
\leq \frac{2c_{3}}{(1+\|a_{0}\|_{\dot{B}_{p}^{N/p}})^{2}}.
\end{align}
For the density, we know that
\begin{align*}
\|a\|_{\tilde{L}_{T}^{\infty}(\dot{B}_{p}^{N/p})} \leq \|D u\|_{\tilde{L}_{T}^{1}(\dot{B}_{p}^{N/p})}\|a_{0}\|_{\dot{B}_{p}^{N/p}}
+ \|a_{0}\|_{\dot{B}_{p}^{N/p}} + \|D u\|_{\tilde{L}_{T}^{1}(\dot{B}_{p}^{N/p})}.
\end{align*}
Hence, by some simple calculations, we obtain
\begin{align}\label{a1}
\|a\|_{\tilde{L}_{T}^{\infty}(\dot{B}_{p}^{N/p})} \leq 2 \|a_{0}\|_{\dot{B}_{p}^{N/p}} + \frac{c_{2}+c_{3}}{(1+\|a_{0}\|_{\dot{B}_{p}^{N/p}})^{5}}.
\end{align}
So combining (\ref{rho1}), (\ref{u1}), (\ref{B1}) and (\ref{a1}), for dimension $N = 3$, we know that the solution satisfies
\begin{align}\label{summary}
\begin{split}
& \frac{c_{0}}{2} \leq \rho(t,x) \leq 2 c_{0}^{-1}, \\
& \|a\|_{\tilde{L}_{T}^{\infty}(\dot{B}_{p}^{3/p})} \leq 2 \|a_{0}\|_{\dot{B}_{p}^{3/p}} + \frac{c_{2}+c_{3}}{(1+\|a_{0}\|_{\dot{B}_{p}^{3/p}})^{5}}, \\
& \|u\|_{\tilde{L}_{T}^{\infty}(\dot{B}_{p}^{3/p-1})} + \|u\|_{\tilde{L}_{T}^{1}(\dot{B}_{p}^{3/p+1})}
\leq \frac{c_{2}+c_{3}}{(1+\|a_{0}\|_{\dot{B}_{p}^{3/p}})^{5}},   \\
& \|B\|_{\tilde{L}_{T}^{\infty}(\dot{B}_{p}^{3/p-1})} + \|B\|_{\tilde{L}_{T}^{1}(\dot{B}_{p}^{3/p+1})} \leq \frac{2 c_{3}}{(1+\|a_{0}\|_{\dot{B}_{p}^{3/p}})^{2}}.
\end{split}
\end{align}

At this stage, we estimate $\|a\|_{\tilde{L}_{T}^{\infty}(\dot{B}_{p}^{3/p}(\omega))}$ carefully.
As in \cite{NS2013Zhifei}, we have
\begin{align*}
\sum_{j > M}2^{j\frac{3}{p}}\|\Delta_{j}a_{0}\|_{L^{p}} \leq C 2^{-\frac{1}{2}M}\|a_{0}\|_{H^{2}},
\end{align*}
and
\begin{align*}
\sum_{j \leq M}2^{j\frac{3}{p}}\omega_{j}(T)\|\Delta_{j}a_{0}\|_{L^{p}} \leq C 2^{2M} T^{\frac{1}{2}} \|a_{0}\|_{H^{2}}
+ C 2^{-\frac{1}{2}M} \|a_{0}\|_{H^{2}}.
\end{align*}
So if we take $M$ large enough and
\begin{align}\label{the form T}
T = \frac{c}{(1+\|a_{0}\|_{\dot{B}_{p}^{3/p}})^{4}(1+\|a_{0}\|_{H^{2}})^{12}},
\end{align}
we have
\begin{align}\label{bound of a omega}
\|a_{0}\|_{\dot{B}_{p}^{3/p}(\omega)} \leq c',
\end{align}
where $c$, $c'$ are small enough constants.
Similar to (\ref{a1}), we also have
\begin{align}\label{a1 weight}
\|a\|_{\tilde{L}_{T}^{\infty}(\dot{B}_{p}^{3/p}(\omega))} \leq \|a_{0}\|_{\dot{B}_{p}^{3/p}(\omega)} + \frac{c_{2}+c_{3}}{(1+ \|a_{0}\|_{\dot{B}_{p}^{3/p}})^{4}}.
\end{align}

In the following proposition, we show that this solution allows us to propagate the regularity of the initial data in Sobolev space
with low regularity.
\begin{proposition}\label{lowp}
Let $p \in (3,6)$ and $1-\frac{3}{p} < \delta < \frac{3}{p}$. Assume that $(\rho, u, B)$ is a solution for system (\ref{origianl MHD}).
If $(a_{0}, u_{0}, B_{0}) \in H^{1-\delta} \times \dot{H}^{-\delta} \times \dot{H}^{-\delta}$, then
\begin{align}\label{low pro}
\begin{split}
& \|a\|_{\tilde{L}_{T}^{\infty}(\dot{B}_{2,2}^{1-\delta})} \leq C \left( \|a_{0}\|_{H^{1-\delta}} + \|u_{0}\|_{\dot{H}^{-\delta}}
+ \|B_{0}\|_{\dot{H}^{-\delta}} \right),  \\
& \|(u,B)\|_{\tilde{L}_{T}^{\infty}(\dot{B}_{2,2}^{-\delta})} + \|(u,B)\|_{\tilde{L}_{T}^{1}(\dot{B}_{2,2}^{2-\delta})} \\
& \quad \leq C \left( 1 + \|a_{0}\|_{\dot{B}_{p}^{3/p}} \right)\left( \|(u_{0},B_{0})\|_{\dot{H}^{-\delta}} + TP_{+}\|a_{0}\|_{H^{1-\delta}} \right).
\end{split}
\end{align}
\end{proposition}
\begin{proof}
The proof can be divided into three steps. Set $\beta = 1-\delta$. Without loss of generality, we may assume that
$c$, $c_{2}$, $c_{3}$, $c'$, $T$ (choose as in (\ref{bound of a omega}) and  (\ref{the form T})) small enough such that
\begin{align}\label{condition}
\begin{split}
& C\frac{c_{2}+c_{3}}{(1+\|a_{0}\|_{\dot{B}_{p}^{3/p}})^{5}} + 2C \left( c' + \frac{c_{2}+c_{3}}{(1+\|a_{0}\|_{\dot{B}_{p}^{3/p}})^{4}} \right)   \\
& \quad
+ C \frac{2c_{2}}{(1+\|a_{0}\|_{\dot{B}_{p}^{3/p}})} + C \frac{2c_{2}}{(1+\|a_{0}\|_{\dot{B}_{p}^{3/p}})^{2}} + C T P_{+} \leq \frac{1}{2},
\end{split}
\end{align}
where $C$ is the constant appearing in the following estimates.\\
$\mathbf{Step \,\, 1:Estimate\,\,\, for\,\,\, the\,\,\,Transport\,\,\, Equation.}$
This estimate is similar to the Navier-Stokes system, so we omit the details and only give the results as follows
\begin{align}\label{transport estimate 1}
\|a\|_{\tilde{L}_{T}^{\infty}(\dot{B}_{2,2}^{\beta})} \leq C \left( \|a_{0}\|_{\dot{B}_{2,2}^{\beta}} + \|u\|_{\tilde{L}_{T}^{1}(\dot{B}_{2,2}^{\beta+1})} \right).
\end{align}
For details, we give the reference \cite{NS2013Zhifei} on page 192. \\
$\mathbf{Step \,\, 2:Estimate\,\,\, for\,\,\, the\,\,\, Momentum\,\,\, Equation.}$
Denote $\bar{\mu} = \frac{\mu}{\rho}$ and $\bar{\lambda} = \frac{\lambda}{\rho}$, apply the operator $\Delta_{j}$ to the
momentum equation of (\ref{origianl MHD}), we can obtain
\begin{align}\label{localequationfor u}
\begin{split}
& \partial_{t}\Delta_{j}u - \mathrm{div}(\bar{\mu}\nabla\Delta_{j}u) - \nabla(\bar{\lambda}\mathrm{div}\Delta_{j}u)   \\
& \quad\quad\quad = \Delta_{j}G + \mathrm{div}([\Delta_{j},\bar{\mu}]\nabla u) + \nabla ([\Delta_{j},\bar{\lambda}]\mathrm{div}u)    \\
& \quad\quad\quad\quad\, + \Delta_{j}\left( \frac{1}{\rho}B\cdot\nabla B \right) - \frac{1}{2}\Delta_{j}\left( \frac{1}{\rho}\nabla |B|^{2} \right),
\end{split}
\end{align}
where
\begin{align*}
G = -u\cdot\nabla u - \frac{\bar{\rho}P'(\rho)}{\rho}\nabla u + \frac{\mu}{\rho^{2}}\nabla\rho \nabla u + \frac{\lambda}{\rho^{2}}\nabla\rho \mathrm{div}u.
\end{align*}
Taking $L^{2}$ energy estimate for weighted Besov space, we have
\begin{align}\label{estimate weight 1}
\begin{split}
\|u\|_{\tilde{L}_{T}^{1}(\dot{B}_{2,2}^{\beta+1})} & + \|u\|_{\tilde{L}_{T}^{2}(\dot{B}_{2,2}^{\beta})} \leq
C (\|u_{0}\|_{\dot{B}_{2,2}^{\beta-1}} + \|G\|_{\tilde{L}_{T}^{1}(\dot{B}_{2,2}^{\beta-1}(\omega))})    \\
& + C \left\| 2^{j\beta}\omega_{j}(T)\left( \|[\Delta_{j},\bar{\mu}]\nabla u\|_{L_{T}^{1}(L^{2})} + \|[\Delta_{j},\bar{\lambda}]\nabla u\|_{L_{T}^{1}(L^{2})} \right) \right\|_{\ell^{2}} \\
& + C \left\| \frac{1}{\rho}B\cdot\nabla B \right\|_{\tilde{L}_{T}^{1}(\dot{B}_{2,2}^{\beta-1})}
+ C \left\| \frac{1}{\rho}\nabla (|B|^{2}) \right\|_{\tilde{L}_{T}^{1}(\dot{B}_{2,2}^{\beta-1})}
\end{split}
\end{align}
Similar to the Navier-Stokes system, we have
\begin{align}\label{same ns 1}
\begin{split}
& \left\| 2^{j\beta}\omega_{j}(T)\left( \|[\Delta_{j},\bar{\mu}]\nabla u\|_{L_{T}^{1}(L^{2})} + \|[\Delta_{j},\bar{\mu}]\nabla u\|_{L_{T}^{1}(L^{2})} \right) \right\|_{\ell^{2}} \\
& \quad
\leq C \|a\|_{\tilde{L}_{T}^{\infty}(\dot{B}_{p}^{3/p}(\omega))} \|u\|_{\tilde{L}_{T}^{1}(\dot{B}_{2,2}^{\beta+1})}
\end{split}
\end{align}
and
\begin{align}\label{same ns 2}
\begin{split}
\|G\|_{\tilde{L}_{T}^{1}(\dot{B}_{2,2}^{\beta-1}(\omega))} \leq & C \big( \|u\|_{\tilde{L}_{T}^{2}(\dot{B}_{p}^{3/p})}\|u\|_{\tilde{L}_{T}^{2}(\dot{B}_{2,2}^{\beta})}    \\
& + \|a\|_{\tilde{L}_{T}^{\infty}(\dot{B}_{p}^{3/p}(\omega))}\|u\|_{\tilde{L}_{T}^{1}(\dot{B}_{2,2}^{\beta+1})}
+ T P_{+} \|a\|_{\tilde{L}_{T}^{\infty}(\dot{B}_{2,2}^{\beta})} \big).
\end{split}
\end{align}
Then, we estimate the term which is not appeared in the Navier-Stokes system as follows
\begin{align}\label{diff ns 1}
\begin{split}
\left\|\frac{1}{\rho}B\cdot\nabla B \right\|_{\tilde{L}_{T}^{1}(\dot{B}_{2,2}^{\beta-1})} \leq & C (1+\|a\|_{\tilde{L}_{T}^{\infty}(\dot{B}_{p}^{3/p})})\|B\cdot\nabla B\|_{\tilde{L}_{T}^{1}(\dot{B}_{2,2}^{\beta-1})}  \\
\leq & C (1+\|a\|_{\tilde{L}_{T}^{\infty}(\dot{B}_{p}^{3/p})}) \|B\|_{\tilde{L}_{T}^{2}(\dot{B}_{p}^{3/p})} \|B\|_{\tilde{L}_{T}^{2}(\dot{B}_{2,2}^{\beta})}.
\end{split}
\end{align}
Plugging estimates (\ref{same ns 1}), (\ref{same ns 2}) and (\ref{diff ns 1}) into (\ref{estimate weight 1}), we obtain
\begin{align}\label{estimate weight 2}
\begin{split}
\|u\|_{\tilde{L}_{T}^{1}(\dot{B}_{2,2}^{\beta+1})} & + \|u\|_{\tilde{L}_{T}^{2}(\dot{B}_{2,2}^{\beta})} \leq C \|u_{0}\|_{\dot{B}_{2,2}^{\beta-1}}
+ C \|u\|_{\tilde{L}_{T}^{2}(\dot{B}_{p}^{3/p})}\|u\|_{\tilde{L}_{T}^{2}(\dot{B}_{2,2}^{\beta})}    \\
& + C \|a\|_{\tilde{L}_{T}^{\infty}(\dot{B}_{p}^{3/p}(\omega))}\|u\|_{\tilde{L}_{T}^{1}(\dot{B}_{2,2}^{\beta+1})} + C T P_{+}\|a\|_{\tilde{L}_{T}^{\infty}(\dot{B}_{2,2}^{\beta})}    \\
& + C \big(1+ \|a\|_{\tilde{L}_{T}^{\infty}(\dot{B}_{p}^{3/p})} \big) \|B\|_{\tilde{L}_{T}^{2}(\dot{B}_{p}^{3/p})} \|B\|_{\tilde{L}_{T}^{2}(\dot{B}_{2,2}^{\beta})}.
\end{split}
\end{align}
Taking $L^{2}$ energy estimates for equation (\ref{localequationfor u}), we obtain
\begin{align}\label{estimate 1}
\begin{split}
\|u\|_{\tilde{L}_{T}^{\infty}(\dot{B}_{2,2}^{\beta-1})} \leq & C \big( \|u_{0}\|_{\dot{B}_{2,2}^{\beta-1}} + \|G\|_{\tilde{L}_{T}^{1}(\dot{B}_{2,2}^{\beta-1})} \big) \\
& + C \left\| 2^{j\beta} \left( \|[\Delta_{j},\bar{\mu}]\nabla u\|_{L_{T}^{1}(L^{2})} + \|[\Delta_{j},\bar{\lambda}]\nabla u\|_{L_{T}^{1}(L^{2})} \right) \right\|_{\ell^{2}} \\
& + C \left\| \frac{1}{\rho} B\cdot\nabla B \right\|_{\tilde{L}_{T}^{1}(\dot{B}_{2,2}^{\beta-1})}
+ C \left\| \frac{1}{\rho}\nabla (|B|^{2}) \right\|_{\tilde{L}_{T}^{1}(\dot{B}_{2,2}^{\beta-1})}.
\end{split}
\end{align}
The first two terms can be estimated as in the Navier-Stokes system, we only give the results as follows
\begin{align}\label{same ns 3}
\begin{split}
& \left\| 2^{j\beta}\left( \|[\Delta_{j},\bar{\mu}]\nabla u\|_{L_{T}^{1}(L^{2})} + \|[\Delta_{j},\bar{\lambda}]\nabla u\|_{L_{T}^{1}(L^{2})} \right) \right\|_{\ell^{2}}  \\
& \quad
\leq C \|a\|_{\tilde{L}_{T}^{\infty}(\dot{B}_{p}^{3/p})} \|u\|_{\tilde{L}_{T}^{1}(\dot{B}_{2,2}^{\beta + 1})},
\end{split}
\end{align}
\begin{align}\label{same ns 4}
\begin{split}
\|G\|_{\tilde{L}_{T}^{1}(\dot{B}_{2,2}^{\beta-1})} \leq & C \big( \|u\|_{\tilde{L}_{T}^{2}(\dot{B}_{p}^{3/p})}\|u\|_{\tilde{L}_{T}^{2}(\dot{B}_{2,2}^{\beta})}
+ \|a\|_{\tilde{L}_{T}^{\infty}(\dot{B}_{p}^{3/p})}\|u\|_{\tilde{L}_{T}^{1}(\dot{B}_{2,2}^{\beta+1})}   \\
& + T P_{+} \|a\|_{\tilde{L}_{T}^{\infty}(\dot{B}_{2,2}^{\beta})}\big).
\end{split}
\end{align}
For the terms not appeared in the Navier-Stokes system, we can estimate as (\ref{diff ns 1}).
Plugging (\ref{same ns 3}), (\ref{same ns 4}) and (\ref{diff ns 1}) into (\ref{estimate 1}), we have
\begin{align}\label{estimate 2}
\begin{split}
\|u\|_{\tilde{L}_{T}^{\infty}(\dot{B}_{2,2}^{\beta-1})} \leq & C \|u_{0}\|_{\dot{B}_{2,2}^{\beta-1}}
+ C \|a\|_{\tilde{L}_{T}^{\infty}(\dot{B}_{p}^{3/p})}\|u\|_{\tilde{L}_{T}^{1}(\dot{B}_{2,2}^{\beta+1})} \\
& + C \|u\|_{\tilde{L}_{T}^{2}(\dot{B}_{p}^{3/p})} \|u\|_{\tilde{L}_{T}^{2}(\dot{B}_{2,2}^{\beta})} + C T P_{+} \|a\|_{\tilde{L}_{T}^{\infty}(\dot{B}_{2,2}^{\beta})} \\
& + C (1+ \|a\|_{\tilde{L}_{T}^{\infty}(\dot{B}_{p}^{3/p})} )\|B\|_{\tilde{L}_{T}^{2}(\dot{B}_{p}^{3/p})} \|B\|_{\tilde{L}_{T}^{2}(\dot{B}_{2,2}^{\beta})}.
\end{split}
\end{align}
$\mathbf{Step \,\, 3:Estimate\,\,\, for\,\,\, the\,\,\, Magnetic\,\,\, Field\,\,\, Equation.}$
Applying the operator $\Delta_{j}$ to the third equation of system (\ref{origianl MHD}), we obtain
\begin{align*}
\partial_{t}\Delta_{j}B - \nu \Delta\Delta_{j}B = - \Delta_{j}\mathrm{div}(Bu^{T} - uB^{T}).
\end{align*}
Performing $L^{2}$ energy estimates, we could have
\begin{align}\label{magnetic besov}
\begin{split}
\|B\|_{\tilde{L}_{T}^{\infty}(\dot{B}_{2,2}^{\beta-1})} + \|B\|_{\tilde{L}_{T}^{1}(\dot{B}_{2,2}^{\beta+1})}
\leq & C \|B_{0}\|_{\dot{B}_{2,2}^{\beta-1}}    \\
& + C \|\mathrm{div}(Bu^{T} - uB^{T})\|_{\tilde{L}_{T}^{1}(\dot{B}_{2,2}^{\beta-1})}.
\end{split}
\end{align}
Since
\begin{align}\label{magnetic 1}
\begin{split}
\|u\cdot\nabla B\|_{\tilde{L}_{T}^{1}(\dot{B}_{2,2}^{\beta-1})} \leq & C \|u\|_{\tilde{L}_{T}^{2}(\dot{B}_{p}^{3/p})} \|\nabla B\|_{\tilde{L}_{T}^{2}(\dot{B}_{2,2}^{\beta-1})}   \\
\leq & C \|u\|_{\tilde{L}_{T}^{2}(\dot{B}_{p}^{3/p})} \|B\|_{\tilde{L}_{T}^{2}(\dot{B}_{2,2}^{\beta})},
\end{split}
\end{align}
and
\begin{align}\label{magnetic 2}
\begin{split}
\|B\cdot\nabla u\|_{\tilde{L}_{T}^{1}(\dot{B}_{2,2}^{\beta-1})} \leq C \|B\|_{\tilde{L}_{T}^{2}(\dot{B}_{p}^{3/p})} \|u\|_{\tilde{L}_{T}^{2}(\dot{B}_{2,2}^{\beta})},
\end{split}
\end{align}
we know that
\begin{align}\label{estimate B}
\begin{split}
\|B\|_{\tilde{L}_{T}^{\infty}(\dot{B}_{2,2}^{\beta-1})} + \|B\|_{\tilde{L}_{T}^{1}(\dot{B}_{2,2}^{\beta+1})} \leq &
C \|B_{0}\|_{\dot{B}_{2,2}^{\beta-1}} + C \|u\|_{\tilde{L}_{T}^{2}(\dot{B}_{p}^{3/p})} \|B\|_{\tilde{L}_{T}^{2}(\dot{B}_{2,2}^{\beta})}   \\
& + C \|B\|_{\tilde{L}_{T}^{2}(\dot{B}_{p}^{3/p})} \|u\|_{\tilde{L}_{T}^{2}(\dot{B}_{2,2}^{\beta})}.
\end{split}
\end{align}

Combing estimates (\ref{estimate weight 2}) and (\ref{estimate B}), we will get
\begin{align}\label{temp estimate 1}
\begin{split}
& \|u\|_{\tilde{L}_{T}^{1}(\dot{B}_{2,2}^{\beta+1})} + \|u\|_{\tilde{L}_{T}^{2}(\dot{B}_{2,2}^{\beta})} + \|B\|_{\tilde{L}_{T}^{1}(\dot{B}_{2,2}^{\beta+1})}    \\
& \quad\quad
\leq C \big( \|u_{0}\|_{\dot{B}_{2,2}^{\beta-1}} + \|B_{0}\|_{\dot{B}_{2,2}^{\beta-1}} + T P_{+} \|a\|_{\tilde{L}_{T}^{\infty}(\dot{B}_{2,2}^{\beta})} \big).
\end{split}
\end{align}
Combining estimates (\ref{estimate 2}) and (\ref{estimate B}), we will have
\begin{align}\label{temp estimate 2}
\begin{split}
& \|u\|_{\tilde{L}_{T}^{\infty}(\dot{B}_{2,2}^{\beta-1})} + \|B\|_{\tilde{L}_{T}^{\infty}(\dot{B}_{2,2}^{\beta-1})}   \\
& \quad\quad
\leq C \big( \|u_{0}\|_{\dot{B}_{2,2}^{\beta-1}} + \|B_{0}\|_{\dot{B}_{2,2}^{\beta-1}} + T P_{+} \|a\|_{\tilde{L}_{T}^{\infty}(\dot{B}_{2,2}^{\beta})} \big) \\
& \quad\quad\quad
+ C \|a_{0}\|_{\dot{B}_{p}^{3/p}} \|u\|_{\tilde{L}_{T}^{1}(\dot{B}_{2,2}^{\beta+1})}.
\end{split}
\end{align}
At this stage, combining (\ref{temp estimate 1}) and (\ref{temp estimate 2}), we obtain
\begin{align}\label{temp estimate 3}
\begin{split}
& \|u\|_{\tilde{L}_{T}^{\infty}(\dot{B}_{2,2}^{\beta-1})} + \|B\|_{\tilde{L}_{T}^{\infty}(\dot{B}_{2,2}^{\beta-1})}
+ \|u\|_{\tilde{L}_{T}^{1}(\dot{B}_{2,2}^{\beta+1})} + \|B\|_{\tilde{L}_{T}^{1}(\dot{B}_{2,2}^{\beta + 1})} \\
& \quad\quad
\leq C \big(1+ \|a_{0}\|_{\dot{B}_{p}^{3/p}}\big) \big( \|u_{0}\|_{\dot{B}_{2,2}^{\beta-1}} + \|B_{0}\|_{\dot{B}_{2,2}^{\beta-1}}
+ T P_{+} \|a\|_{\tilde{L}_{T}^{\infty}(\dot{B}_{2,2}^{\beta})} \big).
\end{split}
\end{align}
At last, combining (\ref{transport estimate 1}) and (\ref{temp estimate 3}), we finally get the desired results (\ref{low pro}).
\end{proof}

In the last part of this section, we prove the solution propagate the regularity of the initial data in Sobolev space with high regularity.
\begin{proposition}\label{highp}
Assume that $(\rho, u, B)$ is a solution of system (\ref{origianl MHD}) on $[0,T]$, which satisfies $\rho \geq c_{0}$,
\begin{align}\label{condition high}
a \in \tilde{L}_{T}^{\infty}(\dot{B}_{p}^{3/p}), \quad u,B \in \tilde{L}_{T}^{\infty}(\dot{B}_{p}^{3/p}) \cap \tilde{L}_{T}^{1}(\dot{B}_{p}^{3/p+1}),
\end{align}
where $a = \rho - 1$.
If $(a_{0}, u_{0}, B_{0}) \in H^{s} \times H^{s-1} \times H^{s-1}$ for $s \geq 3$, then we have
\begin{align}\label{high result}
\begin{split}
a \in \tilde{L}_{T}^{\infty}(H^{s}), \quad u, B \in \tilde{L}_{T}^{\infty}(H^{s-1}) \cap \tilde{L}_{T}^{1}(H^{s+1}).
\end{split}
\end{align}
\end{proposition}
\begin{proof}
Considering (\ref{condition high}), we can divide the time interval $[0,T]$ into finitely many small intervals $[T_{i},T_{i+1}]$
with $i = 0,1,\ldots,N$ such that
\begin{align}\label{small assumption}
\begin{split}
& T_{i+1} - T_{i} \leq \epsilon, \quad \|a\|_{\tilde{L}^{\infty}(T_{i},T_{i+1};\dot{B}_{p}^{3/p}(\omega^{i}))} \leq \epsilon,  \\
& \|(u,B)\|_{\tilde{L}^{1}(T_{i},T_{i+1};\dot{B}_{p}^{3/p+1})} + \|(u,B)\|_{\tilde{L}^{2}(T_{i},T_{i+1};\dot{B}_{p}^{3/p})} \leq \epsilon,
\end{split}
\end{align}
for some $\epsilon$ small enough. Here $\omega^{i} = \{\omega_{k}^{i}\}$ stands for
\begin{align*}
\omega_{k}^{i} = \sum_{\ell \geq k} 2^{k-\ell} (1-e^{-c2^{2\ell (T_{i+1}-T_{i})}})^{\frac{1}{2}}.
\end{align*}
Denote $\tilde{L}_{T_{i}}^{q}(\dot{B}_{p,r}^{s}) := \tilde{L}^{q}(T_{i},T_{i+1};\dot{B}_{p,r}^{s})$.
For the density, the estimate is similar to the Navier-Stokes equations, so we omit the details
\begin{align}\label{density high}
\|a\|_{\tilde{L}_{T_{i}}^{\infty}(\dot{B}_{2,2}^{s})} \leq C \big( \|a(T_{i})\|_{\dot{B}_{2,2}^{s}} + \|u\|_{\tilde{L}_{T_{i}}^{1}(\dot{B}_{2,2}^{s+1})} \big).
\end{align}
For the velocity field, as in the proof of Proposition \ref{lowp}, we have
\begin{align}\label{high u}
\begin{split}
& \|u\|_{\tilde{L}_{T_{i}}^{1}(\dot{B}_{2,2}^{s+1})} + \|u\|_{\tilde{L}_{T_{i}}^{2}(\dot{B}_{2,2}^{s})}   \\
& \quad\quad
\leq C \big( \|u(T_{i})\|_{\dot{B}_{2,2}^{s-1}} + \|G\|_{\tilde{L}_{T_{i}}^{1}(\dot{B}_{2,2}^{s-1}(\omega^{i})} \big)   \\
& \quad\quad\quad
+ C \left\| 2^{js} \omega_{j}(t) \left( \|[\Delta_{j},\bar{\lambda}]\nabla u\|_{L_{T_{i}^{1}(L^{2})}}
+ \|[\Delta_{j},\bar{\mu}]\nabla u\|_{L_{T_{i}}^{1}(L^{2})} \right) \right\|_{\ell^{2}} \\
& \quad\quad\quad
+ C \left\| \frac{1}{\rho}B\cdot\nabla B \right\|_{\tilde{L}_{T_{i}}^{1}(\dot{B}_{2,2}^{s-1})}
+ C \left\| \frac{1}{\rho}\nabla (|B^{2}|) \right\|_{\tilde{L}_{T_{i}}^{1}(\dot{B}_{2,2}^{s-1})}.
\end{split}
\end{align}
Similar to the Navier-Stokes system, we have
\begin{align}\label{same ns 5}
\begin{split}
& \|G\|_{\tilde{L}_{T_{i}}^{1}(\dot{B}_{2,2}^{s-1}(\omega^{i}))} \leq C \big( \|u\|_{\tilde{L}_{T_{i}}^{2}(\dot{B}_{p}^{3/p})} \|u\|_{\tilde{L}_{T_{i}}^{2}(\dot{B}_{2,2}^{s})}  \\
& \quad\quad\quad\quad
+ \|a\|_{\tilde{L}_{T_{i}}^{\infty}(\dot{B}_{p}^{3/p}(\omega^{i}))}\|u\|_{\tilde{L}_{T_{i}}^{1}(\dot{B}_{2,2}^{s+1})} + \|a\|_{\tilde{L}_{T_{i}}^{\infty}(\dot{B}_{2,2}^{s})}
\|u\|_{\tilde{L}_{T_{i}}^{1}(\dot{B}_{p}^{3/p+1})}    \\
& \quad\quad\quad\quad
+ (T_{i+1} - T_{i})\|a\|_{\tilde{L}_{T_{i}}^{\infty}(\dot{B}_{2,2}^{s})}
\big),
\end{split}
\end{align}
and
\begin{align}\label{same ns 6}
\begin{split}
& \left\| 2^{js} \omega^{i}_{j}(t) \left( \|[\Delta_{j},\bar{\lambda}]\nabla u\|_{L_{T_{i}}^{1}(L^{2})}
+ \|[\Delta_{j},\bar{\mu}]\nabla u\|_{L_{T_{i}}^{1}(L^{2})} \right) \right\|_{\ell^{2}} \\
& \quad\quad\quad
\leq C \big( \|a\|_{\tilde{L}_{T_{i}}^{\infty}(\dot{B}_{p}^{3/p}(\omega^{i}))} \|u\|_{\tilde{L}_{T_{i}}^{1}(\dot{B}_{2,2}^{s+1})}
+ \|a\|_{\tilde{L}^{\infty}_{T_{i}}(\dot{B}_{2,2}^{s})} \|u\|_{\tilde{L}^{1}_{T_{i}}(\dot{B}_{p}^{3/p+1})} \big).
\end{split}
\end{align}
For the last two terms in \ref{high u}, all can be estimated as follows
\begin{align}\label{estimate high ub}
\begin{split}
& \left\| \frac{1}{\rho}B\cdot\nabla B \right\|_{\tilde{L}^{1}_{T_{i}}(\dot{B}_{2,2}^{s-1})} \leq C \|B\cdot\nabla B\|_{\tilde{L}^{1}_{T_{i}}(\dot{B}_{2,2}^{s-1})}
+ C \|\frac{a}{1+a} B\cdot\nabla B\|_{\tilde{L}^{1}_{T_{i}}(\dot{B}_{2,2}^{s-1})} \\
& \,\,
\leq C \|B\|_{\tilde{L}^{2}_{T_{i}}(\dot{B}_{p}^{3/p})} \|B\|_{\tilde{L}^{2}_{T_{i}}(\dot{B}_{2,2}^{s})}
+ C \|a\|_{\tilde{L}^{\infty}_{T_{i}}(\dot{B}_{p}^{3/p})} \|B\cdot\nabla B\|_{\tilde{L}^{1}_{T_{i}}(\dot{B}_{2,2}^{s-1})}    \\
& \,\quad
+ C \|B\cdot\nabla B\|_{\tilde{L}^{1}_{T_{i}}(\dot{B}_{p}^{3/p-1})}\|a\|_{\tilde{L}^{\infty}_{T_{i}}(\dot{B}_{2,2}^{s})}    \\
& \,\,
\leq  C \|B\|_{\tilde{L}^{2}_{T_{i}}(\dot{B}_{p}^{3/p})} \|B\|_{\tilde{L}^{2}_{T_{i}}(\dot{B}_{2,2}^{s})}
+ C \|a\|_{\tilde{L}^{\infty}_{T_{i}}(\dot{B}_{p}^{3/p})} \|B\|_{\tilde{L}^{2}_{T_{i}}(\dot{B}_{p}^{3/p})} \|B\|_{\tilde{L}^{2}_{T_{i}}(\dot{B}_{2,2}^{s})}   \\
& \,\,\quad
+ C \|B\|_{\tilde{L}^{2}_{T_{i}}(\dot{B}_{p}^{3/p})} \|B\|_{\tilde{L}^{2}_{T_{i}}(\dot{B}_{p}^{3/p})} \|a\|_{\tilde{L}^{\infty}_{T_{i}}(\dot{B}_{2,2}^{s})}.
\end{split}
\end{align}
Plugging (\ref{same ns 5}), (\ref{same ns 6}) and (\ref{estimate high ub}) into (\ref{high u}), we will obtain
\begin{align}\label{high u weight}
\begin{split}
& \|u\|_{\tilde{L}^{1}_{T_{i}}(\dot{B}_{2,2}^{s+1})} + \|u\|_{\tilde{L}^{2}_{T_{i}}(\dot{B}_{2,2}^{s})} \\
& \quad
\leq C \big( \|u(T_{i})\|_{\dot{B}_{2,2}^{s-1}} + \|a\|_{\tilde{L}^{\infty}_{T_{i}}(\dot{B}_{2,2}^{s})}(\|u\|_{\tilde{L}^{1}_{T_{i}}(\dot{B}_{p}^{3/p})} + (T_{i+1}-T_{i})) \big) \\
& \quad\quad
+ C \|B\|_{\tilde{L}^{2}_{T_{i}}(\dot{B}_{p}^{3/p})}\|B\|_{\tilde{L}^{2}_{T_{i}}(\dot{B}_{2,2}^{s})}
+ C \|a_{0}\|_{\dot{B}_{p}^{3/p}} \|B\|_{\tilde{L}^{2}_{T_{i}}(\dot{B}_{p}^{3/p})} \|B\|_{\tilde{L}^{2}_{T_{i}}(\dot{B}_{2,2}^{s})}   \\
& \quad\quad
+ C \|B\|_{\tilde{L}^{2}_{T_{i}}(\dot{B}_{p}^{3/p})}^{2} \|a\|_{\tilde{L}^{\infty}_{T_{i}}(\dot{B}_{2,2}^{s})}.
\end{split}
\end{align}
For the equation of magnetic field, using similar methods as in Proposition \ref{lowp}, we obtain
\begin{align}\label{estimate B high}
\begin{split}
\|B\|_{\tilde{L}^{\infty}_{T_{i}}(\dot{B}_{2,2}^{s-1})} + \|B\|_{\tilde{L}^{1}_{T_{i}}(\dot{B}_{2,2}^{s+1})}
\leq & C \|B_{0}\|_{\dot{B}_{2,2}^{s-1}} + C \|u\cdot\nabla B\|_{\tilde{L}^{1}_{T_{i}}(\dot{B}_{2,2}^{s-1})}    \\
& + C \|B\cdot\nabla u\|_{\tilde{L}^{1}_{T_{i}}(\dot{B}_{2,2}^{s-1})}.
\end{split}
\end{align}
Since
\begin{align*}
\|u\cdot\nabla B\|_{\tilde{L}^{1}_{T_{i}}(\dot{B}_{2,2}^{s-1})} \leq & C (\|u\|_{\tilde{L}^{2}_{T_{i}}(\dot{B}_{p}^{3/p})}\|\nabla B\|_{\tilde{L}^{2}_{T_{i}}(\dot{B}_{2,2}^{s-1})}
+ \|\nabla B\|_{\tilde{L}^{2}_{T_{i}}(\dot{B}_{p}^{3/p-1})} \|u\|_{\tilde{L}^{2}_{T_{i}}(\dot{B}_{2,2}^{s})})  \\
\leq & C \|u\|_{\tilde{L}^{2}_{T_{i}}(\dot{B}_{p}^{3/p})}\|B\|_{\tilde{L}^{2}_{T_{i}}(\dot{B}_{2,2}^{s})}
+ C \|B\|_{\tilde{L}^{2}_{T_{i}}(\dot{B}_{p}^{3/p})} \|u\|_{\tilde{L}^{2}_{T_{i}}(\dot{B}_{2,2}^{s})},
\end{align*}
and
\begin{align*}
\|B\cdot\nabla u\|_{\tilde{L}^{1}_{T_{i}}(\dot{B}_{2,2}^{s-1})} \leq C ( \|B\|_{\tilde{L}^{2}_{T_{i}}(\dot{B}_{p}^{3/p})} \|u\|_{\tilde{L}^{2}_{T_{i}}(\dot{B}_{2,2}^{s})}
+ \|u\|_{\tilde{L}^{2}_{T_{i}}(\dot{B}_{p}^{3/p})} \|B\|_{\tilde{L}^{2}_{T_{i}}(\dot{B}_{2,2}^{s})} ),
\end{align*}
we know that
\begin{align}\label{estimate B high 2}
\begin{split}
& \|B\|_{\tilde{L}^{\infty}_{T_{i}}(\dot{B}_{2,2}^{s-1})} + \|B\|_{\tilde{L}^{1}_{T_{i}}(\dot{B}_{2,2}^{s+1})} + \|B\|_{\tilde{L}^{2}_{T_{i}}(\dot{B}_{2,2}^{s})} \\
& \quad
\leq C\|B(T_{i})\|_{\dot{B}_{2,2}^{s-1}} + C \|u\|_{\tilde{L}^{2}_{T_{i}}(\dot{B}_{p}^{3/p})} \|B\|_{\tilde{L}^{2}_{T_{i}}(\dot{B}_{2,2}^{s})}    \\
& \quad\quad
+ C \|B\|_{\tilde{L}^{2}_{T_{i}}(\dot{B}_{p}^{3/p})} \|u\|_{\tilde{L}^{2}_{T_{i}}(\dot{B}_{2,2}^{s})}.
\end{split}
\end{align}
Combining (\ref{density high}), (\ref{high u weight}) and (\ref{estimate B high 2}), we have
\begin{align}\label{estimate high 1}
\begin{split}
& \|a\|_{\tilde{L}^{\infty}_{T_{i}}(\dot{B}_{2,2}^{s})} + \|u\|_{\tilde{L}^{1}_{T_{i}}(\dot{B}_{2,2}^{s+1})} + \|u\|_{\tilde{L}^{2}_{T_{i}}(\dot{B}_{2,2}^{s})}
+ \|B\|_{\tilde{L}^{1}_{T_{i}}(\dot{B}_{2,2}^{s+1})} + \|B\|_{\tilde{L}^{2}_{T_{i}}(\dot{B}_{2,2}^{s})} \\
& \quad
\leq C ( \|a(T_{i})\|_{\dot{B}_{2,2}^{s}} + \|u(T_{i})\|_{\dot{B}_{2,2}^{s-1}} + \|B(T_{i})\|_{\dot{B}_{2,2}^{s-1}} ).
\end{split}
\end{align}
Estimate velocity as in (\ref{estimate 1}) with regularity index to be $s-1$ and combine the estimate
about magnetic field $B$, we finally get
\begin{align}\label{estimate high 2}
\begin{split}
\|u\|_{\tilde{L}^{\infty}_{T_{i}}(\dot{B}_{2,2}^{s-1})} + \|B\|_{\tilde{L}^{\infty}_{T_{i}}(\dot{B}_{2,2}^{s-1})} \leq &
C (1+ \|a\|_{\tilde{L}^{\infty}_{T_{i}}(\dot{B}_{p}^{3/p})})( \|a(T_{i})\|_{\dot{B}_{2,2}^{s}}  \\
& + \|u(T_{i})\|_{\dot{B}_{2,2}^{s-1}} + \|B(T_{i})\|_{\dot{B}_{2,2}^{s-1}} ).
\end{split}
\end{align}
Combining (\ref{estimate high 1}) and (\ref{estimate high 2}), we arrive at
\begin{align}\label{final high}
\begin{split}
& \|a\|_{\tilde{L}^{\infty}_{T_{i}}(\dot{B}_{2,2}^{s})} + \|(u,B)\|_{\tilde{L}^{\infty}_{T_{i}}(\dot{B}_{2,2}^{s-1})} + \|(u,B)\|_{\tilde{L}^{1}_{T_{i}}(\dot{B}_{2,2}^{s+1})}  \\
& \quad
\leq C (1+ \|a\|_{\tilde{L}^{\infty}_{T_{i}}(\dot{B}_{p}^{3/p})})( \|a(T_{i})\|_{\dot{B}_{2,2}^{s}}
+ \|u(T_{i})\|_{\dot{B}_{2,2}^{s-1}} + \|B(T_{i})\|_{\dot{B}_{2,2}^{s-1}} ).
\end{split}
\end{align}
By induction, we know that
\begin{align}\label{final in high}
\begin{split}
& \|a\|_{\tilde{L}^{\infty}_{T}(\dot{B}_{2,2}^{s})} + \|(u,B)\|_{\tilde{L}^{\infty}_{T}(\dot{B}_{2,2}^{s-1})} + \|(u,B)\|_{\tilde{L}^{1}_{T}(\dot{B}_{2,2}^{s+1})}  \\
& \quad
\leq C (1+\|a\|_{\tilde{L}^{\infty}_{T}(\dot{B}_{p}^{3/p})})^{N+1}
( \|a(T_{i})\|_{\dot{B}_{2,2}^{s}} + \|u(T_{i})\|_{\dot{B}_{2,2}^{s-1}} + \|B(T_{i})\|_{\dot{B}_{2,2}^{s-1}} ).
\end{split}
\end{align}
Hence, we complete the proof.
\end{proof}


\section{Hoff's Energy Method}

In two papers \cite{hoff1995,hoff1997}, D. Hoff construct a global weak solution for Navier-Stokes equations with discontinuous initial data with
small energy. In the paper \cite{mhd2012hoff}, A. Suen and D. Hoff generalize the results for Navier-Stokes equations to
compressible MHD system. Here, we use the idea to our case. Comparing to \cite{mhd2012hoff}, we remove the restriction on viscosity for
we have stronger condition on the initial data.

We set $\sigma (t) := \min (1,t)$, define
\begin{align*}
A_{1}(T) & = \sup_{0 \leq t \leq T} (\|\nabla u(t)\|^{2}_{L^{2}} + \|\nabla B(t)\|^{2}_{L^{2}})
+ \int_{0}^{T}\int_{\mathbb{R}^{3}} \rho |\dot{u}|^{2} + |B_{t}|^{2} \, dx dt,   \\
A_{2}(T) & = \sup_{0 \leq t \leq T}\left( \sigma(t)\int_{\mathbb{R}^{3}} \rho |\dot{u}|^{2} \, dx + \sigma(t)\int_{\mathbb{R}^{3}} |B_{t}|^{2} \, dx \right)    \\
& \quad\quad\quad\quad\quad
+ \int_{0}^{T}\int_{\mathbb{R}^{3}} \sigma(t) |\nabla\dot{u}|^{2} \, dx dt + \int_{0}^{T}\int_{\mathbb{R}^{3}} \sigma(t) |\nabla B_{t}|^{2} \, dx dt,
\end{align*}
and
\begin{align*}
E(T) & = \int_{\mathbb{R}^{3}} \sigma(t) \left( |\nabla B|^{2} |B|^{2} + |\nabla B|^{2} |u|^{2} + |\nabla u|^{2} |B|^{2} \right) \, dx,    \\
H(T) & = \int_{0}^{T}\int_{\mathbb{R}^{3}} |\nabla B|^{3} \, dx dt + \int_{0}^{T}\int_{\mathbb{R}^{3}} |\nabla u|^{3} \, dx dt      \\
& \quad\quad\quad\quad\quad
+ \int_{0}^{T}\int_{\mathbb{R}^{3}} \sigma(t) |\nabla B|^{4} \, dx dt + \int_{0}^{T}\int_{\mathbb{R}^{3}} \sigma(t) |\nabla u|^{4} \, dx dt.
\end{align*}
In the following sections, we denote
\begin{align*}
\dot{f} = f_{t} + u\cdot\nabla f, \quad \mu' = \lambda + \mu,
\end{align*}
and
\begin{align*}
C_{0} = \|\rho_{0} - \bar{\rho}\|^{2}_{L^{2}} + \|u_{0}\|^{2}_{H^{2}} + \|B_{0}\|^{2}_{H^{1}}.
\end{align*}
Throughout this section, we denote by $C$ a constant depending only on $\lambda, \mu, c_{0}, \bar{\rho}, P_{+}$ and $P_{-}^{-1}$.

\begin{theorem}\label{hofftheorem}
Let $(\rho, u, B)$ be a solution of system (\ref{origianl MHD}) satisfying
\begin{align*}
\rho - \bar{\rho} \in C([0,T]; H^{2}), \quad u,B \in C([0,T]; H^{2}) \cap L^{2}(0,T; H^{3}).
\end{align*}
Then, there exists a constant $\epsilon_{0}$ depending only on $\nu, \lambda, \mu, c_{0}, \bar{\rho}, P_{+}, P_{-}^{-1}$ such that if
the initial data $(\rho_{0}, u_{0}, B_{0})$ satisfies
\begin{align*}
& c_{0} \leq \rho_{0}(x) \leq c_{0}^{-1}, \quad x \in \mathbb{R}^{3},     \\
& \|\rho - \bar{\rho}\|^{2}_{L^{2}} + \|u_{0}\|^{2}_{H^{1}} + \|B_{0}\|^{2}_{H^{1}} \leq \epsilon_{0},
\end{align*}
then we have
\begin{align*}
& \frac{c_{0}}{2} \leq \rho(t,x) \leq 2 c_{0}^{-1}, \quad (t,x) \in [0,T] \times \mathbb{R}^{3},  \\
& A(T) := A_{1}(T) + A_{2}(T) \leq \epsilon_{0}^{\frac{1}{2}}.
\end{align*}
\end{theorem}
\begin{proof}
Considering the assumption, there exists a $0 < T_{0} \leq T$ such that the solution $(\rho, u, B)$ satisfies
\begin{align*}
& \frac{c_{0}}{2} \leq \rho(t,x) \leq 2 c_{0}^{-1}, \quad (t,x) \in [0,T_{0}] \times \mathbb{R}^{3},  \\
& A_{1}(T_{0}) + A_{2}(T_{0}) \leq \epsilon_{0}^{\frac{1}{2}}.
\end{align*}
Without loss of generality, we assume that $T_{0}$ is a maximal time so that the above inequalities hold.
In the following, we will give a refined estimates on $[0,T]$ for the solution.
Due to the proof is too long, we divide the proof into several lemmas.
\begin{lemma}\label{hofflemma1}($L^{2}$ energy estimate)
\begin{align*}
\int_{\mathbb{R}^{3}} |\rho - \bar{\rho}|^{2} + \rho|u|^{2} & + |B|^{2} \, dx
+ \int_{0}^{T}\int_{\mathbb{R}^{3}} |\nabla u|^{2} + |\nabla B|^{2} \, dx dt \leq C C_{0}.
\end{align*}
\end{lemma}
The proof of this lemma is exactly the same as Lemma 2.2 in \cite{mhd2011Chengchun}, so we omit it.
\begin{lemma}\label{hofflemma2}($H^{1}$ energy estimate)
\begin{align*}
A_{1}(T_{0}) \leq & C C_{0} + C \int_{0}^{T_{0}}\int_{\mathbb{R}^{3}} |\nabla B|^{2}|B|^{2} + |\nabla B|^{2}|u|^{2} + |\nabla u|^{2}|B|^{2} \, dx dt    \\
& + C \int_{0}^{T_{0}}\int_{\mathbb{R}^{3}} |\nabla u|^{3} \, dx dt.
\end{align*}
\end{lemma}
\begin{proof}
Multiply the second equation in (\ref{origianl MHD}) by $\dot{u}$ and integrate over $\mathbb{R}^{3}$, we will obtain
\begin{align}\label{hiesforu}
\begin{split}
\int_{\mathbb{R}^{3}} \rho |\dot{u}|^{2} \, dx = & \int_{\mathbb{R}^{3}} ( -\dot{u}\cdot\nabla P + \mu \Delta u \, \dot{u}
+ \mu' \nabla\mathrm{div}u \, \dot{u}   \\
& - \nabla (\frac{1}{2}|B|^{2})\dot{u} + \mathrm{div}(B \, B^{T}) \, \dot{u} ) \, dx.
\end{split}
\end{align}
By the continuity equation, we have
\begin{align}
\partial P + \mathrm{div}(uP) = \mathrm{div}u \, (P - P'\rho).
\end{align}
There are some terms can be estimated as in the Navier-Stokes equations, so we omit the details and only give the results as follows
\begin{align}\label{h1ns1}
\int_{\mathbb{R}^{3}}-\dot{u}\cdot\nabla P\, dx \leq & \partial_{t}\int_{\mathbb{R}^{3}}\mathrm{div}u\, (P-P(\bar{\rho}))\, dx
+ C \|\nabla u\|^{2}_{L^{2}},
\end{align}
\begin{align}\label{h1ns2}
\mu \int_{\mathbb{R}^{3}} \Delta u\, \dot{u} \, dx \leq & - \frac{\mu}{2} \partial_{t}\int_{\mathbb{R}^{3}} |\nabla u|^{2} \, dx
+ C \int_{\mathbb{R}^{3}} |\nabla u|^{3} \, dx,
\end{align}
\begin{align}\label{h1ns3}
\mu' \int_{\mathbb{R}^{3}} \nabla\mathrm{div} u \, \dot{u} \, dx \leq - \frac{\lambda}{2} \partial_{t} \int_{\mathbb{R}^{3}} |\mathrm{div}u|^{2} \, dx
+ C \int_{\mathbb{R}^{3}} |\nabla u|^{3} \, dx,
\end{align}
and
\begin{align}\label{h1ns4}
\begin{split}
\int_{\mathbb{R}^{3}} \mathrm{div}u \, (P - P(\bar{\rho})) \, dx & \leq \|\mathrm{div}u\|_{L^{2}} \|P(\rho) - P(\bar{\rho})\|_{L^{2}} \\
& \leq C C_{0}^{\frac{1}{2}} \|\mathrm{div}u\|_{L^{2}}.
\end{split}
\end{align}
The following terms are not appeared in Navier-Stokes system
\begin{align}\label{h1nn1}
\begin{split}
\int_{\mathbb{R}^{3}} - \nabla (\frac{1}{2}|B|^{2})\, \dot{u} \, dx + \int_{\mathbb{R}^{3}} \mathrm{div}(B\, B^{T}) \, \dot{u} \, dx
& \leq C \int_{\mathbb{R}^{3}} |\nabla B| |B| |\dot{u}| \, dx     \\
\leq C & \int_{\mathbb{R}^{3}} |\nabla B|^{2} |B|^{2} \, dx + \epsilon \int_{\mathbb{R}^{3}} |\dot{u}|^{2}\, dx.
\end{split}
\end{align}
Multiply the third equation in (\ref{origianl MHD}) with $B_{t}$ and integrate over $\mathbb{R}^{3}$, we have
\begin{align*}
\int_{\mathbb{R}^{3}} |B_{t}|^{2} \, dx + \int_{\mathbb{R}^{3}} \mathrm{div}(B\,u^{T} - u\, B^{T})\, B_{t}\, dx
= -\frac{\nu}{2} \partial_{t} \int_{\mathbb{R}^{3}} |\nabla B|^{2} \, dx.
\end{align*}
Integrate the above equality from $0$ to $T_{0}$, we will obtain
\begin{align}\label{h1esforB}
\begin{split}
&\frac{\nu}{2}\int_{\mathbb{R}^{3}} |\nabla B|^{2} \, dx + \int_{0}^{T_{0}}\int_{\mathbb{R}^{3}} |B_{t}|^{2} \, dx   \\
& \quad\quad\quad
\leq C C_{0} + C \int_{0}^{T_{0}}\int_{\mathbb{R}^{3}} |\nabla B|^{2}|u|^{2} + |\nabla u|^{2}|B|^{2} \, dx dt.
\end{split}
\end{align}
Combining (\ref{hiesforu}), (\ref{h1ns1}), (\ref{h1ns2}), (\ref{h1ns3}), (\ref{h1ns4}), (\ref{h1nn1}) and (\ref{h1esforB}),
we finally get the desired result.
\end{proof}
\begin{lemma}\label{hofflemma3}($H^{2}$ energy estimate)
\begin{align*}
A_{2}(T_{0}) \leq & C C_{0} + C A_{1}(T_{0}) + C \int_{0}^{T_{0}}\int_{\mathbb{R}^{3}}\sigma |\nabla u|^{4} \, dx dt  \\
& + C \int_{0}^{T}\int_{\mathbb{R}^{3}}\sigma (|u|^{2}+|B|^{2})(|B_{t}|^{2}+|\dot{u}|^{2}+|u|^{2}(|\nabla B|^{2}+|\nabla u|^{2}))\, dx dt.
\end{align*}
\end{lemma}
\begin{proof}
Take material derivative to the second equation of (\ref{origianl MHD}) to obtain
\begin{align*}
& \rho\dot{u}_{t} + \rho u\cdot\nabla \dot{u} + \nabla P_{t} + \mathrm{div}(\nabla P \otimes u) + \nabla (\frac{1}{2}|B|^{2})_{t}
+ \mathrm{div}(u\cdot\nabla(\frac{1}{2}|B|^{2}))    \\
& \quad\quad\quad\quad\quad\quad\quad\quad\quad
- \mathrm{div}(B\, B^{T})_{t} - \mathrm{div}(u\,\mathrm{div}(B\, B^{T})) = \mu \Delta u_{t} + \mu \mathrm{div} (u\, \Delta u)   \\
& \quad\quad\quad\quad\quad\quad\quad\quad\quad
+ \mu' \nabla\mathrm{div}u_{t} + \mu' \mathrm{div}(u\cdot\nabla\mathrm{div}u).
\end{align*}
Multiply $\sigma(t)\dot{u}$ on both sides of the above equality, we will have
\begin{align}\label{h2esforu}
\begin{split}
\partial_{t}\left( \frac{\sigma}{2}\int_{\mathbb{R}^{3}}\rho |\dot{u}|^{2} \, dx \right)
- \frac{1}{2}\sigma' & \int_{\mathbb{R}^{3}} \rho |\dot{u}|^{2} \, dx =
- \int_{\mathbb{R}^{3}} \sigma\dot{u}^{j} (\partial_{j}P_{t} + \mathrm{div}(\partial_{j}P \, u))\, dx   \\
= & - \int_{\mathbb{R}^{3}} \sigma \dot{u}^{j} ((\frac{1}{2}|B|^{2})_{x_{j},t} + \mathrm{div}(u\, \partial_{j}(\frac{1}{2}|B|^{2})))\, dx   \\
& + \int_{\mathbb{R}^{3}} \sigma \dot{u}^{j} (\mathrm{div}(B^{j}B)_{t} + \mathrm{div}(u\mathrm{div}(B^{j}B))) \, dx \\
& + \mu \sigma \int_{\mathbb{R}^{3}} \dot{u} \, (\Delta u_{t} + \mathrm{div}(\Delta u \otimes u))\, dx   \\
& + \mu' \sigma \int_{\mathbb{R}^{3}} \dot{u} \, (\nabla\mathrm{div}u_{t} + \mathrm{div}(\nabla\mathrm{div}u \otimes u)) \, dx.
\end{split}
\end{align}
Similar to the $H^{1}$ estimate, there are some terms same as Navier-Stokes equations and we only give the results as follows
\begin{align}\label{h2ns1}
- \int_{\mathbb{R}^{3}} \sigma \dot{u}^{j} (\partial_{j}P_{t} + \mathrm{div}(\partial_{j}P\, u)) \, dx
\leq C \sigma \|\nabla u\|_{L^{2}} \|\nabla\dot{u}\|_{L^{2}},
\end{align}
\begin{align}\label{h2ns2}
\mu\sigma \int_{\mathbb{R}^{3}} \dot{u}\, (\Delta u_{t} + \mathrm{div}(\Delta u \otimes u)) \, dx
\leq \sigma \int_{\mathbb{R}^{3}} -\frac{3\mu}{4} |\nabla\dot{u}|^{2} \, dx + \sigma C \int_{\mathbb{R}^{3}} |\nabla u|^{4} \, dx,
\end{align}
\begin{align}\label{h2ns3}
\begin{split}
\mu' \sigma \int_{\mathbb{R}^{3}} \dot{u} \, (\nabla\mathrm{div}u_{t} & + \mathrm{div}(\nabla\mathrm{div}u \otimes u)) \, dx  \\
& \leq \sigma \int_{\mathbb{R}^{3}} -\frac{\lambda}{2} |\mathrm{div}\dot{u}|^{2} + \frac{\mu}{4} |\nabla\dot{u}|^{2} + C |\nabla u|^{4} \, dx.
\end{split}
\end{align}
For the terms not appeared in the Navier-Stokes equations, we have
\begin{align}\label{h2nnnB}
\begin{split}
& - \int_{\mathbb{R}^{3}} \sigma \dot{u}^{j} (\partial_{j}(\frac{1}{2}|B|^{2})_{t} + \mathrm{div}(u\,\partial_{j}(\frac{1}{2}|B|^{2})))\, dx   \\
& \quad\quad\quad\quad\quad
\leq \epsilon \int_{\mathbb{R}^{3}} \sigma |\nabla\dot{u}|^{2} \, dx
+ C \int_{\mathbb{R}^{3}} \sigma |B|^{2} (|B_{t}|^{2} + |u|^{2}|\nabla B|^{2}) \, dx.
\end{split}
\end{align}
For the equation about magnetic field, multiplying $\sigma B_{t}$ on both sides, we will have
\begin{align}\label{h2esforB}
\begin{split}
& \frac{1}{2}\sigma \int_{\mathbb{R}^{3}} |B_{t}|^{2} \, dx + \nu \int_{0}^{T_{0}}\int_{\mathbb{R}^{3}} \sigma |\nabla B_{t}|^{2} \, dx dt    \\
& \quad\quad
= - \int_{0}^{T_{0}}\int_{\mathbb{R}^{3}} \sigma B_{t} (\mathrm{div}(B \, u^{T} - u \, B^{T}))_{t} \, dx dt
+ \frac{1}{2}\int_{0}^{T_{0}}\int_{\mathbb{R}^{3}}\sigma' |B_{t}|^{2} \, dx dt.
\end{split}
\end{align}
For the first term on the right hand side of the above equality, we have
\begin{align}\label{h2esforBin}
\begin{split}
& - \int_{0}^{T_{0}}\int_{\mathbb{R}^{3}}\sigma B_{t}(\mathrm{div}(B \, u^{T} - u\, B^{T}))_{t}\, dx dt   \\
= & \int_{0}^{T_{0}}\int_{\mathbb{R}^{3}}\sigma \nabla B_{t} (B_{t}\, u^{T} + B\, u_{t}^{T} - u_{t}\, B^{T} - u\, B_{t}^{T}) \, dx dt   \\
\leq & \epsilon \int_{0}^{T_{0}}\int_{\mathbb{R}^{3}}\sigma |\nabla B_{t}|^{2} \, dx dt
+ C \int_{0}^{T_{0}}\int_{\mathbb{R}^{3}} \sigma (|B_{t}|^{2}|u|^{2} + |B|^{2}|u_{t}|^{2})\, dx dt  \\
\leq & \epsilon \int_{0}^{T_{0}}\int_{\mathbb{R}^{3}}\sigma |\nabla B_{t}|^{2} \, dx dt     \\
& \quad\quad\quad\,\,
+ C \int_{0}^{T_{0}}\int_{\mathbb{R}^{3}}\sigma ( |B_{t}|^{2}|u|^{2} + |B|^{2}|\dot{u}|^{2} + |B|^{2}|u|^{2}|\nabla u|^{2} ) \, dx dt,
\end{split}
\end{align}
where $\epsilon$ is a small enough positive number.
Combining estimates (\ref{h2esforu}), (\ref{h2ns1}), (\ref{h2ns2}), (\ref{h2ns3}), (\ref{h2nnnB}), (\ref{h2esforB}) and (\ref{h2esforBin}), we will obtain
\begin{align}\label{h2combin}
\begin{split}
& \sigma \int_{\mathbb{R}^{3}} \rho |\dot{u}|^{2} + |B_{t}|^{2} \, dx
+ \int_{0}^{T_{0}}\int_{\mathbb{R}^{3}} \sigma (|\nabla\dot{u}|^{2} + |\nabla B_{t}|^{2}) \, dx dt  \\
\leq & C \int_{0}^{T_{0}}\int_{\mathbb{R}^{3}} \sigma |\nabla u|_{L^{2}}|\nabla\dot{u}|_{L^{2}} \, dx dt
+ C \int_{0}^{T_{0}}\int_{\mathbb{R}^{3}} \sigma |\nabla u|^{4} \, dx dt    \\
& + C \int_{0}^{1\wedge T_{0}}\int_{\mathbb{R}^{3}}\rho |\dot{u}|^{2} \, dx dt
+ C \int_{0}^{1 \wedge T_{0}}\int_{\mathbb{R}^{3}} |B_{t}|^{2} \, dx dt \\
& + C \int_{0}^{T_{0}}\int_{\mathbb{R}^{3}}\sigma (|u|^{2}+|B|^{2})(|B_{t}|^{2} + |\dot{u}|^{2} + |u|^{2}(|\nabla B|^{2} + |\nabla u|^{2})) \, dx dt.
\end{split}
\end{align}
We can easily obtain the following estimates
\begin{align}\label{h2combin1}
\begin{split}
&\int_{0}^{T_{0}}\int_{\mathbb{R}^{3}} \sigma |\nabla u|_{L^{2}} |\nabla\dot{u}|_{L^{2}} \, dx dt    \\
& \quad
\leq C \int_{0}^{T_{0}}\int_{\mathbb{R}^{3}} \sigma |\nabla\dot{u}|^{2} \, dx dt
+ C \int_{0}^{T_{0}}\int_{\mathbb{R}^{3}} \sigma |\nabla u|^{2} \, dx dt    \\
& \quad
\leq C C_{0} + C A_{1}(T_{0}),
\end{split}
\end{align}
\begin{align}\label{h2combin2}
\begin{split}
\int_{0}^{1\wedge T_{0}}\int_{\mathbb{R}^{3}} \rho |\dot{u}|^{2} \, dx dt + \int_{0}^{1 \wedge T_{0}}\int_{\mathbb{R}^{3}} |B_{t}|^{2} \, dx dt \leq C A_{1}(T_{0}).
\end{split}
\end{align}
Substitute (\ref{h2combin1}) and (\ref{h2combin2}) into (\ref{h2combin}), we finally complete the proof.
\end{proof}
\begin{lemma}\label{hofflemma6}
\begin{align*}
A_{1}(T_{0}) & + A_{2}(T_{0}) \leq C C_{0} + C C_{0} (C_{0}^{5/3} + C_{0}A_{2}(T_{0})) A_{1}(T_{0})^{3/2}   \\
& + C C_{0} A_{1}(T_{0})^{2} + C C_{0}^{2} A_{1}(T_{0})^{3/2} (\sigma |\nabla^{2}B|^{2}_{L^{2}})   \\
& + C \int_{0}^{T_{0}}\int_{\mathbb{R}^{3}} |\nabla u|^{3} + |\nabla B|^{3} \, dx dt
+ C \int_{0}^{T_{0}}\int_{\mathbb{R}^{3}} \sigma ( |\nabla u|^{4} + |\nabla B|^{4} ) \, dx dt.
\end{align*}
\end{lemma}
\begin{proof}
Combining Lemma \ref{hofflemma1}, Lemma \ref{hofflemma2} and Lemma \ref{hofflemma3}, we will have
\begin{align}\label{lemma6hoffzong}
A_{1}(T_{0}) & + A_{2}(T_{0}) \leq C C_{0} + C \int_{0}^{T_{0}}\int_{\mathbb{R}^{3}}\sigma |\nabla u|^{4} \, dx dt
+ C \int_{0}^{T_{0}}\int_{\mathbb{R}^{3}} |\nabla u|^{3} \, dx dt    \nonumber \\
& + C \int_{0}^{T_{0}}\int_{\mathbb{R}^{3}} \sigma (|u|^{2}+|B|^{2})(|B_{t}|^{2}+|\dot{u}|^{2}+|u|^{2}(|\nabla B|^{2}+|\nabla u|^{2})) dx dt \\
& + C \int_{0}^{T_{0}}\int_{\mathbb{R}^{3}} (|B|^{2}+|u|^{2})(|\nabla B|^{2}+|\nabla u|^{2}) \, dx dt.    \nonumber
\end{align}
Next we need to estimate some typical terms on the right hand side. \\
$\mathbf{Term \,\, 1:}$ $\int_{0}^{T_{0}}\int_{\mathbb{R}^{3}} |B|^{2} |\nabla B|^{2} \, dx dt$
\begin{align*}
& \int_{0}^{T_{0}}\int_{\mathbb{R}^{3}} |B|^{2} |\nabla B|^{2} \, dx dt \leq
\int_{0}^{T_{0}}\left( \int_{\mathbb{R}^{3}} |B|^{6} \, dx \right)^{1/3} \left( \int_{\mathbb{R}^{3}} |\nabla B|^{3} \, dx \right)^{2/3} \, dt  \\
& \quad\quad\quad
\leq C \int_{0}^{T_{0}}\int_{\mathbb{R}^{3}} |\nabla B|^{3} \, dx dt + C \int_{0}^{T_{0}}\int_{\mathbb{R}^{3}} |B|^{6} \, dx dt \\
& \quad\quad\quad
\leq C \int_{0}^{T_{0}} \left( \int_{\mathbb{R}^{3}} |\nabla B|^{3} \, dx \right)^{3} \, dt
+ C \int_{0}^{T_{0}} \int_{\mathbb{R}^{3}} |\nabla B|^{3} \, dx dt  \\
& \quad\quad\quad
\leq C \left( \int_{\mathbb{R}^{2}} |\nabla B|^{2} \, dx \right)^{2} \int_{0}^{T_{0}}\int_{\mathbb{R}^{3}} |\nabla B|^{2} \, dx dt
+ C \int_{0}^{T_{0}}\int_{\mathbb{R}^{3}} |\nabla B|^{3} \, dx dt   \\
& \quad\quad\quad
\leq C C_{0} A_{1}(T_{0})^{2} + C \int_{0}^{T_{0}}\int_{\mathbb{R}^{3}} |\nabla B|^{3} \, dx dt.
\end{align*}
$\mathbf{Term \,\, 2:}$ $\int_{0}^{T_{0}}\int_{\mathbb{R}^{3}} \sigma |B|^{2} |B_{t}|^{2} \, dx dt$
\begin{align*}
& \int_{0}^{T_{0}}\int_{\mathbb{R}^{3}}\sigma |B|^{2} |B_{t}|^{2} \, dx dt \leq
\left( \int_{0}^{T_{0}}\int_{\mathbb{R}^{3}} |B|^{6} \, dx dt \right)^{1/3}
\left( \int_{0}^{T_{0}}\int_{\mathbb{R}^{3}} \sigma^{3/2} |B_{t}|^{3} \, dx dt \right)^{2/3}    \\
& \quad\quad\quad
\leq \left( \int_{0}^{T_{0}} \left( \int_{\mathbb{R}^{3}} |\nabla B|^{2} \, dx \right)^{3} dt \right)^{1/3}
\left( \sigma \|B_{t}\|^{2}_{L^{2}} \right)^{1/3} \left( \int_{0}^{T_{0}}\int_{\mathbb{R}^{3}} \sigma |B_{t}|^{2} \, dx dt \right)^{2/3}    \\
& \quad\quad\quad
\leq C C_{0}^{1/3} A_{1}(T_{0})^{4/3} A_{2}(T_{0})^{1/3}.
\end{align*}
$\mathbf{Term \,\, 3:}$ $\int_{0}^{T_{0}}\int_{\mathbb{R}^{3}} \sigma |B|^{2} |u|^{2} |\nabla B|^{2} \, dx dt$
\begin{align*}
& \int_{0}^{T_{0}}\int_{\mathbb{R}^{3}} \sigma |B|^{2} |u|^{2} |\nabla B|^{2} \, dx dt \leq
\int_{0}^{T_{0}}\int_{\mathbb{R}^{3}} \sigma (|B|^{8} + |u|^{8} + |\nabla B|^{4}) \, dx dt  \\
& \quad\quad\quad
\leq \int_{0}^{T_{0}}\int_{\mathbb{R}^{3}} \sigma |\nabla B|^{4} \, dx dt
+ (\|u\|^{4}_{L^{4}} + \|B\|^{4}_{L^{4}})\int_{0}^{T_{0}} \sigma (\|u\|^{4}_{L^{\infty}} + \|B\|^{4}_{L^{\infty}}) \, dt
\end{align*}
Since
\begin{align*}
\|u\|^{4}_{L^{4}} + \|B\|^{4}_{L^{4}} \leq \|u\|_{L^{2}} \|\nabla u\|^{3}_{L^{2}} + \|B\|_{L^{2}} \|\nabla B\|^{3}_{L^{2}},
\end{align*}
and
\begin{align*}
& \int_{0}^{T_{0}} \sigma \|u\|^{4}_{L^{\infty}} + \sigma \|B\|^{4}_{L^{\infty}} \, dt \leq
\int_{0}^{T_{0}} \sigma \|u\|^{2}_{L^{6}}\|\nabla u\|^{2}_{L^{6}} + \sigma \|B\|^{2}_{L^{6}} \|\nabla B\|^{2}_{L^{6}} \, dt     \\
& \quad\quad
\leq \int_{0}^{T_{0}}\sigma \|\nabla u\|^{2}_{L^{2}} (\|\rho\dot{u}\|_{L^{2}} + \|P(\rho) - P(\bar{\rho})\|_{L^{6}})^{2} \, dt
+ \int_{0}^{T_{0}} \sigma \|\nabla B\|^{2}_{L^{2}} \|\nabla^{2}B\|^{2}_{L^{2}} \, dt     \\
& \quad\quad
\leq \sigma \|\rho\dot{u}\|^{2}_{L^{2}} \int_{0}^{T_{0}} \|\nabla u\|^{2}_{L^{2}} \, dt
+ C C_{0}^{5/3} + C C_{0} \sigma \|\nabla^{2} B\|^{2}_{L^{2}},
\end{align*}
we can get the estimate
\begin{align*}
\int_{0}^{T_{0}}\int_{\mathbb{R}^{3}} \sigma |B|^{2} |u|^{2} |\nabla B|^{2} \, dx dt \leq &
\int_{0}^{T_{0}}\int_{\mathbb{R}^{3}} \sigma |\nabla B|^{4} \, dx dt    \\
+ & C C_{0} A_{1}(T_{0})^{3/2} (C_{0}^{5/3} + C_{0}\sigma \|\nabla^{2}B\|^{2}_{L^{2}} + C_{0}A_{2}(T_{0})).
\end{align*}
where we used Lemma \ref{hofflemma1} and the definition of $A_{1}$, $A_{2}$.

Using similar procedure as did for Term 1, we have
\begin{align}\label{lemma6ty1}
\begin{split}
& \int_{0}^{T_{0}}\int_{\mathbb{R}^{3}}(|B|^{2} + |u|^{2})(|\nabla B|^{2} + |\nabla u|^{2}) \, dx dt  \\
& \quad\quad\quad\quad
\leq C C_{0} A_{1}(T_{0})^{2} + C \int_{0}^{T_{0}}\int_{\mathbb{R}^{3}} |\nabla B|^{3} \, dx dt
+ C \int_{0}^{T_{0}}\int_{\mathbb{R}^{3}} |\nabla u|^{3} \, dx dt.
\end{split}
\end{align}
Using similar procedure as did for Term 2, we have
\begin{align}\label{lemma6ty2}
\begin{split}
\int_{0}^{T_{0}}\int_{\mathbb{R}^{3}}\sigma (|u|^{2} + |B|^{2})(|B_{t}|^{2} + |\dot{u}|^{2})\, dx dt \leq C C_{0}^{1/3} A_{1}(T_{0})^{4/3} A_{2}(T_{0})^{1/3}.
\end{split}
\end{align}
Using similar procedure as did for Term 3, we have
\begin{align}\label{lemma6ty3}
\begin{split}
& \int_{0}^{T_{0}}\int_{\mathbb{R}^{3}} \sigma (|u|^{2} + |B|^{2})|u|^{2}(|\nabla B|^{2} + |\nabla u|^{2}) \, dx dt   \\
& \quad\quad\quad\quad
\leq C C_{0} A_{1}(T_{0})^{3/2} (C_{0}^{5/3} + C_{0}\sigma \|\nabla^{2}B\|^{2}_{L^{2}} + C_{0}A_{2}(T_{0})) \\
& \quad\quad\quad\quad\quad
+ \int_{0}^{T_{0}}\int_{\mathbb{R}^{3}} \sigma |\nabla B|^{4} \, dx dt.
\end{split}
\end{align}
At last substitute (\ref{lemma6ty1}), (\ref{lemma6ty2}), (\ref{lemma6ty3}) into (\ref{lemma6hoffzong})
and do simple reduction, we will complete our proof.
\end{proof}

Before going to the following part, we need to introduce some notations. Let
\begin{align}\label{effectivedefi}
F = (\lambda + \mu) \mathrm{div}u - P(\rho) + P(\bar{\rho}),
\end{align}
and
\begin{align}\label{omegadefi}
\omega^{j,k} = u^{j}_{x_{k}} - u^{k}_{x_{j}}, \quad j,k = 1,2,3.
\end{align}
Through simple calculations, we could find that
\begin{align}\label{equforomega}
\mu\Delta \omega^{j,k} = (\rho\dot{u}^{j})_{x_{k}} - (\rho\dot{u}^{k})_{x_{j}} - (B\cdot\nabla B^{j})_{x_{k}} + (B\cdot\nabla B^{k})_{x_{j}},
\end{align}
and
\begin{align}\label{equforF}
\Delta F = \mathrm{div} g
\end{align}
where $g^{j} = \rho\dot{u}^{j} + (\frac{1}{2}|B|^{2})_{x_{j}} - \mathrm{div}(B^{j}B)$ with $j=1,2,3$.
\begin{lemma}\label{hofflemma5}
\begin{align*}
E(T_{0}) \leq C C_{0} + P_{E}(A(T_{0})),
\end{align*}
where $A(T_{0}) = A_{1}(T_{0}) + A_{2}(T_{0})$ and $P_{E}$ is a polynomial function with $2$ as the lowest order.
\end{lemma}
\begin{proof}
For the term $\int_{\mathbb{R}^{3}}\sigma |\nabla B|^{2} |u|^{2} \, dx$, we have
\begin{align}\label{lemma5ty}
\begin{split}
\int_{\mathbb{R}^{3}}\sigma |\nabla B|^{2}|u|^{2}\, dx & \leq \sigma \|u\|^{2}_{L^{\infty}} \int_{\mathbb{R}^{2}} |\nabla B|^{2} \, dx  \\
& \leq C A_{1}(T_{0}) \left( \sigma \|u\|^{2}_{L^{4}} + \sigma \|\nabla u\|^{2}_{L^{4}} \right).
\end{split}
\end{align}
For term $\sigma \|u\|^{2}_{L^{4}}$ which appeared in the above inequality, we have
\begin{align}\label{lemma5right}
\sigma \|u\|^{2}_{L^{4}} \leq \sigma \|u\|^{1/2}_{L^{2}} \|\nabla u\|^{3/2}_{L^{2}} \leq C C_{0}^{1/2} A_{1}(T_{0})^{3/4}.
\end{align}
The analysis about $\sigma \|\nabla u\|^{2}_{L^{4}}$ is a little bit more complex, by (\ref{effectivedefi}), we know that
\begin{align}\label{lemma5esti1}
\sigma \|\nabla u\|^{2}_{L^{4}} \leq \sigma \|F\|^{2}_{L^{4}} + \sigma \|\omega\|^{2}_{L^{4}} + \sigma \|P(\rho) - P(\bar{\rho})\|^{2}_{L^{4}}.
\end{align}
For the first term on the right hand side of the above inequality, we have
\begin{align}\label{lemma5esti2}
\begin{split}
\sigma \|F\|^{2}_{L^{4}} & \leq \sigma \|F\|_{L^{2}}^{1/2} \|\nabla F\|_{L^{2}}^{3/2} \\
\leq C & \left(\|\nabla u\|_{L^{2}} + \|\rho - \bar{\rho}\|_{L^{2}}\right)^{1/2} \left(\int_{\mathbb{R}^{3}}\sigma |\dot{u}|^{2} \, dx
+ \int_{\mathbb{R}^{3}} \sigma |\nabla B|^{2} |B|^{2} \, dx \right)^{3/4}   \\
\leq C & \left( A_{1}(T_{0})^{1/4} + C_{0}^{1/4} \right)\left( A_{2}(T_{0})^{3/4} + E(T_{0})^{3/4} \right).
\end{split}
\end{align}
For the second term on the right hand side of (\ref{lemma5esti1}), we have
\begin{align}\label{lemma5esti3}
\begin{split}
\sigma \|\omega\|^{2}_{L^{4}} & \leq \sigma \|\omega\|_{L^{2}}^{1/2} \|\nabla \omega\|_{L^{2}}^{3/2}    \\
& \leq \|\nabla u\|_{L^{2}}^{1/2} \left( \int_{\mathbb{R}^{3}} \sigma |\nabla \omega|^{2} \, dx \right)^{3/4}   \\
& \leq A_{1}(T_{0})^{1/4} \left( \int_{\mathbb{R}^{3}}\sigma |\dot{u}|^{2} \, dx + \int_{\mathbb{R}^{3}} \sigma |\nabla B|^{2} |B|^{2} \, dx \right)^{3/4}  \\
& \leq A_{1}(T_{0})^{1/4} \left( A_{2}(T_{0})^{3/4} + E(T_{0})^{3/4} \right).
\end{split}
\end{align}
For the third term on the right hand side of (\ref{lemma5esti1}), we have
\begin{align}\label{lemma5esti4}
\sigma \|P(\rho) - P(\bar{\rho})\|^{2}_{L^{4}} \leq C \left( \int_{\mathbb{R}^{3}} |\rho - \bar{\rho}|^{2} \, dx \right)^{1/2}
\leq C C_{0}^{1/2}.
\end{align}
Substitute (\ref{lemma5esti2}), (\ref{lemma5esti3}), (\ref{lemma5esti4}) into (\ref{lemma5esti1}), we will obtain
\begin{align}\label{lemma5esti5}
\sigma \|\nabla u\|^{2}_{L^{4}} \leq C C_{0}^{1/2} + C (A_{1}(T_{0})^{1/4} + C_{0}^{1/4})(A_{2}(T_{0})^{3/4} + E(T_{0})^{3/4})
\end{align}
Combining (\ref{lemma5right}), (\ref{lemma5esti5}) and (\ref{lemma5ty}), we finally have
\begin{align}\label{lemma5esti6}
\begin{split}
\int_{\mathbb{R}^{3}} \sigma |\nabla B|^{2} |u|^{2} \, dx \leq & C A_{1}(T_{0})
\Big( C_{0}^{1/2} A_{1}(T_{0})^{3/4} + C_{0}^{1/2}  \\
& + (A_{1}(T_{0})^{1/4} + C_{0}^{1/4})(A_{2}(T_{0})^{3/4} + E(T_{0})^{3/4}) \Big).
\end{split}
\end{align}
Using similar arguments as for term $\int_{\mathbb{R}^{3}}\sigma |\nabla B|^{2} |u|^{2} \, dx$, we have
\begin{align}\label{lemma5esti7}
\begin{split}
& \int_{\mathbb{R}^{3}} \sigma |\nabla u|^{2} |B|^{2} \, dx + \int_{\mathbb{R}^{3}} \sigma |\nabla B|^{2} |B|^{2} \, dx   \\
& \quad\quad\quad
\leq C A_{1}(T_{0}) \Big( C_{0}^{1/2} A_{1}(T_{0})^{3/4} + A_{1}(T_{0})^{1/4} (A_{2}(T_{0}) + E(T_{0}))^{3/4} \Big).
\end{split}
\end{align}
Summing up (\ref{lemma5esti6}) and (\ref{lemma5esti7}), we obtain
\begin{align*}
E(T_{0}) & \leq C A_{1}(T_{0}) \Big( C_{0}^{1/2} A_{1}(T_{0})^{3/4} + C_{0}^{1/2}   \\
& \quad\quad\quad\quad
+ (A_{1}(T_{0})^{1/4} + C_{0}^{1/4})(A_{2}(T_{0})^{3/4} + E(T_{0})^{3/4}) \Big) \\
& \lesssim C_{0}^{1/2} A_{1}(T_{0})^{7/4} + C_{0}^{1/2} A_{1}(T_{0}) + A_{1}(T_{0})^{5/4} A_{2}(T_{0})^{3/4}  \\
& \quad
+ A_{1}(T_{0})^{5/4} E(T_{0})^{3/4} + A_{1}(T_{1}) C_{0}^{1/4} A_{2}(T_{0})^{3/4} + C_{0}^{1/4} A_{1}(T_{0}) E(T_{0})^{3/4}.
\end{align*}
Using Young's inequality, we can get
\begin{align}\label{lemma5esti8}
\begin{split}
E(T_{0}) \leq & C C_{0}^{\frac{1}{2}} A_{1}(T_{0})^{\frac{7}{4}} + C C_{0}^{\frac{1}{2}} A_{1}(T_{0}) + C C_{0} A_{1}(T_{0})^{4}  \\
& + C A_{1}(T_{0})^{5} + C A_{1}(T_{0})^{\frac{5}{4}} A_{2}(T_{0})^{\frac{3}{4}} + C C_{0}^{\frac{1}{4}} A_{1}(T_{0})A_{2}(T_{0})^{\frac{3}{4}}.
\end{split}
\end{align}
Through (\ref{lemma5esti8}), we could easily obtain our desired result.
\end{proof}

At this stage, we could give the estimate about term $\sigma \|\nabla^{2}B\|_{L^{2}}^{2}$ as follows
\begin{align}\label{refinedestimates}
\begin{split}
\int_{\mathbb{R}^{3}} \sigma |\nabla^{2}B|^{2} \, dx & \leq C \int_{\mathbb{R}^{3}} \sigma |B_{t}|^{2} \, dx
+ C \int_{\mathbb{R}^{3}} \sigma |\nabla B|^{2} |u|^{2} + \sigma |\nabla u|^{2} |B|^{2} \, dx   \\
& \leq C A_{2}(T_{0}) + C E(T_{0})      \\
& \leq C (1 + C_{0}^{1/2}) A_{2}(T_{0}) + C A_{1}(T_{0})^{5/4} A_{2}(T_{0})^{3/4}   \\
& \quad
+ C C_{0}^{1/4} A_{1}(T_{0}) A_{2}(T_{0})^{3/4} + C A_{1}(T_{0})^{5} + C C_{0} A_{1}(T_{0})^{4},
\end{split}
\end{align}
where we used Lemma \ref{hofflemma5}.

\begin{lemma}\label{hofflemma4}
\begin{align*}
H(T_{0}) \leq C P_{HC}(C_{0}) + C P_{HA}(A(T_{0})),
\end{align*}
where
$P_{HC}$ is a polynomial function with lowest order $\frac{3}{4}$ and
$P_{HA}$ is a polynomial function with lowest order $\frac{9}{8}$.
\end{lemma}
\begin{proof}
We need to estimate every term appeared in the definition of $H(T_{0})$.
For $\int_{0}^{T_{0}}\int_{\mathbb{R}^{3}} \sigma |\nabla B|^{4} \, dx dt$, we have
\begin{align}\label{lemma4est1}
\begin{split}
\int_{0}^{T_{0}}\int_{\mathbb{R}^{3}} \sigma |\nabla B|^{4} \, dx dt & \leq C \int_{0}^{T_{0}}\sigma \|\nabla B\|_{L^{2}} \|\nabla^{2}B\|^{3}_{L^{2}} \, dt \\
& \leq C A_{1}(T_{0})^{1/2} \sigma^{1/2}\|\nabla^{2}B\|_{L^{2}} \int_{0}^{T_{0}} \sigma \|\nabla^{2}B\|^{2}_{L^{2}} \, dt.
\end{split}
\end{align}
For the last term of the above inequality, we have
\begin{align}\label{lemma4est2}
\begin{split}
\int_{0}^{T_{0}}\int_{\mathbb{R}^{3}} \sigma |\nabla^{2}B|^{2} \, dx dt
& \leq C \int_{0}^{T_{0}}\int_{\mathbb{R}^{3}} \sigma ( |B_{t}|^{2} + |\nabla B|^{2} |u|^{2} + |\nabla u|^{2} |B|^{2}) \, dx dt \\
& \lesssim A_{1}(T_{0}) + \int_{0}^{T_{0}}\int_{\mathbb{R}^{3}} \sigma (|\nabla B|^{2}|u|^{2} + |\nabla u|^{2}|B|^{2}) \, dx dt.
\end{split}
\end{align}
For the last two terms in the above inequality, we firstly have
\begin{align}\label{lemma4est3}
\begin{split}
& \int_{0}^{T_{0}}\int_{\mathbb{R}^{3}} \sigma |\nabla B|^{2} |u|^{2} \, dx dt    \\
\leq & \int_{0}^{T_{0}} \left( \int_{\mathbb{R}^{3}} \sigma |\nabla B|^{4} \, dx \right)^{1/2}
\left( \int_{\mathbb{R}^{3}} \sigma |u|^{4} \, dx \right)^{1/2} dt  \\
\leq & \left( \int_{0}^{T_{0}}\int_{\mathbb{R}^{3}} \sigma |\nabla B|^{4} \, dx dt \right)^{1/2}
\left( \int_{0}^{T_{0}}\int_{\mathbb{R}^{3}} \sigma |u|^{4} \, dx dt \right)^{1/2}  \\
\leq & H(T_{0})^{1/2}
\left( \int_{0}^{T_{0}} \left( \int_{\mathbb{R}^{3}} |u|^{2} \, dx \right)^{1/2} \left( \int_{\mathbb{R}^{3}} |\nabla u|^{2} \, dx \right)^{3/2} dt \right)^{1/2}   \\
\leq & H(T_{0})^{1/2} C_{0}^{1/4} \left( \int_{\mathbb{R}^{3}} |\nabla u|^{2} \, dx \right)^{1/4}
\left( \int_{0}^{T_{0}}\int_{\mathbb{R}^{3}} |\nabla u|^{2} \, dx dt \right)^{1/2}  \\
\leq & H(T_{0})^{1/2} C_{0}^{3/4} A_{1}(T_{0})^{1/4}.
\end{split}
\end{align}
Similarly, we could get
\begin{align}\label{lemma4est4}
\int_{0}^{T_{0}}\int_{\mathbb{R}^{3}} \sigma |\nabla u|^{2} |B|^{2} \, dx dt \leq H(T_{0})^{1/2} C_{0}^{3/4} A_{1}(T_{0})^{1/4}.
\end{align}
Substitute (\ref{lemma4est3}) and (\ref{lemma4est4}) into (\ref{lemma4est2}), we will have
\begin{align}\label{lemma4est5}
\int_{0}^{T_{0}}\int_{\mathbb{R}^{3}} \sigma |\nabla^{2}B|^{2} \, dx dt
\leq C A_{1}(T_{0}) + C H(T_{0})^{1/2} C_{0}^{3/4} A_{1}(T_{0})^{1/4}.
\end{align}
Substitute estimates (\ref{refinedestimates}) and (\ref{lemma4est5}) into (\ref{lemma4est1}), we obtain
\begin{align}\label{lemma4est6}
\begin{split}
\int_{0}^{T_{0}}\int_{\mathbb{R}^{3}} \sigma |\nabla B|^{4} \, dx dt \leq & C A_{1}(T_{0})^{1/2}
((1+C_{0})^{1/2}A_{2}(T_{0})    \\
& + A_{1}(T_{0})^{5/4}A_{2}(T_{0})^{3/4} + C_{0}^{1/4}A_{1}(T_{0})A_{2}(T_{0})^{3/4}    \\
& + A_{1}(T_{0})^{5} + C_{0}A_{1}(T_{0})^{4} )^{1/2}( A_{1}(T_{0})  \\
& + C_{0}^{3/4}A_{1}(T_{0})^{1/4}H(T_{0})^{1/2} )
\end{split}
\end{align}
For the term $\int_{0}^{T_{0}}\int_{\mathbb{R}^{3}} \sigma |\nabla u|^{4} \, dx dt$, we have
\begin{align}\label{lemma4est7}
\begin{split}
\int_{0}^{T_{0}}\int_{\mathbb{R}^{3}} \sigma |\nabla u|^{4} \, dx dt \lesssim
\int_{0}^{T_{0}}\int_{\mathbb{R}^{3}} \sigma (|\rho - \bar{\rho}|^{4} + |F|^{4} + |\omega|^{4}) \, dx dt.
\end{split}
\end{align}
Concerning the first term on the right hand side of the above inequality, we have
\begin{align}\label{lemma4est8}
\int_{0}^{T_{0}}\int_{\mathbb{R}^{3}} \sigma |\rho - \bar{\rho}|^{4} \, dx dt \lesssim C_{0} + \int_{0}^{T_{0}}\int_{\mathbb{R}^{3}} \sigma |F|^{4} \, dx dt.
\end{align}
For the second term on the right hand side of (\ref{lemma4est7}), we have
\begin{align}\label{lemma4est9}
\begin{split}
& \int_{0}^{T_{0}}\int_{\mathbb{R}^{3}}\sigma |F|^{4} \, dx dt \leq \int_{0}^{T_{0}} \sigma \|F\|_{L^{2}} \|\nabla F\|^{3}_{L^{2}} \, dt  \\
& \quad\quad\quad
\leq \left( \int_{\mathbb{R}^{3}}|F|^{2}\, dx \right)^{1/2} \left( \int_{\mathbb{R}^{3}}\sigma |\nabla F|^{2}\, dx \right)^{1/2}
\int_{0}^{T_{0}}\int_{\mathbb{R}^{3}} |\nabla F|^{2} \, dx dt.
\end{split}
\end{align}
Due to
\begin{align}\label{lemma4est10}
\begin{split}
\left( \int_{\mathbb{R}^{3}} |F|^{2} \, dx \right)^{1/2} & \leq \left( \int_{\mathbb{R}^{3}} |\nabla u|^{2} \, dx \right)^{1/2}
+ \left( \int_{\mathbb{R}^{3}} |\rho - \bar{\rho}|^{2} \, dx \right)^{1/2}  \\
& \leq C (C_{0} + A_{1}(T_{0}))^{1/2},
\end{split}
\end{align}
\begin{align}\label{lemma4est11}
\begin{split}
\left( \int_{\mathbb{R}^{3}} \sigma |\nabla F|^{2} \, dx \right)^{1/2} & \leq C \int_{\mathbb{R}^{3}} \sigma (|\dot{u}|^{2} + |\nabla B|^{2}|B|^{2}) \, dx  \\
& \leq C A_{2}(T_{0}) + C E(T_{0}),
\end{split}
\end{align}
and
\begin{align}\label{lemma4est12}
\begin{split}
& \int_{0}^{T_{0}}\int_{\mathbb{R}^{3}} |\nabla F|^{2} \, dx dt \leq C \int_{0}^{T_{0}} \int_{\mathbb{R}^{3}} |\dot{u}|^{2}
+ |\nabla B|^{2} |B|^{2} \, dx dt   \\
& \quad\quad\quad
\leq C A_{1}(T_{0}) + C \left( \int_{0}^{T_{0}}\int_{\mathbb{R}^{3}} |\nabla B|^{3} \, dx dt \right)^{2/3}
\left( \int_{0}^{T_{0}}\int_{\mathbb{R}^{3}} |B|^{6} \, dx dt \right)^{1/3} \\
& \quad\quad\quad
\leq C A_{1}(T_{0}) + C H(T_{0})^{2/3} C_{0}^{1/2} A_{1}(T_{0})^{1/6},
\end{split}
\end{align}
we can refine the estimate (\ref{lemma4est10}) to
\begin{align}\label{lemma4est13}
\begin{split}
\int_{0}^{T_{0}}\int_{\mathbb{R}^{3}} \sigma |F|^{4} \, dx dt \leq & C (C_{0} + A_{1}(T_{0}))^{1/2}(A_{2}(T_{0}) + E(T_{0}))
(A_{1}(T_{0})   \\
& + H(T_{0})^{2/3}C_{0}^{1/2}A_{1}(T_{0})^{1/6})
\end{split}
\end{align}
Concerning the last term on the right hand side of (\ref{lemma4est7}), we have
\begin{align}\label{lemma4est14}
\begin{split}
& \int_{0}^{T_{0}}\int_{\mathbb{R}^{3}} \sigma |\omega|^{4} \, dx dt  \leq \int_{0}^{T_{0}} \sigma \|\omega\|_{L^{2}} \|\omega\|^{3}_{L^{6}} \, dt    \\
& \quad\quad
\leq \|\omega\|_{L^{2}}\int_{\mathbb{R}^{3}}\sigma |\nabla\omega|^{2} \, dx \int_{0}^{T_{0}}\int_{\mathbb{R}^{3}} |\nabla\omega|^{2} \, dx dt   \\
& \quad\quad
\leq C A_{1}(T_{0})^{1/2}(A_{2}(T_{0}) + E(T_{0}))\Big( \int_{0}^{T_{0}}\int_{\mathbb{R}^{3}}\rho |\dot{u}|^{2} \, dx dt    \\
& \quad\quad\quad
+ \int_{0}^{T_{0}}\int_{\mathbb{R}^{3}} |\nabla B|^{2} |B|^{2} \, dx dt  \Big)   \\
& \quad\quad
\leq C A_{1}(T_{0})^{1/2} (A_{2}(T_{0}) + E(T_{0})) \Big( A_{1}(T_{0})  \\
& \quad\quad\quad
+ \big( \int_{0}^{T_{0}}\int_{\mathbb{R}^{3}}|\nabla B|^{3} \, dx dt \big)^{2/3}
\big( \int_{0}^{T_{0}}\int_{\mathbb{R}^{3}} |B|^{6} \, dx dt \big)^{1/3} \Big)      \\
& \quad\quad
\leq C A_{1}(T_{0})^{1/2} (A_{2}(T_{0}) + E(T_{0})) (A_{1}(T_{0}) + C_{0}^{1/3}H(T_{0})^{2/3}A_{1}(T_{0})^{1/3}).
\end{split}
\end{align}
Substitute (\ref{lemma4est8}), (\ref{lemma4est13}) and (\ref{lemma4est14}) into (\ref{lemma4est7}), we obtain that
\begin{align}\label{lemma4est15}
\begin{split}
& \int_{0}^{T_{0}}\int_{\mathbb{R}^{3}} \sigma |\nabla u|^{4} \, dx dt \leq C C_{0}  +
C (C_{0} + A_{1}(T_{0}))^{1/2} (A_{2}(T_{0})    \\
& \quad\quad\quad\quad\quad\quad\quad\quad\quad\quad\,
+ E(T_{0})) (A_{1}(T_{0})
+ C_{0}^{1/2}A_{1}(T_{0})^{1/6}H(T_{0})^{2/3} )   \\
& \quad\quad\quad\quad\quad\quad\quad\quad\quad\quad\,
 + C A_{1}(T_{0})^{1/2} (A_{2}(T_{0}) + E(T_{0})) (A_{1}(T_{0})     \\
& \quad\quad\quad\quad\quad\quad\quad\quad\quad\quad\,
+ C_{0}^{1/3} A_{1}(T_{0})^{1/3}H(T_{0})^{2/3}).
\end{split}
\end{align}
Concerning the term $\int_{0}^{T_{0}}\int_{\mathbb{R}^{3}}|\nabla B|^{3}\, dx dt$, we have
\begin{align}\label{lemma4est16}
\begin{split}
& \int_{0}^{T_{0}}\int_{\mathbb{R}^{3}} |\nabla B|^{3} \, dx dt \leq C \int_{0}^{T_{0}} \|\nabla B\|^{3/2}_{L^{2}} \|\nabla^{2}B\|^{3/2}_{L^{2}} \, dt    \\
& \quad\quad
\leq \left( \int_{0}^{T_{0}} \|\nabla B\|^{6}_{L^{2}}\, dt \right)^{1/4} \left( \int_{0}^{T_{0}} \|\nabla^{2}B\|^{2}_{L^{2}} \, dt \right)^{3/4}  \\
& \quad\quad
\leq C \|\nabla B\|_{L^{2}} \left( \int_{0}^{T_{0}}\int_{\mathbb{R}^{3}} |\nabla B|^{2} \, dx dt \right)^{1/4}
\left( \int_{0}^{T_{0}}\int_{\mathbb{R}^{3}} |\nabla^{2}B|^{2} \, dx dt \right)^{3/4}   \\
& \quad\quad
\leq C A_{1}(T_{0})^{1/2} C_{0}^{1/4}
\left( \int_{0}^{T_{0}}\int_{\mathbb{R}^{3}} |B_{t}|^{2} + |\nabla B|^{2} |u|^{2} + |\nabla u|^{2} |B|^{2} \, dx dt \right)^{3/4}   \\
& \quad\quad
\leq C A_{1}(T_{0})^{1/2} C_{0}^{1/4} \left( A_{1}(T_{0})^{3/4} + C_{0}^{2/3} A_{1}(T_{0})^{4/3} H(T_{0})^{1/3} \right).
\end{split}
\end{align}
For the term $\int_{0}^{T_{0}}\int_{\mathbb{R}^{3}}|\nabla u|^{3}\, dx dt$, we have
\begin{align}\label{lemma4est17}
\begin{split}
& \int_{0}^{T_{0}}\int_{\mathbb{R}^{3}} |\nabla u|^{3} \, dx dt \leq \int_{0}^{\sigma(T_{0})} \int_{\mathbb{R}^{3}} |\nabla u|^{3} \, dx dt
+ \int_{\sigma(T_{0})}^{T_{0}} \int_{\mathbb{R}^{3}} |\nabla u|^{3} \, dx dt        \\
& \quad\quad\quad\quad
\leq \int_{0}^{\sigma(T_{0})} \|\nabla u\|^{3/2}_{L^{2}}( \|\rho\dot{u}\|^{3/2}_{L^{2}} + \|P(\rho) - P(\bar{\rho})\|^{3/2}_{L^{6}} ) \, dt   \\
& \quad\quad\quad\quad\quad
+ \left( \int_{\sigma(T_{0})}^{T_{0}}\int_{\mathbb{R}^{3}} |\nabla u|^{2} dx dt \right)^{1/2}
\left( \int_{\sigma(T_{0})}^{T_{0}}\int_{\mathbb{R}^{3}} |\nabla u|^{4} dx dt \right)^{1/2} \\
& \quad\quad\quad\quad
\leq C C_{0}^{1/4}A_{1}(T_{0})^{3/4} + A_{1}(T_{0})^{3/2} + C C_{0}^{1/2} H(T_{0})^{1/2}.
\end{split}
\end{align}
Combining estimates (\ref{lemma4est6}), (\ref{lemma4est15}), (\ref{lemma4est16}), (\ref{lemma4est17}) and (\ref{lemma5esti8}),
we could finish the proof by a long but tedious calculations.
\end{proof}

Now, combining estimate (\ref{refinedestimates}), Lemma \ref{hofflemma5}, Lemma \ref{hofflemma4} and Lemma \ref{hofflemma6}, we will obtain
\begin{align}\label{estimateforA1A2}
A_{1}(T_{0}) + A_{2}(T_{0}) \leq C P_{AC}(C_{0}) + C P_{AA}(A_{1}(T_{0}) + A_{2}(T_{0})),
\end{align}
where $P_{AC}(\cdot)$ is a polynomial function with lowest order $\frac{3}{4}$ and $P_{AA}(\cdot)$ is a polynomial function
with lowest order $\frac{9}{8}$.
Since $A_{1}(T_{0}) + A_{2}(T_{0}) \leq \epsilon_{0}^{1/2}$ and
$C_{0} = \|\rho_{0} - \bar{\rho}\|^{2}_{L^{2}} + \|u_{0}\|^{2}_{H^{1}} + \|B_{0}\|^{2}_{H^{1}} \leq \epsilon_{0}$,
if $\epsilon_{0} > 0$ is small enough, we have
\begin{align}\label{estimateforAfinal}
\begin{split}
A_{1}(T_{0}) + A_{2}(T_{0}) & \leq C \epsilon_{0}^{1/4}\epsilon_{0}^{1/2} + C \epsilon_{0}^{1/16} \epsilon_{0}^{1/2}
\leq \frac{1}{2} \epsilon_{0}^{1/2}.
\end{split}
\end{align}

It remains to prove the lower and upper bound of the density.
Set $\Gamma = \log (\rho)$ which satisfies
\begin{align*}
(\lambda + \mu) \dot{\Gamma} + (P(\rho) - P(\bar{\rho})) = - F.
\end{align*}
For $0 < t < \sigma(T_{0})$, we will have
\begin{align}\label{estimaterho1}
\begin{split}
& \int_{0}^{\sigma(T_{0})} \|F\|_{L^{\infty}} \, dt \leq C \int_{0}^{\sigma(T_{0})} \|F\|^{1/2}_{L^{6}} \|\nabla F\|^{1/2}_{L^{6}} \, dt    \\
& \quad\quad\quad
\leq C \int_{0}^{\sigma(T_{0})} (\|\sqrt{\rho}\dot{u}\|^{1/2}_{L^{2}} + \|B\cdot\nabla B\|^{1/2}_{L^{2}})
(\|\rho\dot{u}\|^{1/2}_{L^{6}} + \|B\cdot\nabla B\|^{1/2}_{L^{6}}) \, dt    \\
& \quad\quad\quad
\leq C \int_{0}^{\sigma(T_{0})} (\|\sqrt{\rho}\dot{u}\|^{1/2}_{L^{2}} + \|B\cdot\nabla B\|^{1/2}_{L^{2}})
(\|\nabla\dot{u}\|^{1/2}_{L^{2}} + \|B\cdot\nabla B\|^{1/2}_{L^{6}}) \, dt.
\end{split}
\end{align}
By estimate (\ref{estimateforAfinal}), we have
\begin{align}\label{rhoes1}
\int_{0}^{\sigma(T_{0})} \|\sqrt{\rho}\dot{u}\|^{1/2}_{L^{2}} \|\nabla\dot{u}\|^{1/2}_{L^{2}} \, dt \leq C \epsilon_{0}^{1/4},
\end{align}
\begin{align}\label{rhoes2}
\begin{split}
& \int_{0}^{\sigma(T_{0})} \|B\cdot\nabla B\|^{1/2}_{L^{2}} \|\nabla\dot{u}\|^{1/2}_{L^{2}} \, dt     \\
= & \int_{0}^{\sigma(T_{0})} t^{-1/2} (\sigma \|B\cdot\nabla B\|^{2}_{L^{2}})^{1/4} (\sigma \|\nabla\dot{u}\|^{2}_{L^{2}})^{1/4} \, dt  \\
\leq & C \epsilon_{0}^{1/4} \int_{0}^{\sigma(T_{0})} t^{-1/2} (\sigma \|\nabla\dot{u}\|^{2}_{L^{2}})^{1/4} \, dt    \\
\leq & C \epsilon_{0}^{1/4} \left( \int_{0}^{\sigma(T_{0})} t^{-2/3} \, dt \right)^{3/4}
\left( \int_{0}^{\sigma(T_{0})} \sigma \|\nabla\dot{u}\|^{2}_{L^{2}} \, dt \right)^{1/4}    \\
\leq & C \epsilon_{0}^{1/4} \epsilon_{0}^{1/8} \leq C \epsilon^{3/8},
\end{split}
\end{align}
\begin{align}\label{rhoes3}
\begin{split}
& \int_{0}^{\sigma(T_{0})} \|\sqrt{\rho}\dot{u}\|^{1/2}_{L^{2}} \|B\cdot\nabla B\|^{1/2}_{L^{6}} \, dt    \\
\leq & \int_{0}^{\sigma(T_{0})} t^{-3/8} \|\sqrt{\rho}\dot{u}\|^{1/2}_{L^{2}}\|\nabla B\|^{1/4}_{L^{2}} t^{3/8} \|\nabla^{2}B\|^{3/4}_{L^{2}} \, dt \\
& + \int_{0}^{\sigma(T_{0})} t^{-1/4} \|\sqrt{\rho}\dot{u}\|^{1/2}_{L^{2}} t^{1/4} \|\nabla B\|_{L^{4}} \, dt   \\
\leq & C \epsilon_{0}^{3/8} \left( \int_{0}^{\sigma(T_{0})} t^{-3/5} \, dt \right)^{5/8}
\left( \int_{0}^{\sigma(T_{0})} t \|\nabla^{2}B\|^{2}_{L^{2}} \, dt \right)^{3/8}   \\
& + C \epsilon^{1/4} \left( \int_{0}^{\sigma(T_{0})} t^{-1/3} \, dt \right)^{3/4}
\left( \int_{0}^{\sigma(T_{0})} t \|\nabla B\|^{4}_{L^{4}} \, dt \right)^{1/4}  \\
\leq & C \epsilon_{0}^{3/8}\epsilon_{0}^{1/2} + C \epsilon_{0}^{1/4} \epsilon_{0}^{9/16} \leq C \epsilon_{0}^{3/4},
\end{split}
\end{align}
and
\begin{align}\label{rhoes4}
\begin{split}
& \int_{0}^{\sigma(T_{0})} \|B\cdot\nabla B\|^{1/2}_{L^{2}} \|B\cdot\nabla B\|^{1/2}_{L^{6}} \, dt    \\
\leq & \int_{0}^{\sigma(T_{0})} t^{-5/8} (t\|B\cdot\nabla B\|^{2}_{L^{2}})^{1/4} \|\nabla B\|^{1/4}_{L^{2}}
(t\|\nabla^{2}B\|^{2}_{L^{2}})^{3/8} \, dt  \\
& \quad\quad\quad\quad\quad\quad\quad\quad\quad\quad\,\,\,
+ \int_{0}^{\sigma(T_{0})} t^{-1/2} (t\|B\cdot\nabla B\|^{2}_{L^{2}})^{1/4} t^{1/4} \|\nabla B\|_{L^{4}} \, dt  \\
\leq & C \epsilon_{0}^{3/8} \int_{0}^{\sigma(T_{0})} t^{-5/8} (t\|\nabla^{2}B\|^{2}_{L^{2}})^{3/8} \, dt
+ C \epsilon_{0}^{1/4} \int_{0}^{\sigma(T_{0})} t^{-1/2} t^{1/4} \|\nabla B\|_{L^{4}} \, dt   \\
\leq & C \epsilon_{0}^{3/8}\epsilon_{0}^{3/16} + C \epsilon_{0}^{1/4}\epsilon_{0}^{9/16} \leq C \epsilon_{0}^{9/16},
\end{split}
\end{align}
where we used estimate (\ref{refinedestimates}) and (\ref{estimateforAfinal}).
Substitute (\ref{rhoes1}), (\ref{rhoes2}), (\ref{rhoes3}) and (\ref{rhoes4}) into (\ref{estimaterho1}), we obtain
\begin{align}\label{rhoes5}
\int_{0}^{T_{0}} \|F\|_{L^{\infty}} \, dt \leq C \epsilon_{0}^{1/4},
\end{align}
which implies that for $t \leq \sigma(T_{0})$,
\begin{align}\label{rhoes6}
\inf (\log \rho_{0}(x)) - C \epsilon_{0}^{1/4} - C t \leq \log \rho(t,x) \leq \sup (\log (\rho_{0}(x))) + C \epsilon_{0}^{1/4} + C t.
\end{align}
So we can choose $\epsilon_{0}$, $\tau$ small enough such that for $t \leq \tau \leq \sigma(T_{0})$,
\begin{align}\label{rhoes7}
\frac{3}{4}c_{0} < \rho(t,x) < \frac{3}{2}c_{0}^{-1}.
\end{align}
For $\tau \leq t_{1} \leq t_{2} \leq T_{0}$, we have
\begin{align*}
& \int_{t_{1}}^{t_{2}} \|F\|_{L^{\infty}} \, dt \leq C \int_{t_{1}}^{t_{2}} (\|\sqrt{\rho}\dot{u}\|^{1/2}_{L^{2}} + \|B\cdot\nabla B\|^{1/2}_{L^{2}})
(\|\nabla\dot{u}\|^{1/2}_{L^{2}} + \|B\cdot\nabla B\|^{1/2}_{L^{6}}) \, dt  \\
& \quad\quad\quad
\leq C (t_{2} - t_{1})^{1/2} \int_{t_{1}}^{t_{2}} \|\sqrt{\rho}\dot{u}\|^{2}_{L^{2}}
+ \|B\cdot\nabla B\|^{2}_{L^{2}} + \|\nabla\dot{u}\|^{2}_{L^{2}} + \|B\cdot\nabla B\|^{2}_{L^{6}} \, dt.
\end{align*}
Now, we estimate the four integral terms on the right hand side of the above inequality.
For the first two terms, we have
\begin{align}\label{rhoes8}
\int_{t_{1}}^{t_{2}} \|\sqrt{\rho}\dot{u}\|^{2}_{L^{2}} \, dt \leq A_{2}(T_{0}) \leq \epsilon_{0}^{1/2}, \quad
\int_{t_{1}}^{t_{2}} \|\nabla\dot{u}\|^{2}_{L^{2}} \, dt \leq A_{2}(T_{0}) \leq \epsilon_{0}^{1/2},
\end{align}
The estimate about the third term appeared in (\ref{lemma4est12}), so we have
\begin{align}\label{rhoes9}
\int_{t_{1}}^{t_{2}} \|B\cdot\nabla B\|^{2}_{L^{2}} \, dt = \int_{t_{1}}^{t_{2}}\int_{\mathbb{R}^{3}} |B|^{2}|\nabla B|^{2} \, dx dt
\leq C \epsilon_{0}^{1/2}.
\end{align}
For the fourth term, we have
\begin{align}\label{rhoes10}
\begin{split}
\int_{t_{1}}^{t_{2}} \|B\cdot\nabla B\|^{2}_{L^{6}} \, dt & \leq \int_{t_{1}}^{t_{2}}\|\nabla B\|^{1/2}_{L^{2}} \|\nabla^{2}B\|^{3/2}_{L^{2}} \, dt
+ \int_{t_{1}}^{t_{2}} \|\nabla B\|^{4}_{L^{4}} \, dt   \\
& \leq \left( \int_{t_{1}}^{t_{2}} \|\nabla B\|^{2}_{L^{2}} \, dt \right)^{1/4}
\left( \int_{t_{1}}^{t_{2}} \|\nabla^{2} B\|^{2}_{L^{2}} \, dt \right)^{3/4} + C \epsilon_{0}^{1/2} \\
& \leq C \epsilon_{0}^{1/4} \epsilon_{0}^{3/8} + C \epsilon_{0}^{1/2} \leq C \epsilon_{0}^{1/2}.
\end{split}
\end{align}
Combining estimates (\ref{rhoes8}), (\ref{rhoes9}) and (\ref{rhoes10}), we obtain
\begin{align}\label{rhoes11}
\begin{split}
\int_{t_{1}}^{t_{2}} \|F\|_{L^{\infty}} \, dt & \leq C (t_{2} - t_{1})^{1/2} \epsilon_{0}^{1/2}     \\
& \leq \eta (t_{2} - t_{1}) + C_{\eta} \epsilon_{0},
\end{split}
\end{align}
where $\eta > 0$ is a small enough positive number.
With estimate (\ref{rhoes11}) at hand, we can mimic the Navier-Stokes case \cite{NS2013Zhifei} to obtain
\begin{align}\label{rhoes12}
\frac{3}{4}c_{0} < \rho(t,x) < \frac{3}{2}c_{0}^{-1},
\end{align}
where $0 \leq t \leq T_{0}$.
With estimate (\ref{estimateforAfinal}) and (\ref{rhoes12}), we can complete the proof by continuity argument.
\end{proof}


\section{Blow up Criterion and the Global Well-Posedness}

In this part, we give a blow up criterion and then prove the local solution can be extended to a global one.
Firstly, let me give the blow up criterion as follows.
\begin{theorem}\label{blowupcriterion}
For dimension $N = 3$, let $(\rho, u, B)$ be a solution of system (\ref{origianl MHD}) satisfying
\begin{align*}
& \rho(0) > 0, \quad \rho - \bar{\rho} \in C([0,T];H^{2}),    \\
& u, B \in C([0,T];H^{2}) \cap L^{2}(0,T;H^{3}).
\end{align*}
Let $T^{*}$ be the maximal existence time of the solution. If $T^{*} < +\infty$, then it is necessary that
\begin{align*}
\limsup_{t \uparrow T^{*}} \left( \|\rho(t)\|_{L^{\infty}} + \|u(t)\|_{L^{q}} + \|B(t)\|_{L^{q}} \right) = + \infty,
\end{align*}
for any $q \geq 6$.
\end{theorem}
\begin{proof}
We use the contradiction argument. Assume that $T^{*} < + \infty$ and
\begin{align}\label{blowup1}
\sup_{t \in [0,T^{*})} \left( \|\rho(t)\|_{L^{\infty}} + \|u(t)\|_{L^{q}} + \|B(t)\|_{L^{q}} \right) = M < + \infty.
\end{align}
In what follows, we denote $C$ to be a constant depending on $T$, $M$, $\|u_{0}\|_{H^{2}}$, $\|B_{0}\|_{H^{2}}$,
$\|\rho_{0} - \bar{\rho}\|_{H^{2}}$.
Firstly, from the energy estimates, we have
\begin{align}\label{blowup2}
\begin{split}
\int_{\mathbb{R}^{3}} |\rho - \bar{\rho}|^{2} + \rho |u|^{2} + |B|^{2} \, dx
+ \int_{0}^{T}\int_{\mathbb{R}^{3}} |\nabla u|^{2} + |\nabla B|^{2} \, dx dt \leq C.
\end{split}
\end{align}
Considering both (\ref{blowup1}) and (\ref{blowup2}), for $r \in [2, \infty]$, we have
\begin{align}\label{blowup3}
\begin{split}
\|\rho - \bar{\rho}\|_{L^{\infty}_{T}(L^{r})} + \|\sqrt{\rho}u\|_{L_{T}^{\infty}(L^{2})}
& + \|B\|_{L_{T}^{\infty}(L^{2})}     \\
& + \|\nabla u\|_{L_{T}^{2}(L^{2})} + \|\nabla B\|_{L_{T}^{2}(L^{2})} \leq C.
\end{split}
\end{align}
Let $v = L^{-1}\nabla P(\rho)$ to be a solution of the following elliptic system
\begin{align}\label{blowup4}
\begin{split}
L v := \mu \Delta v + \lambda \nabla \mathrm{div} v = \nabla P(\rho).
\end{split}
\end{align}
By elliptic estimate, for $r \in [2, \infty]$, we can obtain
\begin{align}\label{blowup5}
\|\nabla v\|_{L^{r}} \leq C \|P(\rho) - P(\bar{\rho})\|_{L^{r}} \leq C, \quad
\|\nabla^{2}v\|_{L^{r}} \leq C \|\nabla \rho\|_{L^{r}}.
\end{align}
Now, we introduce a new unknown $w = u - v$. We can easily know that
\begin{align*}
(\lambda + \mu)\Delta\mathrm{div}w = (\lambda + \mu)\Delta\mathrm{u} - \Delta (P(\rho) - P(\bar{\rho})),
\end{align*}
and $\mathrm{div}w$ is called effective viscous flux which was used by many authors\cite{NS2013Zhifei,mhd2012hoff,hoff1995,hoff1997}.
Through simple calculations, we have
\begin{align}\label{blowup6}
\rho \partial_{t}w - \mu \Delta w - \lambda \nabla \mathrm{div} w = \rho F + B\cdot\nabla B - \frac{1}{2} \nabla (|B|^{2}),
\end{align}
where
\begin{align}\label{blowup7}
F = - u\cdot\nabla u + L^{-1}\nabla\mathrm{div}(P(\rho)u) + L^{-1}\nabla ((\rho P'(\rho) - P(\rho))\mathrm{div}u).
\end{align}
Multiply (\ref{blowup6}) with $\partial_{t}w$ and integrating by parts, we have
\begin{align}\label{blowup8}
\begin{split}
\partial_{t} \int_{\mathbb{R}^{3}} \mu |\nabla w|^{2} + \lambda |\mathrm{div}w|^{2} \, dx & + \frac{1}{2} \int_{\mathbb{R}^{3}} \rho |\partial_{t}w|^{2} \, dx \\
& \leq C \|\sqrt{\rho}F\|_{L^{2}}^{2} + C \|B\cdot\nabla B\|_{L^{2}}^{2}.
\end{split}
\end{align}
For the first term on the right hand side of the above inequality, we can estimate as in the Navier-Stokes equations \cite{blowup12011,blowup22011} to obtain
\begin{align}\label{blowup9}
\|\sqrt{\rho}F\|_{L^{2}} \leq C (1+\|\nabla u\|_{L^{2}}) + \epsilon \|\sqrt{\rho}\partial_{t}w\|_{L^{2}},
\end{align}
where $\epsilon$ is a small enough positive number.
Next, we need to estimate $\|B\cdot\nabla B\|_{L^{2}}$, $\|u\cdot\nabla B\|_{L^{2}}$, $\|B\cdot\nabla u\|_{L^{2}}$ and $\|u\cdot\nabla u\|_{L^{2}}$
to make our later estimate more clear.
For the term $\|B\cdot\nabla B\|_{L^{2}}$, we have
\begin{align*}
\|B\cdot\nabla B\|_{L^{2}} & \leq C \|B\|_{L^{q}} \|\nabla B\|_{L^{2q/(q-2)}}   \\
& \leq C \|\nabla B\|_{L^{2}} + \epsilon \|\nabla^{2}B\|_{L^{2}}    \\
& \leq C \|\nabla B\|_{L^{2}} + \epsilon \|\partial_{t}B\|_{L^{2}} + \epsilon \|u\cdot\nabla B\|_{L^{2}} + \epsilon \|B\cdot\nabla u\|_{L^{2}}.
\end{align*}
Similarly, we have
\begin{align*}
\|u\cdot\nabla B\|_{L^{2}} \leq C \|\nabla B\|_{L^{2}} + \epsilon \|\partial_{t}B\|_{L^{2}} + \epsilon \|u\cdot\nabla B\|_{L^{2}}
+ \epsilon \|B\cdot\nabla u\|_{L^{2}}.
\end{align*}
As in the Navier-Stokes equations \cite{blowup12011,blowup22011}, we can easily obtain
\begin{align*}
& \|B\cdot\nabla u\|_{L^{2}} \leq C + C \|\nabla u\|_{L^{2}} + \epsilon \|\sqrt{\rho}\partial_{t}w\|_{L^{2}}    \\
& \|u\cdot\nabla u\|_{L^{2}} \leq C + C \|\nabla u\|_{L^{2}} + \epsilon \|\sqrt{\rho}\partial_{t}w\|_{L^{2}}.
\end{align*}
So summing up the above four estimates, we obtain
\begin{align}\label{blowup10}
\begin{split}
\|B\cdot\nabla B\|_{L^{2}} & + \|B\cdot\nabla u\|_{L^{2}} + \|u\cdot\nabla B\|_{L^{2}} + \|u\cdot\nabla u\|_{L^{2}} \\
& \leq C + C \|\nabla u\|_{L^{2}} + C \|\nabla B\|_{L^{2}} + \epsilon \|\partial_{t}B\|_{L^{2}} + \epsilon \|\sqrt{\rho}\partial_{t}w\|_{L^{2}}.
\end{split}
\end{align}
Substitute (\ref{blowup9}) and (\ref{blowup10}) into (\ref{blowup8}), we will have
\begin{align}\label{blowup11}
\begin{split}
\partial_{t}\int_{\mathbb{R}^{3}} \mu |\nabla w|^{2} & + \lambda |\mathrm{div}w|^{2} \, dx + \int_{\mathbb{R}^{3}} \rho |\partial_{t}w|^{2} \, dx     \\
& \leq C (1 + \|\nabla u\|_{L^{2}}^{2} + \|\nabla B\|_{L^{2}}^{2}) + \epsilon \|\sqrt{\rho}\partial_{t}w\|_{L^{2}}^{2}
+ \epsilon \|\partial_{t}B\|_{L^{2}}^{2}.
\end{split}
\end{align}
Multiply $\partial_{t}B$ to the third equation of system (\ref{origianl MHD}) and integrate by parts, we will obtain
\begin{align}\label{blowup12}
\begin{split}
\partial_{t}\int_{\mathbb{R}^{3}} |\nabla B|^{2} \, dx & + \int_{\mathbb{R}^{3}} |\partial_{t}B|^{2} \, dx    \\
& \leq C (1 + \|\nabla u\|_{L^{2}}^{2} + \|\nabla B\|_{L^{2}}^{2}) + \epsilon \|\sqrt{\rho}\partial_{t}w\|_{L^{2}}^{2}
+ \epsilon \|\partial_{t}B\|_{L^{2}}^{2},
\end{split}
\end{align}
where we used estimate (\ref{blowup10}).
If $\epsilon$ is small enough, summing up (\ref{blowup11}) and (\ref{blowup12}), we will have
\begin{align*}
\partial_{t}\int_{\mathbb{R}^{3}} \mu |\nabla w|^{2} + \lambda |\mathrm{div}w|^{2} & + |\nabla B|^{2} \, dx
+ \int_{\mathbb{R}^{3}} \rho |\partial_{t}w|^{2} \, dx  \\
& + \int_{\mathbb{R}^{3}} |\partial_{t}B|^{2} \, dx \leq C (1 + \|\nabla u\|_{L^{2}}^{2} + \|\nabla B\|_{L^{2}}^{2}).
\end{align*}
Integrate the above inequality for time variable, we will obtain
\begin{align}\label{blowup13}
\begin{split}
\|\nabla w\|_{L_{T}^{\infty}(L^{2})} + \|\nabla B\|_{L_{T}^{\infty}(L^{2})} & + \|\sqrt{\rho}\partial_{t}w\|_{L_{T}^{2}(L^{2})} \\
& + \|\partial_{t}B\|_{L_{T}^{2}(L^{2})} + \|\nabla^{2}w\|_{L_{T}^{2}(L^{2})} \leq C.
\end{split}
\end{align}
The above inequality combined with (\ref{blowup5}) implies that
\begin{align}\label{blowup14}
\begin{split}
\|\nabla u\|_{L_{T}^{\infty}(L^{2})} + \|\nabla u\|_{L_{T}^{2}(L^{r})} \leq C \quad \text{for} \quad r \in [2,6].
\end{split}
\end{align}
For the term $\|\nabla^{2}B\|_{L_{T}^{2}(L^{2})}$, we will have
\begin{align}\label{blowup15}
\begin{split}
\|\nabla^{2}B\|_{L_{T}^{2}(L^{2})} & \leq C \|\partial_{t}B\|_{L_{T}^{2}(L^{2})} + \|u\cdot\nabla B\|_{L_{T}^{2}(L^{2})}
+ \|B\cdot\nabla u\|_{L_{T}^{2}(L^{2})} \\
& \leq C < + \infty,
\end{split}
\end{align}
where we used (\ref{blowup10}) and (\ref{blowup13}).
Hence, we have $\|\nabla B\|_{L_{T}^{2}(L^{2})} \leq C < + \infty$ by interpolation.
Now let us turn to the high order energy estimate. From the $H^{2}$ energy estimate (Lemma \ref{hofflemma3}), we have
\begin{align}\label{blowup16}
\begin{split}
& \int_{\mathbb{R}^{3}} \sigma(t) (\rho |\dot{u}|^{2} + |B_{t}|^{2}) \, dx
+ \int_{0}^{T}\int_{\mathbb{R}^{3}} \sigma(t) (|\nabla\dot{u}|^{2} + |\nabla B_{t}|^{2}) \, dx dt   \\
&
\leq C C_{0} + C \int_{0}^{T}\int_{\mathbb{R}^{3}} \sigma(t) |\nabla u|^{4} \, dx dt + C \int_{0}^{T} \int_{\mathbb{R}^{3}} |\nabla u|^{3} \, dx dt \\
& \quad
+ C \int_{0}^{T}\int_{\mathbb{R}^{3}} (|\nabla B|^{2}|B|^{2} + |\nabla B|^{2}|u|^{2} + |\nabla u|^{2}|B|^{2}) \, dx dt  \\
& \quad
+ C \int_{0}^{T}\int_{\mathbb{R}^{3}}\sigma(t)(|u|^{2} + |B|^{2})(|B_{t}|^{2} + |\dot{u}|^{2}
+ |u|^{2}(|\nabla B|^{2} + |\nabla u|^{2})) \, dx dt.
\end{split}
\end{align}
Noting that
\begin{align}\label{blowup17}
\begin{split}
\mu \Delta w + \lambda \nabla \mathrm{div} w = \rho \dot{u} - B\cdot\nabla B + \frac{1}{2}\nabla(|B|^{2}),
\end{split}
\end{align}
elliptic estimate yields that
\begin{align}\label{blowup18}
\begin{split}
\|\nabla^{2}w\|_{L^{2}} \leq C \|\rho\dot{u}\|_{L^{2}} + C \|B\cdot\nabla B\|_{L^{2}}.
\end{split}
\end{align}
From (\ref{blowup10}) and (\ref{blowup13}), we obtain
\begin{align}\label{blowup19}
\begin{split}
\int_{0}^{T}\int_{\mathbb{R}^{3}} |\nabla B|^{2} |B|^{2} + |\nabla B|^{2} |u|^{2} + |\nabla u|^{2} |B|^{2} \, dx dt \leq C < +\infty.
\end{split}
\end{align}
Since
\begin{align}\label{blowup20}
\begin{split}
\int_{0}^{T} & \int_{\mathbb{R}^{3}}\sigma(t) |B|^{2} |u|^{2}|\nabla B|^{2} \, dx dt     \\
& \leq \int_{0}^{T}\left( \int_{\mathbb{R}^{3}} |B|^{3}|u|^{3} \, dx \right)^{2/3} \left( \int_{\mathbb{R}^{3}} |\nabla B|^{6} \, dx \right)^{1/3} dt   \\
& \leq \sup_{0 \leq t \leq T}\left( \int_{\mathbb{R}^{3}} |B|^{3} |u|^{3} \, dx \right)^{2/3} \int_{0}^{T} \|\nabla B\|_{L^{6}}^{2} \, dt    \\
& \leq C \|B\|_{L_{T}^{\infty}(L^{6})}^{2} \|u\|_{L_{T}^{\infty}(L^{6})}^{2} \|\nabla B\|_{L_{T}^{2}(L^{6})}^{1/2} \leq C < + \infty,
\end{split}
\end{align}
we can use similar methods to obtain
\begin{align}\label{blowup21}
\begin{split}
\int_{0}^{T}\int_{\mathbb{R}^{3}} \sigma(t) (|u|^{2} + |B|^{2}) |u|^{2} (|\nabla B|^{2} + |\nabla u|^{2}) \, dx dt \leq C < +\infty.
\end{split}
\end{align}
Similar to the Navier-Stokes system \cite{NS2013Zhifei}, we can obtain
\begin{align}\label{blowup22}
\begin{split}
\|\nabla u\|_{L^{4}}^{4} \leq C \|\nabla u\|_{L^{6}}^{2} (1 + \|\sqrt{\rho}\dot{u}\|_{L^{2}}), \quad
\|\nabla u\|_{L^{3}}^{3} \leq C \|\nabla u\|_{L^{6}}^{3/2}.
\end{split}
\end{align}
Plugging (\ref{blowup19}), (\ref{blowup21}) and (\ref{blowup22}) into (\ref{blowup16}) and noting $\|\nabla u(t)\|_{L^{6}} \in L^{2}(0,T)$
by (\ref{blowup14}) and (\ref{blowup15}), we deduce by Gronwall's inequality that
\begin{align}\label{blowup23}
\begin{split}
\int_{\mathbb{R}^{3}} \sigma(t)(\rho |\dot{u}|^{2} + |B_{t}|^{2}) \, dx
+ \int_{0}^{T}\int_{\mathbb{R}^{3}} \sigma(t)(|\nabla\dot{u}|^{2} + |\nabla B_{t}|^{2}) \, dx dt \leq C.
\end{split}
\end{align}
From the above inequality, elliptic estimate (\ref{blowup18}) and Sobolev inequality, we have
\begin{align}\label{blowup24}
\begin{split}
\|\nabla^{2}w\|_{L_{T}^{2}(L^{r})} \leq C, \quad \text{for} \quad r \in [2,6].
\end{split}
\end{align}
Using same methods as for Navier-Stokes system \cite{NS2013Zhifei}, we have
\begin{align}\label{blowup25}
\begin{split}
\|\nabla^{2}u\|_{L_{T}^{2}(L^{r})} \leq C, \quad \|\nabla \rho\|_{L^{r}} \leq C, \quad \text{for} \quad r \in [2,6].
\end{split}
\end{align}
Since
\begin{align*}
\|\nabla^{2}B\|_{L^{2}} & \leq C ( \|B_{t}\|_{L^{2}} + \|u\|_{L^{\infty}} \|\nabla B\|_{L^{2}} + \|B\|_{L^{\infty}} \|\nabla u\|_{L^{2}} )  \\
& \leq C + \frac{1}{2} \|\nabla^{2}u\|_{L^{2}} + \frac{1}{4} \|\nabla^{2}B\|_{L^{2}},
\end{align*}
and
\begin{align*}
\|\nabla^{2}u\|_{L^{2}} & \leq C \|\nabla^{2}w\|_{L^{2}} + C \|\nabla^{2}v\|_{L^{2}}    \\
& \leq C ( \|\rho\dot{u}\|_{L^{2}} + \|B\|_{L^{\infty}} \|\nabla B\|_{L^{2}} + \|\nabla \rho\|_{L^{2}} )    \\
& \leq C + \frac{1}{4} \|\nabla^{2}B\|_{L^{2}},
\end{align*}
we have
\begin{align}\label{blowup26}
\|\nabla^{2}u\|_{L^{2}} + \|\nabla^{2}B\|_{L^{2}} \leq C.
\end{align}
By the above estimates, we easily know $\|u\|_{L^{\infty}} + \|B\|_{L^{\infty}} \leq C$ by interpolation,
so we will obtain
\begin{align}\label{blowup27}
\begin{split}
\|\nabla^{3}B\|_{L_{T}^{2}(L^{2})} \leq & C ( \|\nabla B_{t}\|_{L_{T}^{2}(L^{2})} + \|u\|_{L^{\infty}}\|\nabla^{2}B\|_{L_{T}^{2}(L^{2})}  \\
& + \|B\|_{L^{\infty}} \|\nabla^{2}u\|_{L_{T}^{2}(L^{2})} + \|\nabla u\, \nabla B\|_{L_{T}^{2}(L^{2})} ) \leq C,
\end{split}
\end{align}
where we also used
\begin{align*}
\|\nabla u\,\nabla B\|_{L_{T}^{2}(L^{2})} & \leq \left( \int_{0}^{T} \left( \int_{\mathbb{R}^{3}} |\nabla u|^{4} \, dx \right)^{1/2}
\left( \int_{\mathbb{R}^{3}} |\nabla B|^{4} \, dx \right)^{1/2} dt \right)^{1/2}    \\
& \leq C \|\nabla u\|_{L_{T}^{\infty}(L^{4})} \|\nabla u\|_{L_{T}^{2}(L^{4})} + C \|\nabla B\|_{L_{T}^{\infty}(L^{4})} \|\nabla B\|_{L_{T}^{2}(L^{4})}  \\
& \leq C < + \infty.
\end{align*}
Then, similar to Navier-Stokes system, we can get
\begin{align}\label{blowup28}
\int_{0}^{T} \|\nabla^{3}w\|_{L^{2}}^{2} \, dt \leq C.
\end{align}
From (\ref{blowup5}), (\ref{blowup26}), (\ref{blowup27}) and (\ref{blowup28}), we will obtain
\begin{align}\label{blowup29}
\begin{split}
\|\nabla^{2}u(t)\|_{L^{2}}^{2} + \|\nabla^{2}B(t)\|_{L^{2}}^{2} + \int_{0}^{T} \|\nabla^{3}u\|_{L^{2}}^{2} & + \|\nabla^{3}B\|_{L^{2}}^{2} \, dt  \\
& \leq C + C \int_{0}^{T} \|\nabla^{2}\rho(t)\|_{L^{2}}^{2} \, dt.
\end{split}
\end{align}
Form the continuity equation, we have
\begin{align}\label{blowup30}
\begin{split}
\frac{d}{dt} \|\nabla^{2}(\rho(t) - \bar{\rho})\|_{L^{2}}^{2} \leq & C (1+ \|\nabla u(t)\|_{L^{\infty}}) \|\nabla^{2}(\rho(t) - \bar{\rho})\|_{L^{2}}^{2} \\
& + C \|\nabla^{3}u(t)\|_{L^{2}}^{2}.
\end{split}
\end{align}
Summing up (\ref{blowup29}) and (\ref{blowup30}), we conclude by Gronwall's inequality that for $0 \leq t < T^{*}$,
\begin{align*}
\|\rho(t) - \bar{\rho}\|_{H^{2}} + \|u(t)\|_{H^{2}} + \|B(t)\|_{H^{2}} \leq C.
\end{align*}
This ensures that the solution can be continued after $t = T^{*}$.
\end{proof}

After we get the blow up criterion, we can give the proof about Theorem \ref{mainexistence}.
From Theorem \ref{local well-posedness}, Proposition \ref{lowp} and Proposition \ref{highp}, we obtain a solution $(\rho, u, B)$ of (\ref{origianl MHD})
satisfying
\begin{align*}
& \frac{c_{0}}{2} \leq \rho \leq 2 c_{0}^{-1}, \quad \rho - \bar{\rho} \in \tilde{L}^{\infty}_{T}(\dot{B}_{p,1}^{3/p} \cap \dot{B}_{2,2}^{s}),    \\
& u, B \in \tilde{L}^{\infty}_{T}(\dot{B}_{p,1}^{3/p}\cap \dot{B}_{2,2}^{s-1}) \cap \tilde{L}^{1}_{T}(\dot{B}_{p,1}^{3/p+1}\cap \dot{B}_{2,2}^{s+1}).
\end{align*}
Moreover, it holds that
\begin{align}\label{global1}
\|\rho - \bar{\rho}\|_{\tilde{L}_{T}^{\infty}(\dot{B}_{2,2}^{1-\delta})} \leq C
(\|\rho_{0} - \bar{\rho}\|_{H^{1-\delta}} + \|u_{0}\|_{\dot{H}^{-\delta}} + \|B_{0}\|_{\dot{H}^{-\delta}} ),
\end{align}
\begin{align}\label{global2}
\begin{split}
& \|(u, B)\|_{\tilde{L}_{T}^{\infty}(\dot{B}_{2,2}^{-\delta})} + \|(u, B)\|_{\tilde{L}_{T}^{1}(\dot{B}_{2,2}^{2-\delta})}   \\
& \quad\quad
\leq C (1 + \|a_{0}\|_{\dot{B}_{p,1}^{3/p}} )( \|u_{0}\|_{\dot{H}^{-\delta}} + \|B_{0}\|_{\dot{H}^{-\delta}}
+ T P_{+} \|\rho_{0} - \bar{\rho}\|_{H^{1-\delta}} ),
\end{split}
\end{align}
\begin{align}\label{global3}
\begin{split}
\|a\|_{\tilde{L}_{T}^{\infty}(\dot{B}_{p,1}^{3/p})} \leq 2 \|a_{0}\|_{\dot{B}_{p,1}^{3/p}} + \frac{c_{2} + c_{3}}{(1 + \|a_{0}\|_{\dot{B}_{p,1}^{3/p}})^{5}},
\end{split}
\end{align}
\begin{align}\label{global4}
\begin{split}
\|u\|_{\tilde{L}_{T}^{\infty}(\dot{B}_{p,1}^{3/p-1})} + \|u\|_{\tilde{L}_{T}^{1}(\dot{B}_{p,1}^{3/p+1})}
\leq \frac{c_{2} + c_{3}}{(1 + \|a_{0}\|_{\dot{B}_{p,1}^{3/p}})^{5}},
\end{split}
\end{align}
\begin{align}\label{global5}
\begin{split}
\|B\|_{\tilde{L}_{T}^{\infty}(\dot{B}_{p,1}^{3/p-1})} + \|B\|_{\tilde{L}_{T}^{1}(\dot{B}_{p,1}^{3/p+1})}
\leq \frac{2 c_{3}}{(1 + \|a_{0}\|_{\dot{B}_{p,1}^{3/p}})^{2}}.
\end{split}
\end{align}
Here and in what follows, the constant $C$ depends only on $\lambda, \mu, c_{0}, \bar{\rho}, p, s$.
Recall that the initial data satisfies
\begin{align*}
& \|\rho_{0} - \bar{\rho}\|_{L^{2}} \leq c_{1}, \quad \|u_{0}\|_{\dot{B}_{p,1}^{3/p}} \leq \frac{c_{2}}{(1 + \|a_{0}\|_{\dot{B}_{p,1}^{3/p}} )^{5}},
\quad \|B_{0}\|_{\dot{B}_{p,1}^{3/p}} \leq \frac{c_{2}}{(1 + \|a_{0}\|_{\dot{B}_{p,1}^{3/p}} )^{2}},  \\
& \|u_{0}\|_{\dot{H}^{-\delta}} \leq \frac{c_{2}}{( 1 + \|a_{0}\|_{\dot{B}_{p,1}^{3/p}} )^{5}( 1 + \|\rho_{0}-\bar{\rho}\|_{H^{2}}^{8} + \|u_{0}\|_{L^{2}}^{4/3}
+ \|B_{0}\|_{L^{2}}^{4/3} )} = \tilde{c}_{2},   \\
& \|B_{0}\|_{\dot{H}^{-\delta}} \leq \frac{c_{3}}{( 1 + \|a_{0}\|_{\dot{B}_{p,1}^{3/p}} )^{7/3}( 1 + \|\rho_{0}-\bar{\rho}\|_{H^{2}}^{8} + \|u_{0}\|_{L^{2}}^{4/3}
+ \|B_{0}\|_{L^{2}}^{4/3} )} = \tilde{c}_{3}.
\end{align*}
Taking $c_{2}$ small enough, by Remark \ref{lower bound of time} and (\ref{the form T}), the existence time $T$ has a lower bound
\begin{align}\label{global6}
T \geq \frac{c}{(1 + \|a_{0}\|_{\dot{B}_{p,1}^{3/p}})^{4}(1 + \|a_{0}\|_{H^{2}})^{12}},
\end{align}
where $c$ is a small positive number.

Now, for any $T_{1} \leq T$, we have
\begin{align}\label{global7}
\begin{split}
& \|(u, B)\|_{\tilde{L}_{T_{1}}^{\infty}(\dot{B}_{2,2}^{-\delta})} + \|(u, B)\|_{\tilde{L}_{T_{1}}^{1}(\dot{B}_{2,2}^{2-\delta})}     \\
& \quad\quad
\leq C (1 + \|a_{0}\|_{\dot{B}_{p,1}^{3/p}}) (\tilde{c}_{2} + \tilde{c}_{3} + P_{+}T_{1}\|\rho_{0} - \bar{\rho}\|_{H^{1-\delta}}),
\end{split}
\end{align}
by the conditions on the solution just mentioned.
For $r \in (1, 2-\delta)$, we have
\begin{align}\label{global8}
\begin{split}
& \|\nabla u\|_{L_{T_{1}}^{r}(L^{2})} \leq \sum_{j = - \infty}^{+ \infty} 2^{j} \|\Delta_{j}u\|_{L_{T_{1}}^{2}(L^{2})}    \\
& \quad\quad
\leq \sum_{j \leq 0} 2^{j}T_{1}^{1/r}\|\Delta_{j}u\|_{L_{T}^{\infty}(L^{2})} + \sum_{j > 0}2^{j}\|\Delta_{j}u\|_{L_{T_{1}}^{\infty}(L^{2})}^{1-1/r}
\|\Delta_{j}u\|_{L_{T_{1}}^{1}(L^{2})}^{1/r}    \\
& \quad\quad
\leq C (\|u\|_{\tilde{L}_{T_{1}}}^{\infty}(\dot{B}_{2,2}^{-\delta}) + \|u\|_{\tilde{L}_{T_{1}}^{1}(\dot{B}_{2,2}^{2-\delta})} ) \\
& \quad\quad
\leq C (1 + \|a_{0}\|_{\dot{B}_{p,1}^{3/p}}) (\tilde{c}_{4} + \tilde{c}_{5} + P_{+}T_{1}\|\rho_{0} - \bar{\rho}\|_{H^{2}} ).
\end{split}
\end{align}
Similarly, we have
\begin{align}\label{global9}
\begin{split}
\|u\|_{L_{T_{1}}^{r}(L^{2})}
\leq C (1 + \|a_{0}\|_{\dot{B}_{p,1}^{3/p}}) (\tilde{c}_{4} + \tilde{c}_{5} + P_{+}T_{1}\|\rho_{0} - \bar{\rho}\|_{H^{2}} ),
\end{split}
\end{align}
\begin{align}\label{global10}
\begin{split}
\|\nabla B\|_{L_{T_{1}}^{r}(L^{2})}
\leq C (1 + \|a_{0}\|_{\dot{B}_{p,1}^{3/p}}) (\tilde{c}_{4} + \tilde{c}_{5} + P_{+}T_{1}\|\rho_{0} - \bar{\rho}\|_{H^{2}} ),
\end{split}
\end{align}
and
\begin{align}\label{global11}
\begin{split}
\|B\|_{L_{T_{1}}^{r}(L^{2})}
\leq C (1 + \|a_{0}\|_{\dot{B}_{p,1}^{3/p}}) (\tilde{c}_{4} + \tilde{c}_{5} + P_{+}T_{1}\|\rho_{0} - \bar{\rho}\|_{H^{2}} ).
\end{split}
\end{align}
Then taking $r = \frac{3}{2}$, we have
\begin{align}\label{global12}
\begin{split}
\|(u, B)\|_{L_{T_{1}}^{3/2}(H^{1})} \leq C (1 + \|a_{0}\|_{\dot{B}_{p,1}^{3/p}}) (\tilde{c}_{4} + \tilde{c}_{5} + P_{+}T_{1} \|\rho_{0} - \bar{\rho}\|_{H^{2}}).
\end{split}
\end{align}
Hence, there exist $t_{0} \in (0, T_{1})$ such that
\begin{align}\label{global13}
\begin{split}
\|(u, B)(t_{0})\|_{H^{1}} \leq C (1 + \|a_{0}\|_{\dot{B}_{p,1}^{3/p}}) \left( \frac{\tilde{c}_{4} + \tilde{c}_{5}}{T_{1}^{2/3}}
+ P_{+}T_{1}^{1/3}\|\rho_{0}-\bar{\rho}\|_{H^{2}} \right).
\end{split}
\end{align}
For density, using simply energy estimates, we will get
\begin{align}\label{global14}
\|\rho(t) - \bar{\rho}\|_{L^{2}} \leq C (1+ P_{+}) (c_{1} + T_{1}^{1/2} (\|u_{0}\|_{L^{2}} + \|B_{0}\|_{L^{2}})).
\end{align}
Take $T_{1}$ as
\begin{align*}
T_{1} = \frac{\epsilon_{0}^{3}}{(5C)^{3}(1+P_{+})^{3}(1+\|\rho_{0}-\bar{\rho}\|_{H^{2}}^{12}+\|(u_{0}, B_{0})\|_{L^{2}}^{2})(1+\|a_{0}\|_{\dot{B}_{p,1}^{3/p}})^{4}},
\end{align*}
such that $T_{1} \leq T$ and
\begin{align*}
& C (1+\|a_{0}\|_{\dot{B}_{p,1}^{3/p}})P_{+}T_{1}^{1/3}\|\rho_{0}-\bar{\rho}\|_{H^{2}} \leq \frac{\epsilon_{0}}{5},   \\
& C (1+P_{+})T_{1}^{1/2}(\|u_{0}\|_{L^{2}} + \|B_{0}\|_{L^{2}}) \leq \frac{\epsilon_{0}}{5}.
\end{align*}
Then, we choose $c_{1}$, $c_{2}$ and $c_{3}$ small enough so that
\begin{align*}
C (1 + P_{+})c_{1} \leq \frac{\epsilon_{0}}{5}, \quad \frac{C ( 1 + \|a_{0}\|_{\dot{B}_{p,1}^{3/p}} )\tilde{c}_{2}}{T_{1}^{2/3}} \leq \frac{\epsilon_{0}}{5},
\quad \frac{C ( 1 + \|a_{0}\|_{\dot{B}_{p,1}^{3/p}} )\tilde{c}_{3}}{T_{1}^{2/3}} \leq \frac{\epsilon_{0}}{5}.
\end{align*}
Through this choice of $c_{1}$, $c_{2}$ and $c_{3}$, by (\ref{global13}) and (\ref{global14}), we have
\begin{align*}
\|a(t_{0})\|_{L^{2}} + \|u(t_{0})\|_{H^{1}} + \|B(t_{0})\|_{H^{1}} \leq \epsilon_{0}.
\end{align*}
Theorem \ref{hofftheorem} implies that
\begin{align*}
\frac{c_{0}}{2} \leq \rho \leq 2 c_{0}^{-1}, \quad \|u(t)\|_{L_{T}^{\infty}(L^{6})} + \|u(t)\|_{L_{T}^{\infty}(L^{6})} \leq C.
\end{align*}
So the solution can be extended to a global one by Theorem \ref{blowupcriterion}.


\section{Appendix}

In this section, for the reader's convenience, we firstly give the Lagrangian transform for terms appearing in the MHD system as follows
\begin{align}\label{transform}
\begin{split}
& \widetilde{\nabla_{x}(|B|^{2})} = J^{-1} \mathrm{div}_{y}(\mathrm{adj}(DX)|\tilde{B}|^{2}),    \\
& \widetilde{B\cdot\nabla_{x}B} = \widetilde{\mathrm{div}_{x}(B\otimes B)} = J^{-1}\mathrm{div}_{y}(\mathrm{adj}(DX)^{T}\tilde{B}\otimes \tilde{B}),  \\
& \widetilde{\mathrm{div}_{x}u B} = J^{-1}\mathrm{div}_{y}(\mathrm{adj}(DX)\tilde{u})\tilde{B},  \\
& \widetilde{(\mathrm{div}_{x}u)B} + \widetilde{u\cdot\nabla_{x}B} = \widetilde{\nabla_{x}(uB)} = J^{-1}\mathrm{div}_{y}(\mathrm{adj}(DX)\tilde{u}\tilde{B}), \\
& \widetilde{B\cdot\nabla_{x}u} = \widetilde{\mathrm{div}_{x}(B\otimes u)} = J^{-1}\mathrm{div}_{y}(\mathrm{adj}(DX)^{T}\tilde{u}\otimes \tilde{B}).
\end{split}
\end{align}

Then, let us look at the following Lam\'{e} system with nonconstant coefficients
\begin{align}\label{varheat}
\partial_{t}v - 2 a \mathrm{div}(\mu \mathcal{D}(v)) - b \nabla (\lambda \mathrm{div}v) = f.
\end{align}
We assume that the following uniform ellipticity condition is satisfied:
\begin{align}\label{condition}
\alpha := \mathrm{min}\left( \inf_{(t,x)\in [0,T]\times \mathbb{R}^{N}}(a\mu)(t,x), \, \inf_{(t,x)\in [0,T]\times \mathbb{R}^{N}}(2a\mu + b\lambda)(t,x) \right) > 0
\end{align}
Concerning this equation, we have the following lemma.
\begin{lemma}\label{variable heat eq}\cite{LagrangianCompressible,jiamei}
Let $a, b, \lambda$ and $\mu$ be bounded functions satisfies (\ref{condition}). Assume that $a\nabla \mu$, $b\nabla \lambda$, $\mu\nabla a$
and $\lambda\nabla b$ are in $L^{\infty}(0,T; \dot{B}_{p}^{n/p})$ for some $1 < p < \infty$. There exist two constants $\eta$ and $\kappa$ such that
if for some $m \in \mathbb{Z}$ we have
\begin{align}\label{coninlinear}
\begin{split}
& \mathrm{min}\left( \inf_{(t,x) \in [0,T]\times \mathbb{R}^{N}}S_{m}(2a\mu +b\lambda)(t,x),
\,\inf_{(t,x) \in [0,T]\times \mathbb{R}^{N}}S_{m}(a\mu)(t,x) \right) \geq \frac{\alpha}{2},        \\
& \quad\quad\quad \|(I - S_{m})(\mu\nabla a, a\nabla \mu, \lambda \nabla b, b\nabla\lambda)\|_{L_{T}^{\infty}(\dot{B}_{p}^{n/p-1})} \leq \eta \alpha,
\end{split}
\end{align}
then the solutions to (\ref{varheat}) satisfy for all $t \in [0, T]$,
\begin{align*}
& \|v\|_{L_{t}^{\infty}(\dot{B}_{p}^{s})} + \alpha \|v\|_{L_{t}^{1}(\dot{B}_{p}^{s+2})}   \\
& \quad \leq C \left( \|v_{0}\|_{\dot{B}_{p}^{s}} + \|f\|_{L_{t}^{1}(\dot{B}_{p}^{s})} \right)\exp\left( \frac{C}{\alpha}
\int_{0}^{t} \|S_{m}(\mu\nabla a, a\nabla \mu, \lambda\nabla b, b\nabla \lambda)\|_{\dot{B}_{p}^{n/p}}^{2}\,d\tau
\right),
\end{align*}
whenever $- \min (n/p, n/p') < s \leq n/p-1$ is satisfied.
\end{lemma}

\begin{lemma}\label{heat eq}\cite{danchin local 2001}
Let $p\in [1,+\infty]$, $1\leq \rho_{2} \leq \rho_{1} \leq +\infty$, and let $v$ solve
\begin{align*}
\begin{cases}
\partial_{t} v - \nu \triangle v = f, \\
v|_{t=0} = v_{0}.
\end{cases}
\end{align*}
Denote $\rho_{1}' = (1+1/\rho_{1} -1/\rho_{2})^{-1}$. Then there exist two positive constants $c$ and $C$ depending only on $N$ and
such that
\begin{align*}
\|v\|_{\tilde{L}_{T}^{\rho_{1}}(\dot{B}_{p}^{s+2/\rho_{1}})} \leq &\, \bigg( \sum_{q\in \mathbb{Z}} 2^{qs} \norm{\Delta_{q}v_{0}}_{L^{p}}
\left( \frac{1-e^{-c \nu T 2^{2q} \rho_{1}}}{c\nu \rho_{1}} \right)^{1/\rho_{1}}    \\
&\, \sum_{q\in \mathbb{Z}} 2^{q(s-2+2/\rho_{2})} \norm{\Delta_{q}f}_{L_{T}^{\rho_{2}}(L^{p})}\left( \frac{1-e^{-c\nu T2^{2q}\rho_{1}'}}
{c\nu \rho_{1}'} \right)^{1/\rho_{1}'}
 \bigg).
\end{align*}
In particular, we have
\begin{align*}
\norm{v}_{\tilde{L}_{T}^{\rho_{1}}(\dot{B}_{p}^{s+2/\rho_{1}})} \leq \frac{C}{\nu^{1/\rho_{1}}} \norm{v_{0}}_{\dot{B}_{p}^{s}}
+ \frac{C}{\nu^{1/\rho_{1}'}} \norm{f}_{\tilde{L}_{T}^{\rho_{2}}(\dot{B}_{p}^{s-2+2/\rho_{2}})}.
\end{align*}
Moreover $u$ belongs to $C([0,T]; \dot{B}_{p}^{s})$.
\end{lemma}

\section*{Acknowledgments}
J. Peng and J. Jia's research is support partially by National Natural Science Foundation of China under the grant no.11131006,
and by the National Basic Research Program of China under the grant no.2013CB329404.

\end{document}